\documentclass[11pt,notitlepage,reqno]{amsart}

\usepackage{listings}
\usepackage{graphicx}
\usepackage{amsmath, amsfonts, amssymb, amsthm}
\usepackage[myheadings]{fullpage}
\usepackage[utf8]{inputenc}
\usepackage[margin=1in]{geometry}
\usepackage{enumitem}
\usepackage{xcolor}
\usepackage{caption}
\usepackage{tabularx}
\usepackage{fancyhdr}
\usepackage{hyperref}
\usepackage{tikz}
\usetikzlibrary{angles,quotes}
\usepackage{mdframed}
\usepackage{lipsum}
\usepackage{chngcntr}
\usepackage{imakeidx}
\usepackage{tkz-euclide}
\usepackage{blindtext}
\usepackage[inline]{asymptote}
\usepackage{mathdots}
\usepackage{tikz-cd,mathtools}
\usepackage{ mathrsfs }
\usepackage{float}
\restylefloat{table}

\usepackage{multicol}
\usepackage{endnotes}
\usepackage{idxlayout}
\usepackage{xcolor, forest}
\usepackage{tocvsec2}
\usepackage[toc]{glossaries}
\usepackage{venndiagram}
\usepackage[colorinlistoftodos]{todonotes} 
\usepackage[normalem]{ulem}
\usetikzlibrary{shapes.geometric}
\usepackage{mathtools}

\usepackage{bbm}
\usepackage{mathtools}

\newcommand{\sop}{s.o.p.}
\newcommand{\no}{\noindent}

\newcommand{\abs}[1]{{\left| #1\right|}}
\renewcommand{\exp}[1]{{\textnormal{exp} {\left( #1 \right)}}}

\newcommand{\C}{\mathbb{C}}
\newcommand{\N}{\mathbb{N}}

\newcommand{\E}{\mathbb{E}}

\newcommand{\<}{\langle}
\renewcommand{\>}{\rangle}
\newcommand{\R}{\mathbb{R}}

\newcommand{\Z}{\mathbb{Z}}

\newcommand{\la}{\lambda}

\newcommand{\suppgen}{\text{supp}(\hphi) \subseteq \left[ -\frac{1}{n-a},\frac{1}{n-a}\right]}

\renewcommand{\l}{\ell}
\renewcommand{\t}{{a-1}}

\usepackage{mathtools}
\newcommand{\defeq}{\coloneqq}

\renewcommand{\i}{{\mathrm{i}}} % sqrt -1

\newtheorem{thm}{Theorem}[section]

\newtheorem{ex}[thm]{Example}
\newtheorem{lem}[thm]{Lemma}

\newtheorem{prop}[thm]{Proposition}

\newtheorem{propt}[thm]{Property}

\newtheorem{defn}[thm]{Definition}

\theoremstyle{definition}
\newtheorem{rem}[thm]{Remark}

\newcommand{\xxi}[1]{\xi_{#1}(\phi)}
\newcommand{\bxi}[1]{\overline{\xi_{#1}}(\phi)}

\newcommand{\be}{\begin{equation}}
\newcommand{\ee}{\end{equation}}
\newcommand{\bea}{\begin{eqnarray}}
\newcommand{\eea}{\end{eqnarray}}

\newcommand{\psumn}[1]{\sum_{p_1 \nmid N,\ldots,p_{#1} \nmid N}}
\newcommand{\qsumn}[1]{\sum_{q_1 \nmid N,\ldots,q_{#1} \nmid N}}

\newcommand{\qsum}[1]{\sum_{\substack{q_1 \nmid N,\ldots,q_{#1} \nmid N\\ q_j \text{ distinct}}}}

\newcommand{\psumextra}[1]{\sum_{\substack{p_1 \nmid N,\ldots,p_{#1} \nmid N\\ p_i \text{ distinct} \\ p_i \ne q_j}}}

\newcommand{\psumsimp}[1]{\sum_{p_1 ,\ldots,p_{#1}}}
\newcommand{\qsumsimp}[1]{\sum_{q_1 ,\ldots,q_{#1}}}

\newcommand{\qsumdist}[1]{\sum_{\substack{q_1 ,\ldots,q_{#1} \\ q_j \text{ distinct}}}}
\newcommand{\psumdist}[1]{\sum_{\substack{p_1 ,\ldots,p_{#1} \\ p_i \text{ distinct}}}}

\newcommand{\rootn}{N^{-1/2}}

\newcommand{\indicator}{\chi}%\mathbbm{1}}

\newcommand{\testfn}[1]{\hat{\phi}\left(\frac{\log #1}{\log R}\right)}
\newcommand{\testcomp}[1]{\left(\frac{2\log #1}{\sqrt{#1} \log R}\right)}
\newcommand{\testcompone}[1]{\frac{\log #1}{\sqrt{#1} \log R}}
\newcommand{\testcompvm}[1]{\left( \frac{\chi_0(#1) \Lambda(#1)}{\sqrt{#1}\log R}\right)}

\newcommand{\converrorterm}{O \left( \frac{(\log \log N)^2}{(\log N)^{1/2}} \right)}

\newcommand{\Enm}{E(\vec{n}, \vec{m})}
\newcommand{\Fabcde}{F(\vec{a}, \vec{b}, \vec{c}, \vec{d}, \vec{e})}

\newcommand{\msum}{\sum_{m \le N^\epsilon}}
\newcommand{\Jk}{J_{k-1}}
\newcommand{\bsum}[1]{\sum_{\substack{(b, #1) = 1 \\ b < N^{2022}}}}
\newcommand{\bsumdiv}[2]{\sum_{\substack{(b, #1) = 1 \\ b < N^{2022} \\ #2 \mid b}}}

\newcommand{\bsumfull}[2]{\sum_{\substack{(b, #1) = 1 \\ #2 \mid b}}}

\newcommand{\hs}[1]{\hspace{#1 cm}}

\newcommand{\inv}{^{-1}}

\newcommand{\nth}{n^{\rm th}}

\newcommand\nbea{\begin{eqnarray*}}
\newcommand\neea{\end{eqnarray*}}
\newcommand\bi{\begin{itemize}}
\newcommand\ei{\end{itemize}}
\newcommand\ben{\begin{enumerate}}
\newcommand\een{\end{enumerate}}

\newcommand{\ONe}{O\left(N^{-\epsilon}\right)}
\newcommand{\ONhalfe}{O\left(N^{1/2-\epsilon}\right)}

\newcommand{\soe}{{\rm SO(even)}}
\newcommand{\soo}{{\rm SO(odd)}}

\newcommand{\SO}{{\rm SO}}

\newcommand{\twocase}[5]{#1 \begin{cases} #2 & \text{{\rm #3}}\\ #4
&\text{{\rm #5}} \end{cases}   }

\newcommand{\foh}{\frac{1}{2}}

\newcommand{\gep}{\epsilon}  %lowercase epsilon
\newcommand{\g}{\gamma}      %lowercase gamma
\newcommand{\gl}{\lambda}    %lowercase lambda
\newcommand{\notdiv}{\nmid}

\numberwithin{equation}{section}

\DeclareMathOperator{\supp}{supp}

\newcommand{\bal}{\begin{align}}
\newcommand{\eal}{\end{align}}

\newcommand{\nn}{\nonumber\\}

\newcommand{\hkpmn}{H_k^\pm(N)}
\newcommand{\hkn}{H_k^\ast(N)}
\newcommand{\hksn}{H_k^\sigma(N)}

\newcommand{\jk}[2]{J_{k-1}\left( 4\pi \frac{ \sqrt{ #1 } }{ #2 }\right) }
\newcommand{\phir}[1]{\widehat{\phi}\left( \frac{ \log p_{#1} }{\log R}\right) }

\newcommand{\pfrac}[1]{\frac{2\log p_{#1}}{\sqrt{p_{#1}} \log R}}

\newcommand{\chib}[1]{\overline{\chi}(#1)}

\newcommand{\twoovern}{\left(-\frac{2}{n}, \frac{2}{n} \right)}
\newcommand{\oneoverna}{\left(-\frac{1}{n-a}, \frac{1}{n-a} \right)}

\newcommand{\twonsquare}{\left[ -\frac{2}{n}, \frac{2}{n} \right]}
\newcommand{\nasquare}{\left[-\frac{1}{n-a}, \frac{1}{n-a} \right]}

\newcommand{\zerona}{\left[0, \frac{1}{n-a} \right]}

\newcommand{\hphi}{\widehat{\phi}}  %phi^
\newcommand{\intinf}{\int_{-\infty}^\infty}

\newcommand{\intii}{\intinf}
\newcommand{\I}{\mathbbm{1}} % indicator function

\renewcommand{\d}{{\mathrm{d}}} % differential, used in integrals

\renewcommand{\Re}{{\mathfrak{Re}}}

\usepackage{listings}
\usepackage{color}

\definecolor{dkgreen}{rgb}{0,0.6,0}
\definecolor{gray}{rgb}{0.5,0.5,0.5}
\definecolor{mauve}{rgb}{0.58,0,0.82}

\usepackage{subfiles}
\usepackage{color}

\usepackage{soul}

\DeclareMathOperator{\re}{Re}

\begin{document}

\title{On the moments of one-level densities in families of holomorphic cusp forms in the level aspect}

\author{Peter Cohen}
\address{Massachusetts Institute of Technology}
\email{plcohen@mit.edu}

\author{Justine Dell}
\address{University of California, San Diego}
\email{jsdell@ucsd.edu}

\author{Oscar E. Gonz\'alez}
\address{Department of Mathematics, University of Illinois at Urbana-Champaign, Urbana, IL 61801}
\email{oscareg2@illinois.edu}

\author{Simran Khunger}
\address{University of Michigan, Ann Arbor}
\email{skhunger@umich.edu}

\author{Chung-Hang Kwan}
\address{University College London}
\email{ucahckw@ucl.ac.uk}

\author{Steven J. Miller}
\address{Williams College}
\email{sjm1@williams.edu}

\author{Alexander Shashkov}
\address{University of California, Berkeley}
\email{ashashkov@berkeley.edu}

\author{Alicia Smith Reina}
\address{Williams College}
\email{ags6@williams.edu}

\author{Carsten Sprunger}
\address{University of Michigan}
\email{csprun@umich.edu}

\author{Nicholas Triantafillou}
\address{Department of Mathematics\\
University of Georgia\\
Athens GA 30602}
\email{nicholas.triantafillou@uga.edu}

\author{Nhi Truong}
\address{Stanford University}
\email{ntruongv@stanford.edu}

\author{Roger Van Peski}
\address{Massachusetts Institute of Technology}
\email{rvp@mit.edu}

\author{Stephen Willis}
\address{Williams College}
\email{sdw2@williams.edu}

\thanks{This work was done during the 2011, 2012, 2015, 2017, and 2021 SMALL REU program in Number Theory and Probability at Williams College supported by NSF grants DMS1561945 and DMS1659037. We thank Roger Heath-Brown for his assistance on exponential sums. We also thank Geoffrey Iyer and Yingzi Yang for significant work on parts of earlier drafts of this work.}
\date{\today}

\begin{abstract}

  We study the $\nth$ centered moments of the $1$-level density for the low-lying zeros of  $L$-functions attached to holomorphic cuspidal newforms of large prime level and fixed weight. Assuming the Generalized Riemann Hypotheses, we compute this statistic for any $n\ge 1$ and for all test functions whose Fourier transforms are supported in $\left(-2/n, \, 2/n\right)$. This is believed to be the natural limit of the current technology. Our work significantly extends beyond the trivial range $(-1/n, \, 1/n)$ and  surpasses the previous record of  $(-1/(n-1),\,  1/(n-1))$ whenever $n>2$. The Katz-Sarnak philosophy predicts that the aforementioned statistic can be modeled by the corresponding statistic for the eigenvalues of random orthogonal matrices. We prove that this is the case for test functions with  Fourier support contained in $(-2/n,\, 2/n)$. The main technical innovation is a tractable vantage to evaluate the combinatorial zoo of terms, similar to the work of Conrey-Snaith \cite{CS14} and Mason-Snaith \cite{MS18}. As an application, our work provides better 
  bounds on the order of vanishing at the central point for the $L$-functions in our family.

\end{abstract}



\maketitle

\setcounter{tocdepth}{1}
\tableofcontents

\markboth{\scalebox{0.9}{Cohen, Dell, Gonz\'alez, Iyer, Khunger, Kwan, Miller, Shashkov, Smith Reina, Sprunger, Triantafillou, Truong, Van Peski, Willis, Yang}}{On the moments of one-level densities in families of holomorphic cusp forms in the level aspect}

%%%%%%%%%%%%%%%%%%%%%%%%%%%%%%%%%%%%%%%%%%%%%%%%%%%%%%%%%%%%%%%%%%%%%%%%%%%%%%%%%%%%%%%%%%%%%%%%%%%%%%%%%%%%%%%%%%%%%%%%%%%%%%%%%%%%%
%%%%%%%%%%%%%%%%%%%%%%%%%%%%%%%%%%%%%%%%%%%%%%%%%%%%%%%%%%%%%%%%%%%%%%%%%%%%%%%%%%%%%%%%%%%%%%%%%%%%%%%%%%%%%%%%%%%%%%%%%%%%%%%%%%%%%
%%%%%%%%%%%%%%%%%%%%%%%%%%%%%%%%%%%%%%%%%%%%%%%%%%%%%%%%%%%%%%%%%%%%%%%%%%%%%%%%%%%%%%%%%%%%%%%%%%%%%%%%%%%%%%%%%%%%%%%%%%%%%%%%%%%%%
%%%%%%%%%%%%%%%%%%%%%%%%%%%%%%%%%%%%%%%%%%%%%%%%%%%%%%%%%%%%%%%%%%%%%%%%%%%%%%%%%%%%%%%%%%%%%%%%%%%%%%%%%%%%%%%%%%%%%%%%%%%%%%%%%%%%%

\section{Introduction} \label{sec:introduction section}
%%%%%%%%%%%%%%%%%%%%%%%%%%%%%%%%%%%%%%%%%%%%%%%%%%%%%

\subsection{Historical Perspectives}

Since Montgomery and Dyson's discovery that the two-point correlation of the zeros of the Riemann zeta function agrees with the pair correlation function for eigenvalues of the Gaussian Unitary Ensemble (see \cite{Mon}), the connection between the zeros of $L$-functions and the zeros of random matrices has been a major area of study. It is now widely believed that the statistical behavior of families of $L$-functions can be modeled by ensembles of random matrices. Based on the observation that the spacing statistics of high zeros associated with cuspidal $L$-functions agree with the corresponding statistics for eigenvalues of random unitary matrices under Haar measure (see \cite{RS}, for example), it was originally believed that only the unitary ensemble was important to number theory. However, Katz-Sarnak \cite{KS1, KS2} showed that these statistics are the same for all classical compact groups.
These statistics, the $n$-level correlations, are unaffected by finite numbers of zeros. In particular, they failed to identify differences in behavior near the central point $s = 1/2$. 

The $n$-level density statistic was introduced to distinguish the behavior of families of $L$-functions close to $s=1/2$. Based partially on an analogy with the function field setting, Katz-Sarnak conjectured that the \emph{low-lying zeros} (i.e.,  zeros near $s=1/2$) of families of $L$-functions behave like the eigenvalues near $1$ of classical compact groups (unitary, symplectic, and orthogonal). The behavior of the eigenvalues near 1 is different for each matrix group. A growing body of evidence has shown that this conjecture holds for test functions with suitably restricted support for a wide range of families of $L$-functions. For a non-exhaustive list, see \cite{Alpoge2013, Alpoge2015, barrett2017one, devin2022low, DPR23, DM1, DM2, Entin2012, FiMi, Gao, Guloglu2005, HM, ILS, knightly2019weighted, Mil2, MP, OS1, OS2, Ricotta2007, Ro, Ru, Shin2012, waxman2021lower, yang2009, Yo}. Much of the previous work is focused on the $n=1$ case. We study the \emph{$\nth$ centered moment of the 1-level density} for any $n\ge 1$, a higher-order statistic first introduced by Hughes-Rudnick \cite{HR} which is combinatorially simpler than the $n$-level density.

In this article, we consider the family of $L$-functions associated with holomorphic cusp forms. We prove that the Katz-Sarnak conjecture holds for the $n^\text{th}$ centered moment of the 1-level density for this family and for all test functions $\phi$ with Fourier transform $\hphi$ supported in $\twoovern$, which represents the natural limit of the current analytic machinery based on the prior works of \cite{HR} on the unitary family of primitive Dirichlet characters, as well as \cite{Gao}, \cite{Entin2012} and \cite{LM13} on the $n$-level density for the symplectic family of quadratic Dirichlet $L$-functions. Proving the conjecture past this support likely requires new ideas or stronger hypotheses such as the \emph{`Hypothesis S'} discussed in  \cite[Sect. 10]{ILS} (or more likely its generalizations, see \cite{MMM24}). 

Previously, Iwaniec-Luo-Sarnak \cite{ILS} and  Hughes-Miller \cite{HM}   examined the same family and statistic addressed in this article (with the former limited to the case of $n=1$). Under the Generalized Rieman Hypothesis (GRH), \cite{ILS} prove the conjecture when $n=1$ for $\hphi$ supported in 
$(-2,2)$ when $n=1$, and \cite{HM} prove the conjecture in the restricted range $\left(-\frac{1}{n-1}, \frac{1}{n-1}\right)$ when $n\ge 2$. In \cite{HM}, the authors were unable to handle the new terms which emerge at larger supports on both the number theory and random matrix theory side (nonetheless, the work of \cite{HM} addressed some rather challenging combinatorics). Our work develops an approach to handle all the combinatorial terms that arise when the Fourier support of the test function extends up to $2/n$, thus filling in a non-trivial gap in our current knowledge of higher-order statistics for the zeros of $L$-functions in the orthogonal families.

To the best of the authors' knowledge, apart from this work, the success of a  \emph{purely combinatorial venture} (i.e., without passing to function field settings in the large $q$-limit and using the equidistribution theorems \`{a} la Deligne and Katz) in verifying Katz-Sarnak's philosophy through higher-order statistics is somewhat limited. \cite{CL20} calculate the $n$-level density for a unitary family of Dirichlet $L$-functions. For the symplectic family, Levinson-Miller \cite{LM13} calculated the $n$-level density of a family of Dirichlet $L$-functions for $n\le 7$, extending the work of Gao \cite{Gao}. See also the contributions of \cite{CS14} and \cite{MS18} in the realm of random matrix theory. The pioneering work of \cite{Entin2012} and \cite{En14} calculated the $n$-level density for the symplectic family studied in \cite{LM13} and \cite{Gao} and showed it agreed with random matrix theory by passing to the function field setting. However, their method has not yet been successfully applied to any of the orthogonal families. Furthermore, the success of this article seems to suggest that the $\nth$ centered moment statistic might offer a viable means to overcome the combinatorial barriers encountered in \cite{Gao} and \cite{LM13} when verifying the higher-order Katz-Sarnak comparison for the symplectic family.

\subsection{Number theory setup}
We now describe the main objects of study. Let $H^*_k(N)$ be the set of holomorphic cusp forms of
weight $k$ and level $N$ (with trivial nebentypus) which are \emph{newforms}. For $f\in H^*_k(N)$, denote by $L(s,f)$ the $L$-function attached to $f$. The completed $L$-function is given by 
\begin{equation}
    \Lambda(s,f) \ \defeq\ \left(\frac{\sqrt{N}}{2\pi}\right)^s \Gamma\left(s+\frac{k-1}{2}\right) L(s,f).
\end{equation}
It admits an entire continuation and satisfies the functional equation
\begin{equation}
        \Lambda(s,f) \ = \ \epsilon_f \, \Lambda(1-s,f),
\end{equation}
where $\epsilon_f = \pm 1$ is the root number. The family $H^*_k(N)$ splits naturally into two disjoint sub-families:
\begin{align}\label{subfam}
      H^+_k(N) \ &\defeq \  \{f\in H^*_k(N): \epsilon_f = +1\}, \hspace{20pt}
      H^-_k(N) \ \defeq \  \{ f\in H^*_k(N): \epsilon_f = -1\}. 
\end{align}
For each $f\in H_{k}^{*}(N)$, we denote the non-trivial zeros of $L(s, f)$ by $\rho_f = \tfrac{1}{2} + i \g_f$. The Generalized Riemann Hypothesis (GRH) for $L(s,f)$ asserts that $\g_f \in \R$. 

As in \cite{ILS}, we take 
\begin{align}
    R  \ = \   k^2N
\end{align}
as the working definition of the \emph{analytic conductor} for the families $H_{k}^{\bullet}(N)$ ($\bullet \in \{+, -, *\}$).  Our analysis is greatly simplified by all forms in the family having the same analytic conductor. Varying conductors are easily handled in 1-level calculations, but cause technical difficulties through cross terms once $n \ge 2$, see \cite{Mil2}.  

The \emph{one-level density} of $f \in H_{k}^{\star}(N)$ is a weighted sum over $\gamma_f$ given by
\begin{equation}\label{eq:dfphidef}
    D(f;\phi)\ \defeq\ \sum_{\g_f} \, \phi\left(\frac{\log R}{2\pi}\g_f\right),
\end{equation}
where the test function $\phi: \R \rightarrow \C$ is an even Schwartz function whose Fourier transform $\widehat{\phi}$\, has compact support. We denote this class of test functions by  $\mathcal{S}_{ec}(\R)$ and $\phi\in \mathcal{S}_{ec}(\R) $. Because of the rapid decay of $\phi$, the \emph{low-lying zeros} of $L(s,f)$  contribute the most to the one-level density \eqref{eq:dfphidef}. 

As in \cite{HM}, we use the following shorthand for taking averages over $H_{k}^{\bullet}(N)$:
\begin{equation}\label{eq:unifaverdef}
    \< Q(\,\cdot\, ) \>_\bullet \ = \  \< Q(\,\cdot\, ) \>_{H_{k}^{\bullet}(N)}\  \defeq\ \frac{1}{|H_k^{\bullet}(N)|} \, \sum_{f\in H_k^\bullet (N)} Q(f),
\end{equation}
where $\bullet \in \{+, -, *\}$  and $Q$ is any complex-valued function defined on $H_{k}^{\bullet}(N)$, e.g.,  $f\mapsto D(f;\phi)$ defined in \eqref{eq:dfphidef}.  

In this article, we study the $\nth$-centered moment of the one-level density, defined by
\begin{align}\label{defcentmo}
  \mathcal{D}_{n}^{\bullet}(N; \phi) \ := \  \left\< \ \left(\, D(\ \cdot \ ;\, \phi\, ) \, - \,  \< \, D(\ \cdot \ ;\, \phi\, ) \, \>_{H_{k}^{\bullet}(N)} \, \right)^{n} \ \right\>_{H_{k}^{\bullet}(N)}
\end{align}
where $n \ge 2$ is an integer, $N$ is prime, $k \ge 2$ is even, and $\phi\in \mathcal{S}_{ec}(\R)$ is a test function.  Readers should keep in mind that the weight $k$ is kept fixed and the level $N$ goes to $\infty$ through primes.  

\subsection{Random matrix theory setup}\label{sec: rmt setup}
Let  $\phi \in \mathcal{S}_{ec}(\R)$. The random matrix theory counterpart of the one-level density $D(f;\phi)$ is given by 
 \begin{align}\label{RMTdenfunc}
      Z(U; \phi)\ \defeq\ \sum_{n=1}^M \, \sum_{j=-\infty}^\infty \, \phi\left(\frac M{2\pi}(\theta_{n}+2\pi j)\right),
 \end{align}
 where $U$ is an $M$-by-$M$ special orthogonal matrix with eigenvalues $\{e^{i\theta_n}: n=1, \ldots, M\}$. In fact, the eigenvalues of $U$ always occur in complex-conjugate pairs. We have the following correspondence when comparing (\ref{RMTdenfunc}) to \eqref{eq:dfphidef}: 
 \begin{align*}
     M \ \ &\longleftrightarrow \ \ \log\,  (k^2N) \\
     U \, \in \, SO(M) \ \ &\longleftrightarrow \ \ f \, \in \,  H_{k}^{*}(N) \\
      \left\{ \theta_{n}+ 2\pi j\right\}_{\substack{  j  \in  \Z \\ 1 \le  n  \le M}} \ \ &\longleftrightarrow \ \  \{ \gamma_{f}\}. 
 \end{align*}
Here, it is more standard to use the notation
\begin{align}
    \E_{\SO(M)}
        \left[ Z(\ \cdot\ ; \, \phi\, ) \right] \ := \ \int_{\SO(M)} \,  Z(U; \phi) \, dU
\end{align}
instead of the bracket of  \eqref{eq:unifaverdef}. The Haar measure $dU$ on the compact Lie group $SO(M)$ is normalized to have total measure $1$.  When $\supp(\hphi) \subseteq [-1, 1]$, it is well-known that
    \begin{align}
        \mu_\pm &\ \defeq \ \sideset{}{^\pm}{\lim
}_{M\to\infty} \ \E_{\SO(M)}
        \left[ Z( \, \cdot \, ; \phi) \right]\ =\
        \hphi(0)\ + \ \frac12\int_{-1}^1 \hphi(y)\;\d y, 
    \end{align}
where \, 
$\sideset{}{^+}{\lim}\limits_{M\to \infty}:= \lim\limits_{\substack{M\to \infty \\ M \equiv 0 \, (\bmod\, 2)}}$ \ and \ \  $\sideset{}{^-}{\lim}\limits_{M\to \infty}:= \lim\limits_{\substack{M\to \infty \\ M \equiv 1 \, (\bmod\, 2)}}$. Similar to (\ref{defcentmo}),  the \emph{ $\nth$-centered moment} of $Z(\ \cdot\ ; \phi)$ is defined as
 \begin{align}\label{RMT1levden}
      \mathcal{Z}_{n}(M; \phi) \ :=  \  \E_{\SO(M)} \left[ \ \left(\, Z(\ \cdot \ ; \,\phi\, ) \ - \ \E_{\SO(M)}
        \left[ Z(\ \cdot \ ;\, \phi\, ) \right] \, \right)^n \ \right] \hspace{20pt} (M \, \in \, \mathbb{N}).
 \end{align}

\subsection{Main Results}

We are now ready to present our main result, which extends and generalizes \cite[Theorem 1.1]{ILS} and \cite[Theorem 1.6-1.7]{HM}. 

\begin{thm}\label{thm: KS density main}
Let $k, n\ge 2$ be positive integers with $k$  even. Assume GRH for $L(s, f)$ for all $f \in H_k^*(N)$, where $N$ is any prime or $N = 1$. Assume also RH for $\zeta(s)$ and GRH for all primitive Dirichlet $L$-functions.  Then for $\phi \in \mathcal{S}_{ec}(\R)$ with $\supp (\hphi) \subset \twoovern$, we have
\begin{align}\label{eq: density main}
    \lim_{\substack{N\to\infty \\ N\, \text{\rm  prime}}}  \ \mathcal{D}^{\pm}_{n}(N; \phi) \ \  = \ \  \sideset{}{^\pm}{\lim}_{M\to \infty} \  \mathcal{Z}_{n}(M; \phi).
\end{align}
The moments $\mathcal{D}^{\pm}_{n}(N; \phi)$ and $\mathcal{Z}_{n}(M; \phi)$ are defined in (\ref{defcentmo}) and (\ref{RMT1levden}). 
\end{thm}

Theorem \ref{thm: KS density main} follows immediately from Theorems \ref{thm:lfnsn-2} and \ref{thm:fullRMT} below. In fact, we prove precise formulae for each of the limits in (\ref{eq: density main}) whenever $\supp (\hphi) \subset \oneoverna $, where $a$ is an integer with  $0\le a\le n/2$. These formulae might be of independent interest. To state these results, we need to introduce the following quantities:  
\begin{align} \label{eq:fullNT_var}
        \sigma_\phi^2 \ \coloneqq \ 2\intii |y| \hphi(y)^2\;\d y,
\end{align}
and
\begin{equation}\label{eq: R def}
\begin{split}
    \mathcal{R}(&m, i; \phi) \ \defeq\ 2^{m-1} (-1)^{m+1} \sum_{\ell = 0}^{i-1} (-1)^{ \ell }\binom{m}{\ell} \left( - \frac{1}{2}\, \phi(0)^m \right. \\
    &+\left. \intinf \cdots \intinf \hphi(x_2) \cdots \hphi(x_{\ell + 1}) \intinf \phi(x_1)^{m -\ell} \, \frac{\sin( 2 \pi x_1 (1 + |x_2| + \cdots + |x_{\ell+1}|))}{2\pi x_1}\,  dx_1 \cdots dx_{\ell + 1} \right),
\end{split}
\end{equation}
and
\begin{equation}\label{eq: Snaphi def}
    S(n, a; \phi) \ \defeq  \ \sum_{\ell = 0}^{\lfloor \frac{a-1}{2} \rfloor} \frac{n!}{(n-2\ell)! \ell!} \, \mathcal{R}(n-2\ell, a -2\ell; \phi) \left(\frac{\sigma_{\phi}^2}{2} \right)^{\ell},
\end{equation}
where $\phi\in\mathcal{S}_{ec}(\R)$, $1 \le i \le m$, and $0 \le a \le n/2$.

We first state the number theory result.
\begin{thm}\label{thm:lfnsn-2}
Let $n, a$ be integers with $n \ge 2$, $0 \le a \le n/2$ and let $\phi \in \mathcal{S}_{ec}(\R)$. Under the same assumptions as Theorem \ref{thm: KS density main}, if $\supp(\hphi) \subset \oneoverna $, then
\begin{equation}\label{eq:thmextmomcompact}
\lim_{\substack{N\to\infty \\ N\, \text{\rm  prime}}}  \ \mathcal{D}_{n}^{\pm}(N; \phi)   \ = \ \mathbbm{1}_{\{\text{\rm even}\}}(n) \cdot (n-1)!!\ (\sigma_{\phi}^{2})^{n/2} \ \pm \  S(n, a; \phi)
\end{equation}
where $\mathbbm{1}_{\{\text{\rm even}\}}(n)$ is equal to  $1$ if $n$ is even and is $0$ if $n$ is odd.  
    
\end{thm}

Next we state the random matrix theory result.
\begin{thm}\label{thm:fullRMT}
Let $n, a$ be integers with $n \ge 2$, $0 \le a \le n/2$ and let $\phi \in \mathcal{S}_{ec}(\R)$. If $\supp(\hphi) \subset \oneoverna $, then
    \be
        \sideset{}{^\pm}{\lim}_{\substack{M\to\infty}} \,  \mathcal{Z}_{n}(M; \phi) \ = \ \mathbbm{1}_{\{n\,\text{\rm even}\}}\cdot (n-1)!!\ (\sigma^2_\phi)^{n/2} \ \pm \  S(n, a; \phi) 
    \ee
    where $\mathbbm{1}_{\{\text{\rm even}\}}(n)$ is equal to  $1$ if $n$ is even and is $0$ if $n$ is odd.  
\end{thm}

Hughes-Miller \cite[Theorem E.1]{HM} prove an analogue of Theorem \ref{thm: KS density main} for the \emph{full} family $H_k^*(N)$ under the restriction $n \le 2k$. We remove this restriction in the following theorem, whose proof is given in Appendix \ref{sec: completefamily}. 

\begin{thm}\label{thm:nosplitmoments}\label{thm:mock-Gaussian for D}
Let $k, n\ge 2$ be positive integers with $k$  even and let $\phi \in \mathcal{S}_{ec}(\R)$. Under the same assumptions as Theorem \ref{thm: KS density main}, if $\supp(\hphi) \subset \oneoverna $, then
\begin{align}\label{cleanlimires}
    \lim_{\substack{N\to\infty \\ N\, \text{\rm  prime}}}  \ \mathcal{D}_{n}^{*}(N; \phi) \ \  = \ \   \frac{1}{2}  \left[ \   \sideset{}{^+}{\lim}_{M\to \infty} \  \mathcal{Z}_{n}(M; \phi)\  +\  \sideset{}{^-}{\lim}_{M\to \infty} \  \mathcal{Z}_{n}(M; \phi) \right].
\end{align}
The moments $\mathcal{D}^{\pm}_{n}(N; \phi)$ and $\mathcal{Z}_{n}(M; \phi)$ are defined in (\ref{defcentmo}) and (\ref{RMT1levden}). 
    
\end{thm}

\subsection{Applications}

As noted in \cite{HM} and \cite{Mil}, another application of centered moments is in bounding the order of vanishing of $L$-functions at the central point. In Appendix \ref{sec: order of vanishing}, we show how to use Theorem \ref{thm:lfnsn-2} to bound the probability that a newform with negative sign will have order of vanishing exceeding some $r$ at the central point. Similar calculations may be done for the positive sign family. Our results provide the best known bounds (conditional on GRH) for order of vanishing at the central point when $r \ge 5$, surpassing \cite{ILS, HM, Boldy,LiMiller}. 
\subsection{Proof sketch and structure of the paper}

We evaluate the limit \eqref{eq:thmextmomcompact} using the \emph{explicit formula} and the \emph{Petersson trace formula}. We use the explicit formula to transform the sum over zeros \eqref{eq:dfphidef} into a weighted average of products of Hecke eigenvalues over primes in Lemma \ref{pluexpHeckpr}. After removing many lower order subterms with Lemmas \ref{lem:coeffs are 0} and \ref{lem: n' big}, we apply the Petersson trace formula and study the resulting sums of Kloosterman sums over primes; see \eqref{eq:E Petersson}. We assume GRH for $L(s, f)$ when applying the Petersson trace formula; see Remark \ref{rem: complementary sum}.

Using Lemma \ref{lem:genkloos2}, we convert the Kloosterman sums into sums over Gauss sums. Assuming GRH for Dirichlet $L$-functions, we show in Lemma \ref{lem:char} that the terms involving Gauss sums with non-principal characters contribute negligibly in the limit when $\supp (\hphi) \subset \twoovern$. This requires strong bounds for various character sums over primes, hence the need to assume GRH for Dirichlet $L$-functions. Hughes-Miller \cite[Thm. 1.1, Thm. 1.3]{HM} prove results without GRH for Dirichlet $L$-functions for $\supp(\widehat{\phi}) \subset (-\frac{1}{n}, \frac{1}{n})$ (i.e., the case where $a=0$ in Theorem \ref{thm:lfnsn-2}-\ref{thm:fullRMT}); we use GRH for Dirichlet $L$-functions to extend the support to $\supp(\widehat{\phi}) \subset (-\frac{2}{n}, \frac{2}{n})$. 

We are left to handle certain smooth sums over primes (Proposition \ref{lem:sums to integrals}) and a convolution sum of Ramanujan sums (Proposition \ref{lem:ils7}). We arrive at Theorem \ref{thm:lfnsn-2} upon very careful bookkeeping of the resulting combinatorics (see Sections \ref{sec:combining} and \ref{sec:main term}).

Beyond the regime $\supp(\widehat{\phi}) \subset (-\frac{1}{n-1}, \frac{1}{n-1})$ proven by \cite{HM},  more complicated terms emerge as the size of  $\supp(\hphi)$ increases. This serves as the primary obstacle to generalizing the work of  \cite{HM}. The main insight of our extension lies in the observation that many of these terms actually \emph{vanish in the limit} (see Lemma \ref{lem:coeffs are 0} and Proposition \ref{lem:nonmaintermcalc}). This enables us to ignore the very intricate combinatorics behind these terms. The remaining terms contribute to the limit and exhibit nicer symmetries. We are able to obtain an integral representation for the these terms in Proposition \ref{lem:maintermcalc} upon making our way through the combinatorial jungle.  Our work features many elaborate combinatorial simplifications that result in exact matching with the calculations from the random matrix theory side. As with number theory, the key result that allows us to obtain greater support in random matrix theory is the vanishing of many of the complicated terms which emerge at larger supports (see Lemma \ref{cor:classes_cancel}).

The structure of this paper is as follows. In Section \ref{sec:preliminaries section}, we review the notations and conventions, and state some needed estimates. In Section \ref{sec:number theory section}, we work with the geometric side of the Petersson formula with the focus of locating the main contributions in the expansion.  We also prove Theorem \ref{thm:lfnsn-2} assuming two key propositions (Prop. \ref{lem:nonmaintermcalc}-\ref{lem:maintermcalc}). In Section \ref{sec:main term}, we evaluate the main contributions as well as complete the proofs of Propositions \ref{lem:nonmaintermcalc}-\ref{lem:maintermcalc}. In Section \ref{sec:rmt calc} we work on the random matrix theory side and prove Theorem \ref{thm:fullRMT}.  

In Appendix \ref{sec:pfs nt lemmas}, we prove Lemma \ref{lem:coeffs are 0} regarding the combinatorial expansion for $\mathcal{D}_{n}^{*}(N; \phi)$ which serves as the starting point of the number-theoretic calculations. In Appendix \ref{sec: completefamily}, we give a sketch of proof for Theorem \ref{thm:mock-Gaussian for D} regarding the non-split family. In Appendix \ref{sec:RMT Lemmas}, we include more details for the random matrix theory calculations. In Appendix \ref{sec: order of vanishing} we use Theorem \ref{thm:lfnsn-2} to bound the proportion of cuspidal newforms vanishing to a certain order at the central point.

%%%%%%%%%%%%%%%%%%%%%%%%%%%%%%%%%%%%%%%%%%%%%%%%%%%%%%%%%%%%%%%%%%%%%%%%%%%%%%%%%%%%%%%%%%%%%%%%%%%%%%%%%%%%%%%%%%%%%%%%%%%%%%%%%%%%%
%%%%%%%%%%%%%%%%%%%%%%%%%%%%%%%%%%%%%%%%%%%%%%%%%%%%%%%%%%%%%%%%%%%%%%%%%%%%%%%%%%%%%%%%%%%%%%%%%%%%%%%%%%%%%%%%%%%%%%%%%%%%%%%%%%%%%
%%%%%%%%%%%%%%%%%%%%%%%%%%%%%%%%%%%%%%%%%%%%%%%%%%%%%%%%%%%%%%%%%%%%%%%%%%%%%%%%%%%%%%%%%%%%%%%%%%%%%%%%%%%%%%%%%%%%%%%%%%%%%%%%%%%%%
%%%%%%%%%%%%%%%%%%%%%%%%%%%%%%%%%%%%%%%%%%%%%%%%%%%%%%%%%%%%%%%%%%%%%%%%%%%%%%%%%%%%%%%%%%%%%%%%%%%%%%%%%%%%%%%%%%%%%%%%%%%%%%%%%%%%%

\section{Preliminaries}\label{sec:preliminaries section}
%%%%%%%%%%%%%%%%%%%%%%%%%%%%%%%%%%%%%%%%%%%%%%%%%%%%%

\subsection{Notations and Conventions}

In this article,  $e(x) := e^{2\pi \i x}$, the Fourier transform and its inverse transform are given by 
    \begin{equation}\label{fouT}
        \hphi(y)\ \coloneqq \ \intinf \phi(x) e(-xy) \;\d x, \ \ \ \ \ \phi(x)\ \coloneqq \ \intinf \hphi(y) e(xy) \;\d y
    \end{equation}
    for $x, y \in \R$. The Mellin transform and its inverse transform are
   \begin{align}\label{melT}
     \hspace{-7pt}  \widetilde{\psi}(s) \ :=\   \int_{0}^{\infty} \, \psi(x) x^{s-1} \, \d x, \hspace{33pt} \psi(x) \ :=\ \int_{(\sigma)} \, \widetilde{\psi}(s) x^{-s} \, \frac{ds}{2\pi i}
   \end{align}
   for $x>0$ and $s\in \C$ in a vertical strip, provided that the integrals of (\ref{fouT}) and (\ref{melT}) converge absolutely.  

% We use the following normalization: %\be\label{eq:defn Rnew} 

For $A \subset \R$, let  
    \begin{equation} \label{eq: indicator}
        \twocase{\I_{\{x \in A\}} \ \coloneqq \ }{1}{if $x\in A$}{0}{otherwise.}
    \end{equation}
We will suppress the argument of a characteristic function when it is clear from context. For $x,y \in \R$, let  
    \begin{equation} 
        \twocase{\delta(x,y) \ \coloneqq \ }{1}{if $x = y$}{0}{otherwise.}
    \end{equation}

For $x,y \in \Z$, let $(x, y)$ denote the greatest common divisor of $x$ and $y$. Set $(x, y^{\infty}) = \max_{n \in \N} (x, y^{n})$ and $(x^{\infty}, y) = \max_{n \in \N}(x^{n}, y)$.

To avoid potential confusion, we adopt the following set of conventions throughout this work. 
\begin{enumerate}
    \item We will use `$\i$' to denote the imaginary unit (i.e., $\i^2=-1$) and `$i$' for the indices of summations. 

    \item We will use $p$, $p_{i}$'s, $q_{j}$'s to denote prime numbers.

    \item The test function $\phi: \R \rightarrow \C$ is an even Schwartz function with its Fourier transform $\widehat{\phi}$\,  having compact support.

    \item The weight $k$ is kept fixed and the level $N$ goes to $\infty$ through primes

    \item A quantity is considered to be \emph{negligibly small} if it is $o(1)$ as $N \to \infty$.

   \item The implicit constants for $O$, $\ll$, $\gg$, $\asymp$, etc. may depend on $k$, $\epsilon$, $n$, $a$, and of course the test function $\phi$ (cf. Theorem \ref{thm:lfnsn-2}).  We will omit such dependencies to simplify notations. 

    \item  We shall frequently adopt the \emph{`$\epsilon$-convention'}, i.e.,  $\epsilon> 0$ is an arbitrary small quantity and $O(N^{-\epsilon})\cdot O(N^{- C(k, \epsilon, n, a)\cdot \epsilon})=O(N^{-\epsilon})$ (say).

    \item We always assume GRH for $L(s,f)$ for any $f\in H_{k}^{*}(N)$ with $N=1$ and primes $N$.

\end{enumerate}

\subsection{Analytic Preliminaries}

We will frequently encounter the following exponential/character sums while proving our main theorems.

\begin{defn}
Let $m,n,q \in \Z$ with $q\ge 1$ and $\chi \, (\bmod\, q)$ be  a Dirichlet character. We define
    \begin{align}
        G_\chi(n)\ &\coloneqq \ \sum_{a \bmod q} \, \chi(a) e(an/q), \\
        R(n,q) \ &\coloneqq \ \sideset{}{^\ast}\sum_{a \bmod q} \,  e(an/q) ,\\
        S(m,n;q)\ &\coloneqq \ \sideset{}{^\ast}\sum_{a \bmod q} \, e\left(\frac{ma}{q} + \frac{n\bar{a}}{q}\right),
    \end{align}
where  $\ast$ restricts the summation to the reduced residue classes $a \, (\bmod\, q)$, and  $a \bar{a} \equiv 1 \, (\bmod\, q)$.
    
\end{defn}

The sums $G_\chi(n)$, $ R(n,q)$, and $ S(m,n;q)$ are commonly known as the Gauss sum, the Ramanujan sum, and the Kloosterman sum respectively. When $\chi = \chi_0$ (i.e., the principal character $(\bmod\, q)$), we have $G_{\chi_0}(n)= R(n,q)$.  Also, $R(n,q)=S(0,n; q)$. The  Ramanujan sum admits the following explicit evaluation:
\begin{equation}\label{eq:vonsterneck}
    R(n, q) \ = \ \sum_{d|(n,q)} \mu(q/d) d \ = \  \mu \left(\frac{q}{(q, n)} \right) \frac{\varphi(q)}{\varphi\left(\frac{q}{(q, n)} \right)},
\end{equation}
where $\mu(\cdot)$ and $\phi(\cdot)$ are the M\"obius $\mu$-function and the Euler totient-function respectively. We will need the multiplicativity of the Ramanujan sum as well:
\begin{align}\label{eqn Ramult}
    R(n,q_{1}q_{2}) \ =  \ R(n, q_{1}) R(n, q_{2})
\end{align}
for $(q_{1},q_{2})=1$ and $n\in \Z$.

The following bounds are particularly handy in showing various sums and integrals to be negligibly small.  We have
\begin{align}
    |G_\chi(n)|  \ &\le \ \sqrt{q},
\end{align}
and 
\begin{equation}\label{eq:estimate Kloosterman}
        |S(m,n;q)| \ \leq \ (m,n,q)\ \sqrt{\min\left\{ \frac{q}{(m,q)} , \frac{q}{(n,q)} \right\}}\ \ \tau(q),
    \end{equation}
 where $\tau(\cdot)$ is the divisor function. The bound  \eqref{eq:estimate Kloosterman} is a convenient reformulation of a well-known result of  A. Weil, see  \cite[eq. 2.13]{ILS}.

We often need to handle the Kloosterman sums in more refined ways than merely applying \eqref{eq:estimate Kloosterman}. In the context of low-lying zeros (see \cite{ILS}, \cite{HM}), it is advantageous to expand the Kloosterman sums in terms of  Dirichlet characters and Gauss sums. The following lemma is a generalization of \cite[Lemma C.1]{HM}.  

\begin{lem}\label{lem:genkloos2}
    Let $N$ be a prime not dividing $bQm$. Then
    \begin{equation}\label{eqn KlGaexp}
        S(m^2, NQ; Nb) \ =\  -\frac{1}{\varphi(b)}\sum_{\chi (b)} G_\chi (m^2) G_\chi((Q, b^\infty)) \overline{\chi} \left(\frac{Q}{(Q, b^\infty)}\right) \chi(N).
    \end{equation}
\end{lem}

\begin{proof}
    Set $r = (Q, b^\infty)$ and $Q' = Q/r$. Then $(Q', b) = 1$. Using the orthogonality relation and opening up the Kloosterman sum by its definition, observe that
    \begin{align}
        S(m^2, NQ; Nb) \ &= \ \frac{1}{\varphi(b)} \, \sum_{\chi  (b)} \sideset{}{^*}\sum_{a \, (b)} \, \chi(a) \chib{Q'} S(m^2, Nra; Nb)\nonumber \\
        &= \ \frac{1}{\varphi(b)} \, \sum_{\chi (b)} \  \chib{Q'} \sideset{}{^*}\sum_{d \, (Nb)} e \left(\frac{m^2d}{Nb} \right) \sideset{}{^*}\sum_{a \, (b)} \, \chi(a) \left(\frac{ra\overline{d}}{b} \right).
        \end{align}
 Making  a change of variables $a\to ad$ in the $a$-sum and breaking up the $d$-sum by  $d = u_1N +u_2b$ with $(u_1, b) = 1$ and $(u_2, N) = 1$, it follows that
\begin{align}
    S(m^2, NQ; Nb) \ &= \ \frac{1}{\varphi(b)} \sum_{\chi \, (b)} \, \chib{Q'} G_\chi(r) \, \sideset{}{^*}\sum_{d \, (Nb)} \, \chi(d) e \left(\frac{m^2d}{Nb} \right) \nonumber\\
    \ &= \ \frac{1}{\varphi(b)} \sum_{\chi \, (b)} \, \chib{Q'} G_\chi(r) \chi(N)  \sideset{}{^*}\sum_{u_1 (b)}  \chi(u_1)  e \left(\frac{m^2u_1}{b} \right) \sideset{}{^*}\sum_{u_2 \, (N)} e \left(\frac{m^2u_2}{N} \right). \label{expGauKl}
\end{align}
The $u_{2}$-sum and $u_{1}$-sum of (\ref{expGauKl}) evaluate to  $-1$ and $G_{\chi} (m^2)$ respectively as  $(m^2, N) = 1$. The result follows. 
\end{proof}

\begin{lem}\label{lem:expsum} 
    Let $\chi \, (\bmod\, b)$ be a primitive Dirichlet character.  Under GRH for $L(s, \chi)$, we have 
    \be
        \sum_{n\le x} \Lambda(n) \chi(n) n^{-it} \ = \ O_{\epsilon}\left(x^{1/2} (bxt)^\epsilon\right)
    \ee
    for any $x\ge 2$ and $t\in\R$. 
\end{lem}

\begin{proof}
    This follows from a standard prime-number-theorem type argument, see \cite[Chapter 5]{iwaniec_kowalski_2004}, \cite{Da}, or \cite[Chapter 13]{MV07}. As a quick sketch, we have
    \begin{equation*}
        \sum_{n\le x} \Lambda(n) \chi(n) n^{-it} \ = \ \int_{\frac{3}{2}-ix}^{\frac{3}{2}+ix}\frac{L'}{L} (s+it, \chi) \frac{x^s}{s}ds \, + \, O\left(x^{1/2}\right)  \ = \ \sum_{|\gamma - t| \le x} \frac{x^{\rho - it}}{\rho - it} \, + \,  O_{\epsilon}\left(x^{1/2} (bxt)^\epsilon \right) 
    \end{equation*}
using GRH and the estimate
\begin{equation}
    \frac{L'}{L}(s,\chi) \ll (\log b|s|)^2
\end{equation}
 for $-1 \le \Re(s) \le 2$.  The desired result then follows from the fact that the  number of zeros satisfying $u \le \gamma - t \le u+1$ is $\ll \log \, (b(|u|+|t|)$. 
\end{proof}

The Bessel functions of the first kind occur in the Petersson formula (see Lemma \ref{lemcomplsum}) and hence frequently in this paper.  We collect some results for them. 

\begin{lem}\label{lem:Bessel} 
    Let $k\geq 2$ be an integer. Then the following bounds are satisfied for $x>0$:
    \begin{enumerate}
        \item\label{lb:1} $J_{k-1}(x) \ll_{k} 1$,
        \item\label{lb:2} $J_{k-1}(x) \ll_{k} x$,
        \item\label{lb:3} $J_{k-1}(x) \ll_{k} x^{k-1}$,
        \item\label{lb:4} $J_{k-1}(x) \ll_{k} x^{-\foh}$.
        %\item\label{lb:5} $2J'_\nu(x) = J_{\nu-1}(x) - J_{\nu+1}(x)$.
    \end{enumerate}
\end{lem}

\begin{proof}
   See  \cite{GR,Wat}.
\end{proof}

We will also utilize the Mellin integral representation for the Bessel function. 

\begin{lem}
We have
\begin{equation}\label{eq: Jk inverse Mellin}
    \Jk(x)\ = \ \frac{1}{2\pi i } \int_{\Re(s) = c} G_{k-1}(s) x^{-s} \, \d s
\end{equation}
for $x>0$ and  $1 - k < c < \frac{3}{2}$, where 
    \begin{align}
G_{k-1}(s)  \ := \ 2^{s-1} \, \Gamma\left( \frac{k-1 + s}{2}  \right)\left/ \Gamma\left( \frac{k+ 1 - s}{2} \right)\right 
 . \label{eq: Jk Mellin}
\end{align}

\end{lem}

\begin{proof}
    See \cite[(6.561.14)]{GR}. 
\end{proof}

\subsection{Automorphic Preliminaries}
We collect the essential results from the standard references from \cite[Chp. 14]{iwaniec_kowalski_2004}, \cite[Chp. 6-7]{Iw97},  \cite[Sect. 2-3]{ILS}.

Let $k$ and $N$ be positive integers with $k$ even and $N$ prime. Recall that $H_{k}^{\bullet}(N)$ ($\bullet \in \{+, -, *\}$) denotes the set of holomorphic cuspidal newforms of
weight $k$ and level $N$,  depending on the sign of the functional equation.  From \cite[eq. (2.73)]{ILS}, we have the following dimension formulae: 
\begin{equation}\label{eq:number of terms in hkpm}
|\hkpmn| \ = \ \frac{k-1}{24}N  \ + \  O\left( (kN)^{5/6}
\right),
\end{equation}
and
 \begin{equation} \label{eq:number of terms in hkn}
    |\hkn| \ = \ \frac{k-1}{12}N  \ + \  O\left( (kN)^{5/6} \right).
\end{equation}

Every $f\in H^*_k(N)$ has a Fourier expansion of the form
\begin{equation}
    f(z)\ =\ \sum_{n=1}^\infty \lambda_f(n) n^{\frac{k-1}{2}} e(nz)
\end{equation}
for $z\in \mathbb{H}:= \{ x+iy: x\in \R,\,  y > 0\}$, where $\lambda_{f}(1)=1$. 
The $L$-function associated with $f$ is defined by the Dirichlet series
    \begin{equation}
        L(s,f)\ =\ \sum_{n=1}^\infty \, \lambda_f(n) n^{-s} 
    \end{equation}
 which converges absolutely on $\re s >1$. The completed $L$-function is given by 
\begin{equation}\label{eq:completed_L_func}
    \Lambda(s,f) \ \defeq\ \left(\frac{\sqrt{N}}{2\pi}\right)^s \Gamma\left(s+\frac{k-1}{2}\right) L(s,f).
\end{equation}
It admits an entire continuation and satisfies the functional equation
\begin{equation}
        \Lambda(s,f) \ = \ \epsilon_f \, \Lambda(1-s,f).
\end{equation}
The root number $\gep_f$ admits a nice formula in our case. 

\begin{lem}\label{lem:sign in terms of lambdaN}
    If $f\in\hkn$ and $N$ is prime, then
    \begin{equation}\label{eq:signfneqexpils}
        \gep_f \ = \ -\i^k \gl_f(N) \sqrt{N}.
    \end{equation}
    In particular, we have $|\gl_f(N)| = 1/\sqrt{N}$. 
\end{lem}

\begin{proof}
    See \cite[eq. 3.5]{ILS}. 
\end{proof}

The following Hecke relations will be crucial in demonstrating that the number theoretic combinatorics align with those of random matrix theory.  

\begin{lem}\label{lem:multiplicativity of fourier coeffs}
    Let $f \in \hkn$. 
\begin{enumerate}
    \item For any any $m,n\ge 1$,   \begin{equation}
        \gl_f(m) \gl_f(n) \ = \ \sum_{\substack{d|(m,n) \\ (d,N) = 1}} \gl_f\left( \frac{mn}{d^2} \right).
    \end{equation}
In particular, if $(m,n) = 1$ then
\begin{equation}\label{eqn multHec}
    \gl_f(m) \gl_f(n)  \ = \  \gl_f(mn).
\end{equation}

\item For a prime $p\notdiv N$, we have
\begin{equation}\label{eq:lambda(p) squared}
    \gl_f(p)^2 \ = \ \gl_f(p^2) \, +\,  1.
\end{equation} 

\item For $n\ge 1$ and a prime $p\nmid N$, we have
\begin{equation}\label{eq:GuyLambdaExpansion}
    \lambda_f(p)^n\ = \ \sum_{\alpha = 0}^{[n/2]} \left[ \binom{n}{\alpha} - \binom{n}{\alpha - 1} \right]\lambda_f(p^{n- 2\alpha}).
\end{equation}
\end{enumerate}
\end{lem}

\begin{proof}
    Only the last property is less well-known, see \cite{Guy} for its proof. 
\end{proof}

Now, consider
\begin{equation}\label{eq:deltakndefeq}
    \Delta_{k,N}^\bullet(n) \ \coloneqq \ \sum_{f \in \hksn} \gl_f(n), \quad \bullet \in \{+,-,\ast\}.
\end{equation}
Splitting by sign with Lemma~\ref{lem:sign in terms of lambdaN},  we have 
\begin{align}
    \Delta_{k,N}^\pm(n)  \ = \ \sum_{f\in\hkn} \frac12(1\pm\epsilon_f) \gl_f(n) 
     \ = \ \frac12 \, \Delta_{k,N}^\ast(n)\ \mp\ \frac{\i^k \sqrt{N}}{2} \Delta_{k,N}^\ast(nN) \label{eqdeltapm}
\end{align}
whenever $N$ is a prime and $(n,N)=1$. We have the following useful form of the Petersson formula.
\begin{lem}\label{lemcomplsum}
    If $N$ is prime and $(n,N^2)\mid N$ then
    \begin{equation}\label{eqdeltastarprimeinf}
        \Delta_{k,N}^\ast(n)  \ = \  \Delta_{k,N}'(n) \ + \ \Delta_{k,N}^\infty(n) ,
    \end{equation}
    where
    \begin{multline}\label{proptof}
        \Delta_{k,N}'(n)\ =\ \frac{(k-1)N}{12\sqrt{n}}\ \delta_{n,\Box_Y} \\ 
        + \frac{(k-1)N}{12} \sum_{\substack{(m,N) = 1 \\ m \le Y}} \frac{2\pi \i^k}{m} \sum_{\substack{ c\equiv 0 \bmod N \\ c\geq N}} \frac{ S(m^2,n;c)}{c} \, \jk{m^2n}{c}
    \end{multline}
    with $\delta_{n,\Box_Y} = 1$ only if $n=m^2$ with  $m\leq Y$ and $0$ otherwise, and $\Delta_{k,N}^\infty(n)$ defined in \cite[Lem. 2.12]{ILS}. 
\end{lem}

\begin{proof}
   See \cite[Prop. 2.1, 2.11 and 2.15]{ILS}. 
\end{proof}

\begin{rem}\label{rem: complementary sum}
    The piece $\Delta_{k,N}^\infty(n)$ is called the \emph{complementary sum}. By \cite[Lemma A.1]{HM}, assuming GRH for all $L(s,f)$ with $f \in H_k^*(1) \cup H_k^*(N)$, the complementary sum does not contribute in all cases appearing in this paper.
\end{rem}

 We have the following
lemma.
\begin{lem} \label{lem:trivial bound}
    Assume $(n,N)=1$. Then
    \iffalse
    \begin{equation}\label{eq:deltaprime}
        \frac{1}{|\hkn|}\ \Delta_{k,N}'(n) \ = \ \frac{1}{\sqrt{n}}\ \delta_{n,\Box_Y} + O\left(n^{(k-1)/2} N^{-k+1/2+\gep}\right) ,
    \end{equation}
    and
    \fi
    \begin{equation}\label{eq:219}
        \frac{1}{|\hkn|}\ \Delta_{k,N}'(Nn) \ \ll \ \sqrt{n} N^{-\frac{3}{2}+\gep}.
    \end{equation}
\end{lem}

\begin{proof}
     We take $Y = N^{\epsilon}$ and write $c =
bN$ for $c \equiv 0 \bmod N$. Using \eqref{eq:number of terms in hkpm}, the Weil bound \eqref{eq:estimate
Kloosterman} and  the bound $J_{k-1}(x) \ll x$ from Lemma~\ref{lem:Bessel}, the result follows immediately.
\end{proof}

\subsection{Density and moment sums}\label{sect:intro_mmt sums}

Let $f\in \hkn$ and $\phi \in \mathcal{S}_{ec}(\R)$. Substituting the explicit formula for $L(s,f)$ (see \cite[Sect. 4]{ILS}) into the one-level density function
\begin{eqnarray}\label{eq:expformexpsumzerosdfphi}
    D(f;\phi) & \ := \ & \sum_{\gamma_f} \, \phi\left( \frac{\log R}{2\pi} \gamma_f \right),
\end{eqnarray}
we have
\begin{equation}\label{eqdfphiexpansion}
    D(f;\phi)\ =\ \widehat{\phi}(0) \ + \  \foh \phi(0)  \ - \  P(f;\phi) \ + \  O\left( \frac{\log \log R}{\log R} \right),
\end{equation}
where $R = k^2 N$ and 
\begin{equation}\label{eq:P in terms of gl}
    P(f;\phi)\ :=\ \sum_{p \notdiv N} \, \gl_f(p) \phir{} \pfrac{}.
\end{equation}
See \cite[eq. (4.25)]{ILS} and the relevant remarks of \cite[pp. 88]{ILS} and \cite[pp. 129]{HM}.  Using \eqref{eqdfphiexpansion}, \cite{HM} expresses the $\nth$-centered moments in terms of sums over primes (see \cite[Sect. 2.3]{HM}).

\begin{lem}\label{pluexpHeckpr}
If $\supp(\hphi) \subset (-1,1)$, we have\iffalse
\begin{equation}\label{eq:full_grp_in_terms_of_S1}
        \lim\limits_{\substack{N\to\infty \\ N\, \text{\rm  prime}}} \,  \,  \mathcal{D}_{n}^{\bullet}(N; \phi) \ =\ (-1)^n \lim_{N\to\infty} S_1^{(n)}(N)
\end{equation}
and\fi
\begin{equation}\label{eq:nth_mmt_in_terms_of_S}
   \lim\limits_{\substack{N\to\infty \\ N\, \text{\rm  prime}}}  \,  \mathcal{D}_{n}^{\pm}(N; \phi)  \ =\ (-1)^n  \lim\limits_{\substack{N\to\infty \\ N\, \text{\rm  prime}}} \,  S_1^{(n)}(N)\ \pm\ (-1)^{n+1}  \lim\limits_{\substack{N\to\infty \\ N\, \text{\rm  prime}}} \,  S_2^{(n)}(N)
\end{equation}
provided the limits on the right side exist, where
\begin{equation}\label{eq:S_1}
    S_1^{(n)}(N)\ \defeq\ \sum_{p_1 \nmid N , \,\ldots,\,  p_n \nmid N} \, \prod_{j=1}^n \left( \testfn{p_j} \testcomp{p_j} \right) \left\< \prod_{j=1}^n \lambda_f(p_i) \right\>_\ast
\end{equation}

\no and

\begin{equation}\label{eq:S2n}
    S_2^{(n)}(N)\ \defeq\ \ \i^k \sqrt{N} \sum_{p_1 \nmid N , \ldots, \, p_n \nmid N} \, \prod_{j=1}^n \left( \testfn{p_j} \testcomp{p_j} \right) \left\< \lambda_f(N) \prod_{j=1}^n \lambda_f(p_i) \right\>_\ast.
\end{equation}
\end{lem}

\begin{prop}
Under the same assumptions of Theorem \ref{thm: KS density main}, we have
    \be\label{eq:S1nresult}
    \lim_{\substack{N\to\infty\\N \text{ prime}}} S_1^{(n)}(N)\ = \ \begin{cases}
        (n-1)!!(\sigma_{\phi}^{2})^{n/2} & \text{if $n$ is even}\\
        0 & \text{if $n$ is odd}
    \end{cases},
\ee
for $\phi\in\mathcal{S}_{ec}(\R)$ with  $\supp(\hat{\phi}) \subset \twoovern$,  
where $\sigma_\phi^2$ is defined in \eqref{eq:fullNT_var}.
\end{prop}

\begin{proof}
    See Appendix~\ref{sec: completefamily}. 
\end{proof}

\begin{rem}
 \cite[Thm.  E.1]{HM} prove an analogous result but with an extra restriction $n\le 2k$. 
\end{rem}

In Sections \ref{sec:number theory section} and \ref{sec:main term}, we evaluate $S_{2}^{(n)}(N)$. The main result is Proposition \ref{lem: S2n eval}, in which we express $S_{2}^{(n)}(N)$ in terms of $S(n,a; \phi)$ (see \eqref{eq:thmextmomcompact})   .

%%%%%%%%%%%%%%%%%%%%%%%%%%%%%%%%%%%%%%%%%%%%%%%%%%%%%%%%%%%%%%%%%%%%%%%%%%%%%%%%%%%%%%%%%%%%%%%%%%%%%%%%%%%%%%%%%%%%%%%%%%%%%%%%%%%%%
%%%%%%%%%%%%%%%%%%%%%%%%%%%%%%%%%%%%%%%%%%%%%%%%%%%%%%%%%%%%%%%%%%%%%%%%%%%%%%%%%%%%%%%%%%%%%%%%%%%%%%%%%%%%%%%%%%%%%%%%%%%%%%%%%%%%%
%%%%%%%%%%%%%%%%%%%%%%%%%%%%%%%%%%%%%%%%%%%%%%%%%%%%%%%%%%%%%%%%%%%%%%%%%%%%%%%%%%%%%%%%%%%%%%%%%%%%%%%%%%%%%%%%%%%%%%%%%%%%%%%%%%%%%
%%%%%%%%%%%%%%%%%%%%%%%%%%%%%%%%%%%%%%%%%%%%%%%%%%%%%%%%%%%%%%%%%%%%%%%%%%%%%%%%%%%%%%%%%%%%%%%%%%%%%%%%%%%%%%%%%%%%%%%%%%%%%%%%%%%%%

\section{Proof of Theorem \ref{thm:lfnsn-2} assuming Propositions \ref{lem:nonmaintermcalc} and \ref{lem:maintermcalc}}\label{sec:number theory section}
%%%%%%%%%%%%%%%%%%%%%%%%%%%%%%%%%%%%%%%%%%%%%%%%%%%%%

In this section, we prove Theorem \ref{thm:lfnsn-2} assuming the key Propositions \ref{lem:nonmaintermcalc} and \ref{lem:maintermcalc}. We prove these propositions in Section \ref{sec:main term}. 
First we decompose $S_2^{(n)}(N)$ (see \eqref{eq:S2n}) into subterms and show that many of these subterms vanish as $N\to \infty$ through primes. Then in Section \ref{sec:combining}, we apply Propositions \ref{lem:nonmaintermcalc} and \ref{lem:maintermcalc} to complete the proof of Theorem \ref{thm:lfnsn-2}.

%\subsection{Eliminating subterms of $S_2^{(n)}(N)$.} 

\subsection{Combinatorial expansion and cleaning}\label{combiexp}

We rewrite the sums over primes in \eqref{eq:S2n} as sums over powers of \emph{distinct} primes. This facilitates the applications of the Hecke relations (Lemma \ref{lem:multiplicativity of fourier coeffs}) and the Petersson formula (Lemma \ref{lemcomplsum}). More precisely,   suppose $p_1\cdots p_n = q_1^{n_1}\cdots q_\ell^{n_\ell}$ in \eqref{eq:S2n}, where   $q_1,\ldots,q_{\ell}$ are \emph{distinct} primes, $n\ge 2$ and $\ell\ge 1$.   We have $\langle \lambda_f(N) \prod_{j=1}^n \lambda_f(p_i) \rangle_\ast = \langle\lambda_{f}(N) \prod_{j=1}^{\ell} \lambda_{f}(q_{j})^{n_{j}}\rangle_\ast$. We then apply the Hecke relations \eqref{eq:GuyLambdaExpansion} to each $\lambda_{f}(q_{j})^{n_{j}}$ and then use \eqref{eqn multHec}. Then we remove the distinctness condition from the sum over primes using a delicate inclusion-exclusion process. We conclude this process with the following lemma.
\begin{lem}\label{lem:coeffs are 0}
We have
    \begin{align}\label{eq: S2n E}
     S_2^{(n)}(N) = \sum_{0 \le \omega \le n} \sum_{0 \le n' \le n} \sum_{\substack{\vec{n} :=(n_1,\dots,n_{\omega}) \\ n_j > 1 \\ n_{1}+\cdots +n_{\omega}=n'}} \sum_{\substack{\vec{m} :=(m_1,\dots,m_{\omega}) \\ m_j \equiv n_j \pmod{2} \\ 0\le  m_j < n_j}} C_{\vec n, \vec m} E(\vec n, \vec m)
\end{align}
for some explicit constants $C_{\vec{n}, \vec{m}}$, where\footnote{We have omitted the dependence on $\omega$ and $n'$ in the notation $E(\vec{n}, \vec{m})$ as it is implicitly contained in $\vec{n}, \vec{m}$. We have also suppressed the dependence on $N$ for ease of notation.}
\begin{align}
E(\vec n, \vec m) \defeq \i^k \sqrt{N} &\qsumn{\omega} \prod_{j=1}^{\omega} \left( \testfn{q_j}^{n_j}  \testcomp{q_j}^{n_j} \right) \nonumber\\
&\hs{1} \times \psumn{n-n'} \prod_{i=1}^{n-n'} \testfn{p_i} \testcomp{p_i}\left< \lambda_f(NQ)\right>_\ast \label{eq:E with omega}
\end{align}
and
\begin{align}\label{Qdef}
    Q \ :=\   \prod_{j=1}^{\omega} \, q_j^{m_j} \, \cdot\,  \prod_{i=1}^{n-n'} \, p_{i}.
\end{align}
\end{lem}
Lemma \ref{lem:coeffs are 0} is quite involved combinatorially, and we prove it in Appendix \ref{sec:pfs nt lemmas}. In particular, we carefully decompose $S_2^{(n)}(N)$ into sums over distinct primes in order to establish the condition $m_j < n_j$ in \eqref{eq: S2n E}.

Lemma \ref{lem:coeffs are 0} allows us to apply the Petersson trace formula in the following section, as we have expressed $S_2^{(n)}(N)$ in terms of the average of a single Fourier coefficient $\lambda_f(NQ)$, as opposed to the product of Fourier coefficients in the definition \eqref{eq:S2n}. The coefficients $ C_{\vec{n}, \vec{m}}$ are difficult to calculate in general, but as a consequence of Proposition \ref{lem:nonmaintermcalc} we only need to determine them in specific cases (see \eqref{eq: S2n El}).

\subsection{Cleaning with Weil's bound and Prime Number Theorem (PNT)}

The following result states that $E(\vec{n}, \vec{m})$ contributes to $\lim_{\substack{N\to\infty \\ N\, \text{\rm  prime}}}  \ \mathcal{D}^{\pm}_{n}(N; \phi)$ only if ``most'' of the indices satisfy $n_j = m_j$. 

\begin{lem}\label{lem: n' big}
    Suppose $\supp (\hphi) \subset \oneoverna$ for some $a \le n/2$. If $n' \ge a$, then $\Enm = \ONe$.
\end{lem}

\begin{proof}
    Let $\supp (\hphi) \subset (-\sigma, \sigma)$ with $\sigma < 1/(n-a)$. Using Lemma \ref{lem:trivial bound} and the PNT (with partial summation), the sum over $p_1, \ldots, p_{n-n'}$ in \eqref{eq:E with omega} is $\ll  N^{-3/2 + \epsilon} \left( q_1^{m_1} \cdots q_\omega^{ m_\omega} \right)^{1/2} R^{\sigma(n-n')}$. 
    For $1 \le j \le \omega$, we have $n_j - m_j \ge 2$ as $m_j < n_j$ and $m_j \equiv n_j \pmod{2}$.   Using the PNT again for the sums over $q_1, \ldots, q_\omega$ in \eqref{eq:E with omega}, we have $ E(\vec n, \vec m) \ll N^{-1 + \epsilon} R^{\sigma (n-n')} $. This is $\ONe$ if $n' \ge a$.
\end{proof}

%since $n_j > m_j$ and $n_j \equiv m_j \, (\bmod\, 2) $, 

%applying \eqref{proptof} to \eqref{eq:E with omega} gives

We are now in a position to apply the Petersson formula (Lemma \ref{lemcomplsum}) to \eqref{eq:E with omega}. We assume GRH for $L(s,f)$ so that the complementary sum $\Delta_{k,N}^\infty$ does not contribute (see Remark \ref{rem: complementary sum}). Since $|H^{*}_k(N)| \sim N(k-1)/12$ from \eqref{eq:number of terms in hkn}, it follows that 
\begin{align}\label{eq:E Petersson}
E(\vec{n}, \vec{m}) &\ = \ \frac{2^{n+1} \pi}{\sqrt{N}} \qsumn{\omega} \prod_{j=1}^{\omega} \left( \testfn{q_j}  \testcompone{q_j} \right)^{n_j}\psumn{n-n'} \prod_{i=1}^{n-n'} \testfn{p_i} \testcompone{p_i}\nonumber\\
&\hspace{1cm}\times  \msum \, \frac{1}{m} \, \sum_{b=1}^\infty \, \frac{S(m^2, NQ; Nb)}{b}\Jk \left( \frac{4 \pi m \sqrt{Q}}{b \sqrt{N}} \right) \, + \, \ONe,
\end{align}
where $Q$ was defined in (\ref{Qdef}) and we set $Y = N^\epsilon$. Also, recall that $n' = n_1 + \cdots + n_\omega$.

We impose restrictions to the $b$-sum of \eqref{eq:E Petersson} using the following two lemmas. We will also make use of the following bounds from Section \ref{sec:preliminaries section}: 
\begin{align}\label{coarsebdd}
    S(m^2, NQ; Nb) \, \ll_{\epsilon} \, m^2 b^{1/2} (bN)^\epsilon, \hspace{10pt} \Jk(x) \, \ll_{k} \,  x.
\end{align}
\begin{lem} \label{lem: b,N coprime}
    If $\supp(\hat{\phi}) \subseteq \left( -\frac{5}{2(n-n')}, \frac{5}{2(n-n')}\right)$, then the contribution from the terms in \eqref{eq:E Petersson} with $(b,N)>1$  is $\ONe$.
\end{lem}
\begin{proof}
This is a refinement of \cite[Lem. 4.4]{HM}.     Let $\supp (\hphi) \subset (-\sigma, \sigma)$. If the $b$-sum of \eqref{eq:E Petersson} is restricted to $(b, N)>1$, i.e.,  $b = cN$ for some $ c \ge 1$, then such a sum is $\ll N^{-2 + \epsilon} \sqrt{Q}$ using (\ref{coarsebdd}). Combining  this with the PNT and the fact that $n_j - m_j \ge 2$,  we find that the contribution to \eqref{eq:E Petersson} from terms with $(b,N)>1$ is  $\ll N^{-5/2 + \epsilon}R^{(n-n')\sigma}$, which is negligible when $\sigma < \frac{5}{2(n-n')}$ .
\end{proof}

\begin{lem} \label{lem: b > N2006}
    If $\supp(\hat{\phi}) \subset \left( -\frac{1000}{n-n'}, \frac{1000}{n-n'}\right)$, then the contribution  from the terms in \eqref{eq:E Petersson}  with $b \geq N^{2022}$ is $O(N^{-12})$.
\end{lem}

\begin{proof}
This is a refinement of \cite[Lem. 4.5]{HM}. Let $\supp (\hphi) \subset (-\sigma, \sigma)$.  If the $b$-sum of \eqref{eq:E Petersson} is restricted to  $b \ge N^{2022}$, then (\ref{coarsebdd}) implies such a sum is 
    \begin{equation}\label{eq: b > N2006 bsum bound}
    \sum_{b \ge N^{2022}} \frac{S(m^2, NQ; Nb)}{b} \Jk \left( \frac{4\pi m}{b \sqrt{N}} \right)  \ \ll \  N^{-1/2 + \epsilon} \sum_{b \ge N^{2022}} b^{-3/2 + \epsilon} \ \ll \ b^{-1011 -1/2 + \epsilon}. 
    \end{equation}
When $\sigma < \frac{1000}{n-n'}$,  the contribution of \eqref{eq:E Petersson} with $b\ge N^{2022}$ is  $\ll N^{-1012 + \epsilon}R^{(n-n')\sigma} \ll N^{-12}$ using 
 \eqref{eq: b > N2006 bsum bound}, the PNT, and the fact that $n_j - m_j \ge 2$.  This completes the proof. 

\end{proof}

\subsection{Expanding the Kloosterman sums}

Suppose $\supp(\hphi) \subset \twoovern$ for $n\ge 2$. In particular, the assumptions of Lemma \ref{lem: b,N coprime} and \ref{lem: b > N2006} are satisfied and we may impose the relevant restrictions on the $b$-sum of \eqref{eq:E Petersson}. Also, because $R = k^2N$ and $\supp(\hphi)  \subset (-1, 1)$, the conditions $q_j \nmid N$, $p_i \nmid N$ in \eqref{eq:E Petersson} are automatically satisfied provided the primes $N$  are \emph{sufficiently large}, and thus they will be dropped subsequently. Now,  Lemma \ref{lem:genkloos2} allows us to convert the Kloosterman sums in \eqref{eq:E Petersson}  into Gauss sums. We thus obtain
\begin{align}
E&(\vec{n}, \vec{m}) \ =\  -\frac{2^{n+1} \pi}{\sqrt{N}} \qsumsimp{\omega} \prod_{j=1}^{\omega} \left( \testfn{q_j}  \testcompone{q_j} \right)^{n_j} \psumsimp{n-n'} \prod_{i=1}^{n-n'} \testfn{p_i} \testcompone{p_i} \nonumber \\
&\times  \msum \frac{1}{m} \bsum{N} \frac{1}{b\varphi(b)} \sum_{\chi (b)} G_\chi (m^2) G_\chi((Q, b^\infty)) \overline{\chi} \left(\frac{Q}{(Q, b^\infty)}\right) \chi(N) \Jk \left( \frac{4 \pi m \sqrt{Q}}{b \sqrt{N}} \right) + \ONe. \label{eq:E Gauss}
\end{align}
Denote by $E(\vec{n}, \vec{m})|_{\chi\neq \chi_{0}}$ the expression \eqref{eq:E Gauss} but with an extra restriction $\chi \neq \chi_{0}$ in the sum over $\chi (b)$.  The following lemma shows that $E(\vec{n}, \vec{m})|_{\chi\neq \chi_{0}}$ contributes negligibly as $N\to \infty$.
The shape of the expansion \eqref{eqn KlGaexp} allows us to capture cancellations in two different ways: first, from the sum over $\chi (b)$ via 
\begin{equation}\label{eqn Gaussorth}
    \frac{1}{\varphi(b)}\,  \sum_{\chi (b)} \ \left|G_\chi(x) G_\chi(y) \right| \ \ll \ b,
\end{equation}
and second, the character sums over primes (see Lemma \ref{lem:expsum}). The former follows simply from Cauchy-Schwarz's inequality and the orthogonality of characters, and the uniformity of \eqref{eqn Gaussorth} in $x,y$ is important.  For the latter,   we crucially make use of GRH for the Dirichlet $L$-functions,  and of course the restriction $\supp(\hphi) \subset \twoovern$.  Additionally, the following proposition corrects an error made in  \cite[Lem. 4.7]{HM}. 

\begin{lem}\label{lem:char}
    Assume GRH for the Dirichlet $L$-functions. If $\supp(\hphi) \subset \twoovern$, then
    \begin{equation}
        E(\vec{n}, \vec{m})|_{\chi\neq \chi_{0}} = \ONe.
    \end{equation}
\end{lem}
\begin{proof}

Upon rearranging the sums and taking absolute values, we find $E(\vec{n}, \vec{m})|_{\chi\neq \chi_{0}}$ is bounded by
\begin{align}
\max_{0\le \alpha\le n-n'} \,   \frac{1}{\sqrt{N}} \sum_{\substack{b<N^{2022}\\ (b,N)=1}}  & \, \frac{1}{b} \, \qsumsimp{\omega} \prod_{j=1}^{\omega} \left( \left|\testfn{q_j}  \right|\testcompone{q_j} \right)^{n_j} \sum_{\substack{ p_{1}\mid b, \, \ldots, \, p_{\alpha} \mid b }} \, \prod_{i=1}^{\alpha} \,  \left|\testfn{p_i} \right|\testcompone{p_i}\nonumber \\
&\times  \msum \, \frac{1}{m}\, \cdot \,  \frac{1}{\varphi(b)} \,    \sum_{\substack{\chi (b) \\ \chi \ne \chi_0}}\,   \left|G_\chi (m^2) G_\chi\left( (Q, b^\infty)\right)\right| \nonumber \\
& \times  \left| \, \sum_{\substack{ p_{\alpha+1}\nmid b,\,  \ldots, \, p_{n-n'}\nmid b }}  \, \prod_{\ell =\alpha+1}^{n-n'} \,  \testfn{p_\ell} \frac{\overline{\chi}(p_\ell)\log p_\ell }{\sqrt{p_\ell} \log R} \Jk \left( \frac{4 \pi m \sqrt{Q}}{b \sqrt{N}} \right)\right|.  \label{eq:E Gauss char subterm}
\end{align}
We estimate the sum over $p_{\alpha+1}, \ldots, p_{n-n'}$ in \eqref{eq:E Gauss char subterm} with Lemma \ref{lem:expsum}.  To separate variables, we plug in \eqref{eq: Jk inverse Mellin} with $s = -1 + \epsilon+it$. Interchanging the order of sums and integrals and taking absolute values, we have that  
\begin{align}\label{eq: T factored}
    &\sum_{\substack{ p_{\alpha+1}\nmid b,\,  \ldots, \, p_{n-n'}\nmid b }}  \, \prod_{\ell =\alpha+1}^{n-n'} \,  \testfn{p_\ell} \frac{\overline{\chi}(p_\ell)\log p_\ell }{\sqrt{p_\ell} \log R} \Jk \left( \frac{4 \pi m \sqrt{Q}}{b \sqrt{N}} \right) \nn
    &\hs{1} \ll
     \intii \left( \frac{4 \pi m \sqrt{Q'}}{b \sqrt{N}} \right)^{1 - \epsilon}  \left|G_{k-1}(1 - \epsilon + it)\right|  \left|\sum_{p\,\nmid\, b} \,  \testfn{p} \frac{ \overline{\chi}(p)\log p}{p^{(\epsilon + it)/2} \log R} \right|^{n-n' - \alpha} \d t,
\end{align}
where 
\begin{align}\label{eqn Q'expre}
    Q' \ := \   \prod_{j=1}^{\omega} \, q_j^{m_j} \,\cdot \, \prod_{i=1}^{\alpha} \, p_{i}
\end{align}
and $G_{k-1}(\,\cdot\, )$ was given by \eqref{eq: Jk Mellin}.
 Suppose $\supp(\hphi) \subset (-\sigma, \sigma)$.    By  Lemma \ref{lem:expsum} (with partial summation) and the fact that
 \begin{align}\label{eqn simpdivbd}
     \sum_{p\, \mid\, b} \, 
     \frac{ \log p}{ \log R} \  \le \ \frac{\log b}{\log R},
 \end{align}
 we find 
\begin{equation}\label{eqn strongbdd}
    \sum_{p\, \nmid\, b} \testfn{p} \frac{ \overline{\chi}(p)\log p}{p^{(\epsilon + it)/2} \log R} \ \ll \ R^{\sigma/2} (Rbt)^{\epsilon}.
\end{equation}
By the Stirling  formula $ |\Gamma(x+iy)|  \asymp_{x}  (1+|y|)^{x-1/2} e^{-\frac{\pi}{2}|y|}$,
which holds for any $x\in \R- \{0,-1,-2,\ldots\}$ and $y\in \R$ (see \cite[eq. 5.113]{iwaniec_kowalski_2004}), we have 
\begin{align}\label{eqn stircorr}
    |G_{k-1}(1- \epsilon + it)| \ \ll_{k,\epsilon} \  (1 + |t|)^{-2 +\epsilon}
\end{align}
for any $t\in\R$.  Applying \eqref{eqn strongbdd} and \eqref{eqn stircorr} to  \eqref{eq: T factored}, we have
\begin{equation}\label{eq: T bound}
    \sum_{\substack{ p_{\alpha+1}\nmid b,\,  \ldots, \, p_{n-n'}\nmid b }}  \, \prod_{\ell =\alpha+1}^{n-n'} \,  \testfn{p_\ell} \frac{\overline{\chi}(p_\ell)\log p_\ell }{\sqrt{p_\ell} \log R} \Jk \left( \frac{4 \pi m \sqrt{Q}}{b \sqrt{N}} \right) \ll \left(\frac{m \sqrt{Q'}}{b \sqrt{N}}\right)^{1- \epsilon} R^{\sigma (n-n') /2} (Rb)^{\epsilon}.
\end{equation}
From \eqref{eq:E Gauss char subterm}, \eqref{eqn Q'expre}, and \eqref{eq: T bound}, we have that $ E(\vec{n}, \vec{m})|_{\chi\neq \chi_{0}}  $ is bounded by
\begin{align}
\max_{0\le \alpha\le n-n'} \   \frac{R^{\epsilon + \sigma (n-n'-\alpha) /2} }{N} \sum_{\substack{b<N^{2022}}}  & \, \frac{1}{b^2} \, \sum_{\substack{q_{1}, \ldots, q_{\omega} < R^{\sigma}}} \prod_{j=1}^{\omega} \left( \testcompone{q_j} \right)^{n_j} q_{j}^{m_{j}/2} \, \sum_{\substack{ p_{1}\mid b, \, \ldots, \, p_{\alpha} \mid b }} \, \prod_{i=1}^{\alpha} \,  \,  \frac{\log p_{i}}{\log R} \nonumber \\
&\times  \msum \, m^{-\epsilon}\, \cdot \,  \frac{1}{\varphi(b)} \,    \sum_{\substack{\chi (b) \\ \chi \ne \chi_0}}\,   \left|G_\chi (m^2) G_\chi\left( (Q, b^{\infty})\right)\right|  
\end{align}
Apply  \eqref{eqn Gaussorth}, $\sum\limits_{m\le N^{\epsilon}}m ^{-\epsilon} \ll R^{\epsilon}$, and \eqref{eqn simpdivbd} in sequence, the expression above is further bounded by
\begin{align}
  \max_{0\le \alpha\le n-n'} \   \frac{R^{\epsilon + \sigma (n-n'-\alpha) /2} }{N}  &  \,   \sum_{\substack{q_{1}, \ldots, q_{\omega} < R^{\sigma}}} \prod_{j=1}^{\omega} \left( \testcompone{q_j} \right)^{n_j} q_{j}^{m_{j}/2} \, \sum_{\substack{b<N^{2022}}}  \, \frac{1}{b} \left(\frac{\log b}{\log R}\right)^{n-n'}.
\end{align}
The contribution of the $b$-sum is $O(R^{\epsilon})$. Since $n_{j}-m_{j}\ge 2 $, it follows from the PNT that
\begin{align}
  E(\vec{n}, \vec{m})|_{\chi\neq \chi_{0}}  \ \ll \ N^{\epsilon} N^{\sigma(n-n')/2 -1},
\end{align}
which is negligible if $\sigma < 2/(n-n')$.  The result follows.
\end{proof}

We apply Lemma~\ref{lem:char} to \eqref{eq:E Gauss}. This leaves only the contribution from $\chi_0 \, (\bmod\, b)$  for each $b< N^{2022}$ and $(b,N)=1$. Note that $G_{\chi_0}(x) = R(x, b)$ is the Ramanujan sum and $\chi_0(N) = 1$. Hence, we have under GRH that
\begin{align}\label{eq:E Ram}
E(\vec{n}, \vec{m}) \ &\defeq\ -\frac{2^{n+1} \pi}{\sqrt{N}} \qsumsimp{\omega} \prod_{j=1}^{\omega} \left( \testfn{q_j}  \testcompone{q_j} \right)^{n_j} \psumsimp{n-n'} \prod_{i=1}^{n-n'} \testfn{p_i} \testcompone{p_i}\nn 
& \hs{1} \times  \msum \frac{1}{m} \bsum{N} \frac{R(m^2, b) R((Q, b^\infty), b) }{b\varphi(b)} \chi_0 \left(\frac{Q}{(Q, b^\infty)}\right) \Jk \left( \frac{4 \pi m \sqrt{Q}}{b \sqrt{N}} \right), \nn
& \hs{1} + \ONe,
\end{align}
provided $\supp(\hphi) \subset \twoovern$.

We complete the calculation of $E(\vec{n}, \vec{m})$ with the following two propositions, which we prove in Section \ref{sec:main term}.

\begin{prop}\label{lem:nonmaintermcalc}
Let $E(\vec n, \vec m)$ be as in \eqref{eq:E with omega} and suppose that there exists some $1 \le j \le \omega$ for which $n_j + m_j > 2$. Under the same assumptions of Theorem \ref{thm: KS density main}, if $\supp (\hphi) \subset \twoovern$ with $n\ge 2$, then $\Enm = O(1/\log N)$. 
\end{prop}
Proposition \ref{lem:nonmaintermcalc} shows that many of the terms in the expansion \eqref{eq: S2n E} do not contribute in the limit. We complete the proof of Proposition \ref{lem:nonmaintermcalc} in Section~\ref{sec: vanishing}. 

\begin{prop}\label{lem:maintermcalc}
Let $E(\vec n, \vec m)$ be as in \eqref{eq:E with omega} and suppose $\omega = n' = 0$. Under the same assumptions of Theorem \ref{thm: KS density main}, if $\supp (\hphi) \subset \oneoverna $ with $1\le a \le n/2$, then
\begin{align}
    \Enm 
    \ &=\  (-1)^{n+1} \mathcal{R}(n, a; \phi) + \converrorterm , \label{eq:maintermcalc}
\end{align}
where $R(n, a; \phi)$ was defined  in \eqref{eq: R def}.

\end{prop}
We complete the proof of Proposition \ref{lem:maintermcalc} in Section \ref{sec: main term combo}.
\subsection{Concluding the proof of Theorem \ref{thm:lfnsn-2}}\label{sec:combining}

We will now demonstrate how to use Propositions \ref{lem:nonmaintermcalc} and \ref{lem:maintermcalc} to evaluate $S_2^{(n)}(N)$.

\begin{prop}\label{lem: S2n eval}
    Let $S(n, a; \phi)$ be defined in \eqref{eq: Snaphi def}.  Under the same assumptions of Theorem \ref{thm: KS density main}, if $\supp (\hphi) \subset \oneoverna $ with $1\le a\le n/2$, then
    \begin{equation}\label{eq: S2n final}
     S_2^{(n)}(N) \ =\ (-1)^{n+1} S(n, a; \phi) + \converrorterm.
 \end{equation}
\end{prop}
\begin{proof}
 By Lemma \ref{lem:coeffs are 0} and Proposition~\ref{lem:nonmaintermcalc}, $S_2^{(n)}(N)$ is a sum of terms of the form $E(\vec n, \vec m)$ with $n_j + m_j \le 2$ for each $j$  up to  an error of $O(1/\log N)$. Since $n_j \equiv m_j \pmod 2$, then $n_j = m_j = 1$ or $n_j = 2$ and $m_j = 0$ for each $j$. Let $E_{\ell}(N)$ denote the term $\Enm$ in which $n_j = 2$ and $m_j = 0$ for exactly $\ell$ values of $j$. We have that
 \begin{align}
E_\ell(N) &\ =\  \qsumn{\ell} \prod_{j=1}^{\ell} \left( \testfn{q_j}^2  \testcomp{q_j}^2 \right) \nonumber \\
&\hspace{1cm}\times \i^k \sqrt{N} \psumn{n-2\ell} \prod_{i=1}^{n-2\ell} \testfn{p_i} \testcomp{p_i}\left< \lambda_f(Np_1\cdots p_{n-2\ell})\right>_\ast . \label{eq: El expanded} \end{align}
By Lemma~\ref{lem: n' big}, the contribution to $S_2^{(n)}(N)$ from $E_{\ell}(N)$ with $\ell \ge a/2$ is negligible. We thus have 
 \begin{equation}\label{eq: S2n El}
     S_2^{(n)}(N) \ =\  \sum_{\ell = 0}^{\lfloor \frac{a-1}{2} \rfloor} \frac{n!}{2^{\ell} (n - 2\ell)! \ell! } E_\ell(N) \ + \  O(1/\log N).
 \end{equation}
 The combinatorial factor $\frac{n!}{2^{\ell} (n - 2\ell)! \ell! }$ arises from choosing the indices of the primes for which $n_j = m_j = 1$, or $n_j = 2$ and $m_j = 0$. We choose the primes for which $n_j = m_j = 1$ in $\binom{n}{2\ell}$ ways, and put the remaining primes into pairs in $(2\ell-1)!! = (2\ell)!/(\ell! 2^{\ell})$ ways. Multiplying and simplifying gives the desired combinatorial coefficient. By Proposition \ref{lem:maintermcalc}, we have that
 \begin{align}\label{eq: maintermcalc applied}
     &\i^k \sqrt{N} \psumn{n-2\ell} \prod_{i=1}^{n-2\ell} \testfn{p_i} \testcomp{p_i}\left< \lambda_f(Np_1\cdots p_{n-2\ell})\right>_\ast\nn
     &\hs{1}= (-1)^{n+1} \mathcal{R}(n - 2\ell, a - 2\ell; \phi) + \converrorterm.
 \end{align}
Applying \eqref{eq: maintermcalc applied} to \eqref{eq: El expanded} and factoring the sums over $q_j$ gives
\begin{equation}\label{eq: El factored}
    E_\ell(N) \ =\  (-1)^{n+1} \left[\mathcal{R}(n - 2\ell, a- 2\ell; \phi) + \converrorterm \right] \left[ \sum_{q \notdiv N  }   \testfn{q}^2  \frac{4 \log^2 q}{q \log^2 R}  \right]^{\ell}.
\end{equation}
A standard partial summation argument gives
\begin{equation}\label{eq: primes to sigma}
    \sum_{q \notdiv N  } \testfn{q}^2  \frac{4 \log^2 q}{q \log^2 R} \ \sim \  2 \intii |y| \hphi(y)^2 dy, 
\end{equation}
which is precisely the quantity $\sigma^2_\phi$ given by \eqref{eq:fullNT_var}. Applying \eqref{eq: primes to sigma} to \eqref{eq: El factored},  we have that 
\begin{equation}
E_\ell(N) \  =\  (-1)^{n+1} \left( \sigma^2_\phi \right)^\ell \mathcal{R}(n - 2\ell, a- 2\ell; \phi) + \converrorterm .
\end{equation}
The proposition follows readily upon applying this to \eqref{eq: S2n El}
 and comparing with \eqref{eq: Snaphi def}.  
 \end{proof}

\begin{proof}[Proof of Theorem \ref{thm:lfnsn-2}]
The proof follows immediately from
Proposition \ref{lem: S2n eval},  \eqref{eq:nth_mmt_in_terms_of_S}, and \eqref{eq:S1nresult}. 
\end{proof}

%%%%%%%%%%%%%%%%%%%%%%%%%%%%%%%%%%%%%%%%%%%%%%%%%%%%%%%%%%%%%%%%%%%%%%%%%%%%%%%%%%%%%%%%%%%%%%%%%%%%%%%%%%%%%%%%%%%%%%%%%%%%%%%%%%%%%
%%%%%%%%%%%%%%%%%%%%%%%%%%%%%%%%%%%%%%%%%%%%%%%%%%%%%%%%%%%%%%%%%%%%%%%%%%%%%%%%%%%%%%%%%%%%%%%%%%%%%%%%%%%%%%%%%%%%%%%%%%%%%%%%%%%%%
%%%%%%%%%%%%%%%%%%%%%%%%%%%%%%%%%%%%%%%%%%%%%%%%%%%%%%%%%%%%%%%%%%%%%%%%%%%%%%%%%%%%%%%%%%%%%%%%%%%%%%%%%%%%%%%%%%%%%%%%%%%%%%%%%%%%%
%%%%%%%%%%%%%%%%%%%%%%%%%%%%%%%%%%%%%%%%%%%%%%%%%%%%%%%%%%%%%%%%%%%%%%%%%%%%%%%%%%%%%%%%%%%%%%%%%%%%%%%%%%%%%%%%%%%%%%%%%%%%%%%%%%%%%

%%%%%%%%%%%%%%%%%%%%%%%%%%%%%%%%%%%%%%%%%%%%%%%%%%%%%%%%%%%%%%%%%%%%%%%%%%%%%%%%%%%%%%%%%%%%%%%%%%%%%%%%%%%%%%%%%%%%%%%%%%%%%%%%%%%%%
%%%%%%%%%%%%%%%%%%%%%%%%%%%%%%%%%%%%%%%%%%%%%%%%%%%%%%%%%%%%%%%%%%%%%%%%%%%%%%%%%%%%%%%%%%%%%%%%%%%%%%%%%%%%%%%%%%%%%%%%%%%%%%%%%%%%%
%%%%%%%%%%%%%%%%%%%%%%%%%%%%%%%%%%%%%%%%%%%%%%%%%%%%%%%%%%%%%%%%%%%%%%%%%%%%%%%%%%%%%%%%%%%%%%%%%%%%%%%%%%%%%%%%%%%%%%%%%%%%%%%%%%%%%
%%%%%%%%%%%%%%%%%%%%%%%%%%%%%%%%%%%%%%%%%%%%%%%%%%%%%%%%%%%%%%%%%%%%%%%%%%%%%%%%%%%%%%%%%%%%%%%%%%%%%%%%%%%%%%%%%%%%%%%%%%%%%%%%%%%%%

\section{Proof of Propositions \ref{lem:nonmaintermcalc} \& \ref{lem:maintermcalc} }\label{sec:main term}  
%%%%%%%%%%%%%%%%%%%%%%%%%%%%%%%%%%%%%%%%%%%%%%%%%%%%%

Throughout Section \ref{sec:main term}, we assume $\supp(\hphi) \subset \oneoverna$ with $1\le a\le n/2$ and we further analyze the expression \eqref{eq:E Ram} for $E(\vec{n}, \vec{m})$.

In Section~\ref{sec: sums to integrals},  we rewrite the sums over primes in \eqref{eq:E Ram} into a more analytically tractable expression. The key quantity to be considered is 
\be\label{eq:MTBalpha}
    B(\alpha) \ \defeq \ 
    \rootn \sum_{p_{\alpha+1}, \ldots, p_n} J_{k-1} \left(\frac{4\pi m \sqrt{c \prod\limits_{i=\alpha+1}^n \, p_{i}}}{b \sqrt{N }} \right)\, \prod_{i=\alpha+1}^n \hat{\phi}\left(\frac{\log p_i}{\log R} \right) \frac{\chi_0(p_i)\log p_i}{p_i^{1/2} \log R},
\ee
where the dependencies on $N, m, c,b$ are conveniently suppressed in the notation $B(\alpha) $.  We also set  $\Phi_{r}(x) :=\phi(x)^r$ for $r\ge 0$.  The main result is the following.
\begin{prop}\label{lem:sums to integrals}
 Suppose $\supp(\hphi) \subset \oneoverna$ with $1\le a \le n/2$. Under RH for $\zeta(s)$, we have
    \begin{align}
        B(\alpha)\ &= \ \frac{b}{2\pi m \sqrt{c}  }\sum_{\delta = 0}^{a - \alpha - 1} \binom{n - \alpha}{\delta} \sum_{i=0}^{a - \alpha - \delta - 1} (-1)^i \binom{n - \alpha - \delta}{i} \psumsimp{\delta} \prod_{j=1}^{\delta} \hat{\phi}\left(\frac{ \log p_j}{\log R} \right) \frac{\chi_0(p_j) \log p_j}{p_j \log R}  \nonumber \\
        &\hspace{1cm} \times\int_{x=0}^{\infty} J_{k-1}(x) \widehat{\Phi_{n-\alpha- \delta}} \left( \frac{2\log\left(bx\sqrt{N/(c\prod\limits_{j=1}^{\delta} \, p_{j})} /4\pi m\right)}{\log R} \right) \frac{dx}{\log R} \ + \ \ONe.  \label{eq:BtoFexpanded}
    \end{align}
    where the implicit constant does not depend on $N,m,c,b,\alpha$. 
\end{prop}

The evaluation of $B(\alpha)$ requires a delicate inclusion-exclusion argument to convert the sums over primes of \eqref{eq:MTBalpha} into  sums over integers which appears on the left side of \eqref{eq:sumtoint}. This is followed by Lemma \ref{lem:Triant3.6general} which expresses the sums over integers into integrals via a standard contour-shifting argument (assuming RH). Then the proof of Proposition \ref{lem:sums to integrals} is completed by further combinatorial simplifications. The final step is crucial for matching the calculations with those from the random matrix theory (see Section \ref{sec:rmt calc}).

One must also evaluate the convolution sum of \emph{Ramanujan sums} in \eqref{eq:E Ram}. This task will be carried out in Section \ref{Ramsumco} and the main result is stated as follows. 

\begin{prop}\label{lem:ils7}
    Let $\phi \in \mathcal{S}_{ec}(\R)$. Then as $N\to\infty$, we have 
    \begin{align}
        &\sum_{(b,M) = 1} \frac{R(1,b) R(m^2,b)}{ \varphi(b)} \int_{0}^{\infty} J_{k-1}(y) \hat{\phi} \left( \frac{2 \log(by \sqrt{Q}/4 \pi m)}{\log R} \right) \frac{\d y}{ \log R}\nonumber \\
        = \ & \delta \left(\frac{m}{(m,M^\infty)}, 1\right) \frac{\varphi(M)}{M}\left(-\frac{1}{2}\int_{-\infty}^{\infty}\phi(x)\sin\left(2\pi x\frac{\log(k^2Q/16\pi^2m^2)}{\log R}\right)\frac{\d x}{2\pi x}
        \, + \, \frac{1}{4}\phi(0)\right)\nonumber \\ &  +  O\left(m^{\epsilon} \frac{(\log \log M)^{2}}{(\log R)^{1/2}} \right) \label{eq:ils7}
    \end{align}
    uniformly for $m, M, Q\ge 1$. 
    
\end{prop}

The contents of Sections \ref{sec: sums to integrals} and \ref{Ramsumco} are independent of each other. The proofs of Propositions \ref{lem:nonmaintermcalc} and \ref{lem:maintermcalc} crucially rely on Propositions \ref{lem:sums to integrals} and \ref{lem:ils7}. They are the subjects of Section \ref{sec: vanishing} and Section \ref{sec:maintermlaststep}-\ref{sec: main term combo} respectively. It will be essential to break up the Ramanujan sums judiciously using the multiplicativity \eqref{eqn Ramult} and perform a prime-by-prime analysis with the exact evaluation \eqref{eq:vonsterneck}. In Section~\ref{sec:maintermlaststep}, we managed to find an integral representation for $\Enm$. The proof of Proposition~\ref{lem:maintermcalc} follows from the combinatorial simplification in Section~\ref{sec: main term combo}.

\subsection{Sums over primes: Proof of Proposition \ref{lem:sums to integrals}}\label{sec: sums to integrals}

This subsection is dedicated to proving Proposition \ref{lem:sums to integrals}. We begin by defining the following quantities:
\begin{align}
    C'(\alpha, \beta) &\ \defeq \ \rootn \sum_{p_{\alpha+1}, \ldots, p_n}\sum_{t_{\alpha+1}, \ldots, t_{\alpha+ \beta} = 2}^\infty  J_{k-1} \left(\frac{4\pi m \sqrt{c \prod\limits_{i=\alpha + \beta+1}^n p_{i}}}{b \sqrt{N/ \prod\limits_{j=\alpha+1}^{\alpha + \beta} \, p_{j}^{t_{j}}}} \right) \nonumber \\
    &\hspace{1cm}\times  \prod_{j=\alpha+1}^{\alpha + \beta} \hat{\phi}\left(\frac{t_j \log p_j}{\log R} \right) \frac{\chi_0(p_j) \log p_j}{p_j^{t_j/2} \log R}  \, \prod_{i=\alpha + \beta+1}^n \hat{\phi}\left(\frac{\log p_i}{\log R} \right) \frac{\chi_0(p_i)\log p_i}{p_i^{1/2} \log R},\label{eq:MTCprime}
\end{align}
and
\begin{align}
    C(\alpha, \beta) \ &\defeq \ \rootn \sum_{p_{\alpha+1}, \ldots, p_{\alpha+ \beta}} \, \sum_{t_{\alpha+1}, \ldots, t_{\alpha+ \beta} = 2}^\infty \, \sum_{v_{\alpha+\beta+1}, \ldots, v_n = 1}^\infty  J_{k-1} \left( \frac{4\pi m \sqrt{c \prod\limits_{i=\alpha + \beta+1}^n v_{i} }}{b \sqrt{N/  \prod\limits_{j=\alpha+1}^{\alpha + \beta} \, p_{j}^{t_j} }} \right) \nonumber\\
    &\hspace{1cm}\times \prod_{j=\alpha+1}^{\alpha + \beta} \hat{\phi}\left(\frac{t_j \log p_j}{\log R} \right) \frac{\chi_0(p_j) \log p_j}{p_j^{t_j/2} \log R}  \,  \prod_{i=\alpha + \beta+1}^n \hat{\phi}\left(\frac{\log v_i}{\log R} \right) \frac{\chi_0(v_i)\Lambda (v_i)}{v_i^{1/2} \log R}.\label{eq:MTC}
    \end{align}
    
Observe that $C'(\alpha, \beta)$ is a generalized version of $B(\alpha)$ (see \eqref{eq:MTBalpha}) as  $B(\alpha)= C'(\alpha, 0)$. Also, the expression $C(\alpha, \beta)$ is obtained from  $C'(\alpha,\beta)$ by replacing the sums over primes $p_{j}$'s, where $\alpha+\beta+1\le j \le n$,  by sums over positive integers $v_{j}$'s, as well as the weight $\log p_{j}$ is replaced by  von Mangoldt's function $\Lambda(v_{j})$. We have the following relation between $C'(\alpha, \beta)$ and $C(\alpha, \beta)$.

\begin{propt}\label{propt:CCprime}
We have that
    \be \label{eq:CCprime}
         C'(\alpha, \beta)\ = \ C(\alpha, \beta) - \sum_{i = 1}^{a- 1 -\alpha - \beta} \binom{n-\alpha-\beta}{i} C'(\alpha, \beta + i) \ + \  O(N^{-\epsilon}).
    \ee
\end{propt}

\begin{proof}
     The property follows directly from the definition of von Mangoldt's function and a partitioning argument. The sum over $i$ in \eqref{eq:CCprime} is restricted to  $i\le a - 1 -\alpha -\beta$ because 
     the contribution from  $i > a-1 - \alpha -\beta$  is $\ONe$, which follows readily from  $\Jk(x) \ll 1$. 
\end{proof}

 Property~\ref{propt:CCprime} can be applied repeatedly to obtain a relation between \eqref{eq:MTBalpha} and \eqref{eq:MTC}. 

\begin{propt}\label{propt:BtoC}

We have that
    \be\label{eq:BtoC}
        B(\alpha)\ = \ \sum_{i = 0}^{a-1-\alpha} \binom{n-\alpha}{i} C(\alpha, i) (-1)^i \ + \  O(N^{-\epsilon}).
    \ee
\end{propt}

\begin{proof}
    Define
    \begin{equation}
        B'(\alpha, \eta) \ \defeq \ \sum_{i = 0}^{\eta} \, (-1)^i\binom{n-\alpha}{i} C(\alpha, i)  - \sum_{i = \eta+1}^{a- 1 - \alpha} \binom{n-\alpha}{i} C'(\alpha, i)\sum_{j = 0}^{\eta} (-1)^j \binom{i}{j} + \ONe.
    \end{equation}
    We claim that $B'(\alpha, \eta) = B(\alpha)$ and we proceed by induction on $\eta$. The base case $\eta=0$ holds by Property~\ref{propt:CCprime} and the fact that $B(\alpha) = C'(\alpha, 0)$. For the inductive step,  assume that $B'(\alpha, k) = B(\alpha)$ for some integer $k\ge 0$. This implies %\st{and show that under this assumption, $B'(\alpha, k+1) = B(\alpha)$. }
    \be\label{eq:hypo}
        \begin{split}
            B(\alpha) %&= B'(\alpha, \eta) \\
            \ &= \ \sum_{i = 0}^{k} \binom{n-\alpha}{i} C(\alpha, i) (-1)^i - \sum_{i = k+1}^{a- 1 - \alpha} \binom{n-\alpha}{i} C'(\alpha, i)\sum_{j = 0}^{k} (-1)^j \binom{i}{j} + O(N^{-\epsilon}).
        \end{split}
    \ee
    
    We examine the term where $i = k+1$ and we have
    \be
        -\binom{n-\alpha}{k+1} C'(\alpha, k+1)\sum_{j = 0}^{k} (-1)^j \binom{k+1}{j}\ = \ \binom{n-\alpha}{k+1}(-1)^{k+1} C'(\alpha, k+1).
    \ee
    This follows from the identity  $\sum_{j = 0}^{k+1} (-1)^j \binom{k+1}{j} = 0$, which is an easy consequence of the Binomial Theorem.  It follows from  Property~\ref{propt:CCprime} and re-indexing the sum using the change of variables $\ell = k+1+ j$ that
    \begin{align}
            &\binom{n-\alpha}{k+1}(-1)^{k+1} C'(\alpha, k+1) \nonumber\\
            &= \ \binom{n-\alpha}{k+1} (-1)^{ k + 1} \left[C(\alpha, \eta + 1) -  \sum_{j = 1}^{a -\alpha - k} \binom{n- \alpha - k - 1}{j} C'(\alpha,  k +1 + j)\right] + \ONe \nonumber \\
            &= \ \binom{n - \alpha}{k+1}(-1)^{k+1} C(\alpha, k+1) -  \sum_{j = 1}^{a -\alpha - k}\binom{n-\alpha}{k + 1 + j} \binom{k + 1 + j}{k +1} (-1)^{k+1}  C'(\alpha, k+1 + j) + O(N^{-\epsilon}) \nonumber \\
            &= \ \binom{n-\alpha}{k+1}(-1)^{k+1} C(\alpha, k+1) -  \sum_{\ell = k+2}^{a -\alpha - 1}\binom{n-\alpha}{\ell} \binom{\ell}{k + 1} (-1)^{k+1}  C'(\alpha, \ell) + O(N^{-\epsilon}).
        \end{align}
    Substituting the last equality into \eqref{eq:hypo}, we have  $B(\alpha) = B'(\alpha, k+1)$ and this completes the induction. 
\end{proof}

The innermost sums of \eqref{eq:MTC} can be rewritten as follows. 

\begin{lem}\label{lem:Triant3.6general}
        Under RH for $\zeta(s),$ if $\supp(\hat{\phi}) \subset \left( -\frac{1}{n-a}, \frac{1}{n-a}\right)$, then %\st{ as $N$ tends to infinity,}
    \begin{align} 
        & \sum_{v_1, \ldots, v_{n- \eta}} \left[ \prod_{i=1}^{n- \eta} \testfn{v_i} \testcompvm{v_i} \right] J_{k-1} \left( \frac{ 4\pi m \sqrt {c v_1 \cdots v_{n- \eta}}}{b \sqrt{N}} \right)\nonumber \\
        & = \ \frac{b \sqrt{N} }{2 \pi m \sqrt{c} } \sum_{{\gamma}=0}^{a - \eta - 1} \sum_{j={\gamma}}^{a - \eta - 1} (-1)^{j-{\gamma}} \binom{n-\eta}{j} \binom{j}{{\gamma}}   \sum_{v_1, \ldots, v_\gamma=1}^\infty \left[ \prod_{i=1}^{{\gamma}} \hat{\phi} \left( \frac{\log v_i}{\log  R}\right) \frac{\chi_0 (v_i) \Lambda(v_i)}{v_i \log R} \right] \nonumber \\
        &\hspace{1cm} \times \int_{x=0}^{\infty} J_{k-1} (x) \widehat{\Phi_{n- \eta -{\gamma}}} \left( \frac{2 \log(b x \sqrt{N / (c v_1\cdots v_\gamma )}/ 4 \pi m)}{\log R} \right) \frac{dx}{\log R} \, + \, \ONhalfe, \label{eq:sumtoint}
    \end{align}
    where the implicit constant does not depend on $N, m, c, b$. 
\end{lem}

\begin{proof}
This is a generalization of \cite[Lem. 4.9]{HM}.  Open up the $J$-Bessel function on the left side of \eqref{eq:sumtoint} with the Mellin inversion formula 
  \eqref{eq: Jk inverse Mellin}: 
    \begin{align}\label{eq: U def}
         \frac{1}{2\pi i}\int_{\Re(s) = 1}  \left[ \sum_{v = 1}^\infty \hat{\phi} \left( \frac{\log v}{ \log R} \right) \frac{\chi_0 (v) \Lambda(v) }{v^{(1+s)/2} \log R} \right]^{n-\eta} G_{k-1} (s) \left(\frac{4\pi m\sqrt{c }}{b \sqrt{N}} \right)^{-s}  \ \d s.
    \end{align}
    Under RH, a simple contour-shifting argument (or see \cite[eq. (4.34)]{HM}) gives 
    \begin{align}\label{eq: HM4.34}
\sum_{v=1}^\infty\hat{\phi}\left(\frac{\log v}{\log R}\right)\frac{\chi_0(v)\Lambda(v)}{v^{(1+s)/2}\log R} \ = \ \phi\left(\frac{1-s}{4\pi i}\log R\right) \, + \, \mathcal E(s),
    \end{align}
    where
    \begin{align}
        \mathcal E(s) \ \defeq \ -\frac{1}{2\pi i}\int_{\Re(z) = 3/4} \phi \left( \frac{(2z - 1 - s)\log R}{4\pi i}\right)\frac{L'}{L}(z,\chi_0) \ \d z.
    \end{align}
    As a consequence,  \eqref{eq: U def} becomes
    \begin{align}
  &\frac{1}{2\pi i}\int_{\Re(s) = 1}  \left[ \phi\left(\frac{1-s}{4\pi i}\log R\right) + \mathcal E(s) \right]^{n-\eta} G_{k-1} (s)\left(\frac{4\pi m\sqrt{c }}{b \sqrt{N}} \right)^{-s}  \ \d s\nonumber \\
    &= \ \sum_{\gamma =0}^{n-\eta} \, \binom{n- \eta}{\gamma} \frac{1}{2\pi i} \int_{\Re(s) = 1} \phi\left(\frac{1-s}{4\pi i}\log R\right)^{n -\eta -\gamma} \mathcal E(s)^\gamma G_{k-1} (s)\left(\frac{4\pi m \sqrt{c} }{b \sqrt{N}} \right)^{-s}  \ \d s,
    \end{align}
    where the last line follows from the Binomial Theorem. Now,  \cite[eq. (4.43)]{HM} gives the bound
    \begin{align}
        \frac{1}{2\pi i}&\int_{\Re(s) = 1} \phi \left( \frac{1-s}{4\pi i}\log R \right)^{n- \eta -\gamma} \mathcal E(s)^\gamma G_{k-1} (s)\left(\frac{4\pi m \sqrt{c}}{b \sqrt{N}} \right)^{-s}  \ \d s \  \ll \ N^{(n- \eta - \gamma)\sigma/2 + \epsilon},\label{HMeqn4.43conseq}
    \end{align}
which is in turn $\ONhalfe$ when $\gamma > a - \eta - 1$ and $\sigma < 1/(n-a)$. Hence, \eqref{eq: U def} is equal to
    \begin{equation}
    \begin{split}
    \sum_{\gamma=0}^{a - \eta - 1} \binom{n- \eta}{\gamma}  \int_{\Re(s) = 1} \phi\left(\frac{1-s}{4\pi i}\log R\right)^{n -\eta -\gamma} \mathcal E(s)^\gamma  G_{k-1}(s) \left(\frac{4\pi m \sqrt{c} }{b \sqrt{N}} \right)^{-s}  \ \frac{\d s}{2\pi i}  \, + \,  \ONhalfe.
    \end{split}
    \end{equation}
For integers $\gamma$, $n$ with $0\le \gamma < n$, we define
    \begin{equation}\label{eq: T def}
    T(\gamma, n) \defeq \int_{\Re(s) = 1}  \left[ \sum_{v=1}^\infty\hat{\phi}\left(\frac{\log v}{\log R}\right)\frac{\chi_0(v)\Lambda(v)}{v^{(1+s)/2}\log R}\right]^{\gamma} \phi\left(\frac{1-s}{4\pi i}\log R\right)^{n-\gamma} G_{k-1}(s) \left(\frac{4\pi m \sqrt{c}}{b\sqrt{N}}\right)^{-s}  \ \frac{\d s}{2\pi i}.
    \end{equation}
    It follows that \eqref{eq: U def} is given by
    \begin{align}\label{eq: U final}
      &\sum_{\gamma =0}^{a - \eta - 1}   \, \binom{n- \eta}{\gamma} \sum_{j=0}^{\gamma} \,  (-1)^{j - \gamma} \binom{\gamma}{j} T(j, n- \eta) \ +\  \ONhalfe \nonumber\\
     & \hs{1} = \   \sum_{\gamma =0}^{a - \eta - 1} \sum_{j= \gamma}^{a - \eta - 1} (-1)^{j- \gamma} \binom{n- \eta}{j} \binom{j}{\gamma}   T(\gamma, n- \eta) \ + \  \ONhalfe
    \end{align}
     using the formula for $\mathcal{E}(s)$ in \eqref{eq: HM4.34}.

In \eqref{eq: T def}, applying the formula $\int_{0}^{\infty}\, J_{k-1}(x) x^{s-1} \, dx = G_{k-1}(s)$ with the change of variable $s=1+it$, it follows that 
    \begin{align}
        T(\gamma, n) \ = \ \frac{b\sqrt{N}}{8\pi^2 m \sqrt{c}} \int_{t=-\infty}^{\infty}\, & \phi\left(\frac{-t\log R}{4\pi}\right)^{n-\gamma} \left[ \sum_{v_1, \ldots, v_\gamma = 1 }^\infty  \prod_{i=1}^{\gamma} \hat{\phi} \left(\frac{\log v_i}{\log R}\right)\frac{\chi_0(v_i)\Lambda(v_i)}{v_i^{ it/2 + 1}\log R}\right]\nonumber \\
        & \hspace{10pt} \times \left(\frac{4\pi m \sqrt{c}}{b\sqrt{N}}\right)^{-it} \int_{0}^{\infty} J_{k-1}(x) x^{it} \ \d x \ \d t .
    \end{align}
   Upon rearranging and a change of variables $u = -t \log R/(4\pi)$, we have
    \begin{align}
        T(\gamma, n) \ &= \ \frac{b\sqrt{N}}{2\pi m \sqrt{c}} \sum_{v_1,\ldots,v_\gamma=1}^\infty \left[ \prod_{i=1}^{\gamma}\hat{\phi}\left(\frac{\log v_i}{\log R}\right)\frac{\chi_0(v_i)\Lambda(v_i)}{v_i\log R} \right]\nonumber\\
        & \hs{1} \times \int_{x=0}^{\infty} J_{k-1}(x) \widehat{\Phi_{n-\gamma}}\left(\frac{2\log(bx\sqrt{N/c v_1 \cdots v_\gamma}/4\pi m)}{\log R}\right) \frac{dx}{\log R}. \label{eq: T final}
    \end{align}
Now, \eqref{eq: T final} and \eqref{eq: U final} lead to the desired claim.
\end{proof}

Next we apply Lemma~\ref{lem:Triant3.6general} to $C(\alpha, \beta)$.
\begin{propt}\label{propt:CtoD} We have
    \begin{align}
        \binom{n-\alpha}{\beta}C(\alpha, \beta) \ &= \ \sum_{\gamma = 0}^{a- \alpha - \beta - 1} D(\alpha, \beta, \gamma) \left[\sum_{i = 0}^{a- \alpha -\beta - \gamma - 1} (-1)^{i} \binom{n-\alpha}{\gamma + \beta} \binom{n-\alpha - \beta - \gamma}{i}\binom{\gamma + \beta}{\gamma}  \right] \nonumber\\
        & \hspace{1cm}+ \ONe,
    \end{align}
    where $D(\alpha, \beta, \gamma)$ is defined by
    \begin{align}
     \frac{b}{2\pi m \sqrt{c} } \, \sum_{p_{\alpha + 1}, \ldots, p_{\alpha + \beta + \gamma}} & \, \sum_{t_{\alpha+1}, \ldots, t_{\alpha+\beta} = 2}^\infty \, \sum_{t_{\alpha+\beta + 1}, \ldots, t_{\alpha+\beta + \gamma} = 1}^\infty \, \prod_{j=\alpha+\beta +1}^{\alpha + \beta + \gamma} \hat{\phi} \left( \frac{t_j \log p_j}{\log R} \right) \frac{\chi_0(p_j) \log p_j}{p_j^{t_j} \log R}  \\
        & \times\int_{x=0}^{\infty} J_{k-1}(x) \widehat{\Phi_{n- \alpha- \beta- \gamma}} \left( \frac{2\log(bx\sqrt{N''} /4\pi m)}{\log R}\right) \frac{dx}{\log R}, \label{eq:Dabg} 
\end{align}
and  
\begin{align}
    N'' \ := \  N/(cp_{\alpha+1}^{t_{\alpha+1}} \cdots p_{\alpha + \beta + \gamma}^{t_{\alpha + \beta + \gamma}}). 
\end{align}
\end{propt}

\begin{proof}
    Applying Lemma~\ref{lem:Triant3.6general} to \eqref{eq:MTC}, we have
    \be
       C(\alpha, \beta)\ = \  \sum_{\gamma = 0}^{a- \alpha - \beta - 1} D(\alpha, \beta, \gamma) \left[\sum_{j = \gamma}^{a- \alpha -\beta - 1} (-1)^{j-\gamma} \binom{n-\alpha - \beta}{j}\binom{j}{\gamma}  \right] + \ONe.
    \ee
    The property follows from the re-indexing  $i = j - \gamma$ and simplifying.
\end{proof}

\begin{rem}
    The three multiple sums of \eqref{eq:Dabg} can be interpreted as follows: the first $\alpha$ primes are those that divide $b$; the next two sums involving $\beta$ powers of primes are those left over from converting to sums over integers; and the last sums over $\gamma$ integers are those leftover from applying Lemma~\ref{lem:Triant3.6general}. 
\end{rem}

We apply Property~\ref{propt:CtoD} to the formula for $B(\alpha)$ described in Property~\ref{propt:BtoC}.
\begin{propt}\label{propt:BtoD}
    \be\label{eq:BtoD}
        B(\alpha)\ = \ \sum_{\delta = 0}^{a - \alpha - 1} \binom{n - \alpha}{\delta} \sum_{i=0}^{a - \alpha - \delta - 1} (-1)^i \binom{n - \alpha - \delta}{i} \sum_{\gamma  = 0}^\delta (-1)^{\delta - \gamma} \binom{\delta}{\gamma} D(\alpha, \delta - \gamma, \gamma) + \ONe.   
    \ee
\end{propt}
\begin{proof}
    The property from applying Property~\ref{propt:CtoD} to \eqref{eq:BtoC} and collecting terms with $\delta = \beta + \gamma$.
\end{proof}

We now eliminate the sums over prime powers in \eqref{eq:Dabg}. Define % \st{First we define the following term} 
\begin{align}\label{eq:Fad}
    G(\alpha, \delta) \ &\defeq \ \frac{b}{2\pi m \sqrt{c} }\sum_{p_{1}, \ldots, p_{ \delta}}   \int_{x=0}^{\infty} J_{k-1}(x) \widehat{\Phi_{n-\alpha- \delta}} \left( \frac{2\log(bx\sqrt{N/(cp_{1} \cdots p_{\delta})}  /4\pi m)}{\log R} \right) \frac{dx}{\log R}\nonumber \\
    &\times \prod_{j=\alpha+1}^{\alpha + \delta} \hat{\phi}\left(\frac{ \log p_j}{\log R} \right) \frac{\chi_0(p_j) \log p_j}{p_j \log R},
\end{align}  
which is obtained from \eqref{eq:Dabg} upon specializing all $t_{i}$'s to be $1$ and re-indexing.  The expressions $D(\alpha,\beta,\gamma)$ and $G(\alpha,\delta)$ (with $\delta = \beta + \gamma$ as above) satisfy the following recursion: 

\begin{propt}\label{propt:DtoF} We have
    \be\label{eq:DtoF}
        G(\alpha, \delta)\ = \ \sum_{\gamma  = 0}^\delta (-1)^{\delta - \gamma} \binom{\delta}{\gamma} D(\alpha, \delta - \gamma, \gamma).
    \ee
\end{propt}

This can be deduced from a more general result:

\begin{lem}\label{lem:symmetricfunction}
        Let $f(t_1, \ldots, t_n)$ be a symmetric function which takes as an input a finite sequence $t_1, \ldots, t_n$ of arbitrary length and define the following transform $\mathcal{T}(i, j)$ on $f$:
        \be\label{eq:transform}
            \mathcal{T}(i, j)(f) \ \defeq \ \sum_{t_1, \ldots, t_i = 2}^\infty \sum_{s_1, \ldots , s_j = 1}^\infty f(t_1, \ldots, t_i, s_1, \ldots,  s_j).
        \ee
        Then
        \begin{equation}\label{eq:symmetricfunction}
            \sum_{i=0}^n (-1)^i \binom{n}{i} \mathcal{T}(i, n-i)(f)\ =\ f([1]^n)
        \end{equation}
        where $[1]^n$ denotes the sequence of 1's repeated $n$ times. 
\end{lem}
\begin{proof}
    We proceed by induction on $n$. The base case $n=1$ holds immediately.
    Assume the result holds up to $n$ and define a new function $g(t_1, \ldots, t_n) = f(t_1, \ldots, t_n, 1)$. Then
    \begin{align}
        f([1]^{n+1})\ &= \ g([1]^{n}) \nonumber \\
        &=\ \sum_{i=0}^n (-1)^i \binom{n}{i} \sum_{t_1, \ldots,  t_i=2}^\infty \sum_{s_1, \ldots, s_{n-i} = 1}^\infty f(t_1, \ldots, t_i, s_1, \ldots,  s_{n-i}, 1)\nonumber \\
        &=\ \sum_{i=0}^n  \binom{n}{i}\left[(-1)^i \mathcal{T}(i, n+1-i)(f) +(-1)^{i+1} \mathcal{T}(i+1, n+1-(i+1))(f) \right] \nonumber \\
        &=\ \sum_{i=0}^{n+1} (-1)^i \mathcal{T}(i, n+1-i)(f) \binom{n+1}{i}. 
        \end{align}
    This completes the induction. 
\end{proof}

\begin{proof}[Proof of Property \ref{propt:DtoF}]
    A special case of Lemma \ref{lem:symmetricfunction}: take $f([1]^\delta) = G(\alpha, \delta)$ and $\mathcal{T}(i, j) = D(\alpha, i, j)$.
\end{proof}

Applying Property~\ref{propt:DtoF} to \eqref{eq:BtoD} gives the following relation between $B(\alpha)$ and $G(\alpha,\delta)$.

\begin{propt}\label{propt:BtoF} We have
    \begin{equation}\label{eq:BtoF}
    B(\alpha)\ = \ \sum_{\delta = 0}^{a - \alpha - 1} \binom{n - \alpha}{\delta} \sum_{i=0}^{a - \alpha - \delta - 1} (-1)^i \binom{n - \alpha - \delta}{i} G(\alpha, \delta).   
    \end{equation}      
\end{propt}

\begin{proof}[Proof of Proposition~\ref{lem:sums to integrals}]
  Substitute  \eqref{eq:Fad} into  \eqref{eq:BtoF}, the result follows.  
  
\end{proof}

%%%%%%%%%%%%%%%%%%%%%%%%%%%%%%%%%%%%%%%%%%%%%%%%%%%%%%%%%%%%%%%%%%%%%%%%%%%%%%%%%%%%%%%%%%%%%%%%%%%%%%%%%%%%%%%%%%%%%%%%%%%%%%%%%%%%%%%%%%%%%%%%%%%%%%%%%%%%%%%%%%%%%%%%%%%%%%%%%%%%%%%%%%%%%%%%%%%%%%%%%%%%%%%%%%%%%%%%%%%%%%%%%%%%%%%%%%%%%%%%%%%%%%%%%%%%%%%%%%%%%%%%%%%%%%%%%%%%%%%%%%%%%%%%%%%%%%%%%%%%%%%%%%%%%%%%%%%%%%%%%%%%%%%%%%%%%%%%%%%%%%%%%%%%%%%%%%%%%%%%%%%%%%%%%%%%%%%%%%%%%%%%%%%%%%%%%%%%%%%%%%%%%%%%%%%%%%%%%%%%%%%%%%%%%%%%%%%%%%%%%%%%%%%%%%%%%%%%%%%%%%%%%%%%%%%%%%%%%%%%%%%%%%%%%%%%%%%%%%%%%%%%%%%%%%%%%%%%%%%%%%%%%%%%%%%%%%%%%%%%%%%%%%%%%%%%%%%%%%%%%%%%%%%%%%%%%%%%

\subsection{A convolution sum of Ramanujan sums: Proof of Proposition \ref{lem:ils7}}\label{Ramsumco}

%ILS7 generalization
%%%%%%%%%%%%%%%%%%%%%%%%%%%%%%%%%%%%%

This is a generalization of   \cite[Sect. 7]{ILS} and \cite[Lem. 2.12]{HM}. We take this opportunity to correct typos and include more details that were omitted in previous works.

We first claim that the left side of \eqref{eq:ils7} is equal to
\begin{equation}\label{eq:A1}
       \lim_{\epsilon \to 0} \, \sum_{(b,M)=1} \frac{R(1,b)R(m^2,b)}{\varphi(b) b^{\epsilon}} \int_{0}^{\infty} J_{k-1}(y) \hat{\phi} \left( \frac{2\log(by\sqrt{Q}/4\pi m)}{\log R} \right) \frac{\d y}{\log R}. 
    \end{equation}
Since $\widehat{\phi}$ is compactly supported, the $b$-sum and the $y$-integral of \eqref{eq:A1} in total is bounded by 
\begin{align}
    \sum_{b=1}^{\infty} \, \frac{m^4}{\phi(b)} \int_{y\ll \frac{mR^{O(1)}}{b\sqrt{Q}}} \,  \frac{\d y}{\log R} \ \ll \ \frac{m^5 R^{O(1)}}{\sqrt{Q}\log R} \,  \sum_{b=1}^{\infty} \, \frac{1}{\varphi(b)b}, 
\end{align}
where the bounds $|R(1,b)| = 1$, $|R(m^2, b)| \le m^4$, $J_{k-1}(x) \ll 1$ (see (\ref{eq:vonsterneck}) and Lemma \ref{lem:Bessel}) are used above. The last $b$-sum converges because $\varphi(b) \gg b/\log \log b$ (see \cite[Thm. 13.14]{Ap76}). Our claim now readily follows from the Dominated Convergence Theorem.

Apply the Mellin inversion formulae (\ref{melT}) and   \eqref{eq: Jk Mellin} to the integral of \eqref{eq:A1}. We have
\be
   \hspace{10pt} \int_{0}^{\infty} J_{k-1}(y) \hat{\phi}\left( \frac{2\log(by\sqrt{Q}/4\pi m)}{\log R} \frac{dy}{\log R}\right) \ = \ \int_{-\infty}^{\infty} \phi(x \log R) \left( \frac{2\pi m}{b\sqrt{Q}}\right)^{4\pi i x} \frac{\Gamma\left( \frac{k}{2} - 2\pi i x\right)}{\Gamma\left( \frac{k}{2} + 2\pi i x\right)} \, \d x. \label{eqn Melplan}
\ee
Substituting \eqref{eqn Melplan} into \eqref{eq:A1} and interchanging the sum and integral, we have
\begin{equation}  
     \eqref{eq:A1} \ = \ \lim_{\epsilon \to 0} \int_{-\infty}^{\infty} \phi(x \log R) \left( \frac{2\pi m}{\sqrt{Q}}\right)^{4\pi i x} \frac{\Gamma\left( \frac{k}{2} - 2\pi i x\right)}{\Gamma\left( \frac{k}{2} + 2\pi i x\right)} \, \chi_{_M}(\epsilon + 4\pi i x; m) \ \d x\label{eq:A2},
\end{equation}
where the Dirichlet series
\be
    \chi_{_M}(s; m) \ \defeq \  \sum_{(b,M)=1} \frac{R(1,b)R(m^2,b)}{\varphi(b)b^{s}}\label{eq: RamDS}
\ee
converges absolutely on the half-plane 
$\Re \, s >0$.

We evaluate the integral of \eqref{eq:A2} asymptotically by breaking it into two pieces: one for $|x| \le  X$ and the other for $|x| > X$.  The first piece can be handled by Laurent expansions while the second piece contributes negligibly due to the rapid decay of $\phi$. We must carefully keep track of the dependence on $m$, $M$. 

By \eqref{eqn Ramult}, the Dirichlet series \eqref{eq: RamDS} can be expressed in terms of  an Euler product
    \begin{align}\label{convEPRam}
        \chi_{_M}(s; m) \  = \ \prod_{(p,M) = 1}\sum_{t= 0}^{\infty}\frac{R(1,p^t)R(m^2,p^t)}{\varphi(p^t)p^{ts}}
        \ = \ \prod_{(p,M) = 1}\begin{cases} \left(1 + \frac{1}{(p-1)p^s}\right), & \text{if } (p,m) = 1,\\
        \left(1 - \frac{1}{p^s}\right) & \text{if } (p,m) > 1 \end{cases}
    \end{align}
    on $\Re \, s> 0$, where the second equality follows from the facts that $R(1,b) = \mu(b)$, $R(m^2, 1) = 1$, $R(m^2,p) = -1$ if $(m,p) = 1$, and $R(m^2,p) = \varphi(p) = p-1$ if $(p,m) > 1$. Moreover, (\ref{convEPRam}) can be written as 
    \begin{align}\label{convEP}
        \chi_{_M}(s; m) \ & = \ \prod_{p}\left(1 + \frac{1}{(p-1)p^s}\right) \,  \cdot\, \prod_{p|M} \left(1 + \frac{1}{(p-1)p^s}\right)\inv  \, \cdot \, \prod_{p|\frac{m}{(m,M^\infty)}}\left(1 - \frac{1}{p^s}\right) \left(1 + \frac{1}{(p-1)p^s}\right)\inv \nonumber\\
        \ &=: \  \chi^{(1)}(s) \, \cdot\,  \chi_{_M}^{(2)}(s) \, \cdot \,  \chi_{_M}^{(3)}(s; m). 
    \end{align}
    This corrects a mistake made in \cite[pp. 99]{ILS} regarding the factorization of $\chi_{_M}(s; m)$. For this reason, the relevant  Laurent expansions must be re-computed and the correct results can be obtained as follows. 

    For $\Re \, s >0$, it is easy to verify that
    \begin{align}\label{EPcomp}
       \chi^{(1)}(s)
        \ &= \ \frac{\zeta(1+s)}{\zeta(2+2s)} \, \prod_p \, \left(1 + \frac{1}{(p-1)(p^{1+s} + 1)}\right). 
    \end{align}
For $s=O(1)$, we have
\begin{align}
     \chi^{(1)}(s) \ = \  \left(s^{-1}+O(1)\right) \left( \frac{1}{\zeta(2)} \prod_{p} \, \left(1+ \frac{1}{p^2-1}\right) \, + \, O(|s|)\right) \ = \  s^{-1} \, +\, O(1). 
\end{align}

Define $\zeta_{M}(s):= \prod\limits_{p\mid M} \, (1-p^{-s})^{-1}$. In particular, $\zeta_{M}(1)=M/\varphi(M)$. For any  $s= O(1/\log \log M)$, we have
\begin{align}
     \zeta_{M}(1+s) \cdot \frac{\varphi(M)}{M} \ &= \  \prod_{p\mid M} \, \frac{1-p^{-1}}{1-p^{-1}+ O\left(|s|\frac{\log p}{p}\right)} \ = \  \prod_{p\mid M} \, \left(1 \, + \, O\left(|s|\frac{\log p}{p}\right)\right)^{-1} \nonumber\\
     \ &=  \ 1 \, + \, O\left(|s| \sum_{p\mid M} \, \frac{\log p}{p}\right)  
    \ = \ 1 \, + \, O(|s|\log \log M).  
\end{align}
From this and (\ref{EPcomp}), it follows that
\begin{align}
   \chi_{_M}^{(2)}(s)  \ &= \ \frac{\varphi(M)}{M}  \left( 1 \, + \, O(|s|\log \log M)\right) \cdot \left\{\zeta_{M}(2) \prod_{p\mid M} \left( 1+ \frac{1}{p^2-1}\right)^{-1} \, + \, O(|s|)\right\} \nonumber\\
    \ &= \  \frac{\varphi(M)}{M}  \left( 1 \, + \, O(|s|\log \log M)\right). 
\end{align}

Suppose $m/(m, M^{\infty})>1$. For $s=O(1/\log m)$, we have
\begin{align}
    \prod_{p|\frac{m}{(m,M^\infty)}} \, (1-p^{-s}) \ \ll \   \prod_{p|\frac{m}{(m,M^\infty)}} \, |s|\log p \ \ll \ |s| \prod_{p|\frac{m}{(m,M^\infty)}} \, \log p. 
\end{align}
Observe that
\begin{align}
  \log\,  \prod_{p|\frac{m}{(m,M^\infty)}} \, \log p \ \ll \ \sum_{p=O(\log m)} \, \log \log p \ \ll \ \frac{(\log m) \log \log \log m }{\log \log m},
\end{align}
which implies 
\begin{align}
     \prod_{p|\frac{m}{(m,M^\infty)}} \, (1-p^{-s}) \ \ll \ |s| m^{\epsilon'}
\end{align}
for any $\epsilon'>0$. \footnote{ In this proof, we distinguish $\epsilon$ with $\epsilon'$ for clarity. } When $m/(m, M^{\infty})=1$, we have $ \prod_{p|\frac{m}{(m,M^\infty)}}  (1-p^{-s})=1$.   In other words, 
  \begin{align}
       \chi_{_M}^{(3)}(s; m) \ &= \ \delta\left(\frac{m}{(m,M^\infty)},1\right) \, + \, O(|s|m^{\epsilon'}).
    \end{align}

Altogether,  for $|s| \ll X\ll  1/ ((\log m) \log \log M)$, the following estimate holds:
    \begin{align}
        \chi_{_M}(s; m) \ = \ \frac{\varphi(M)}{sM} \delta\left(\frac{m}{(m,M^\infty)},1\right) \, +  \, O\left( m^{\epsilon'}  \right). 
    \end{align}
    
    From \cite[eq. (8.322)]{GR} we have
    \begin{align}
        \Gamma\left(\frac{k+s}{2}\right) \ = \ \Gamma\left(\frac{k-s}{2}\right)\left(\frac{k}{2}\right)^s\left[1 + O\left(\frac{|s|}{k}\right)\right].
    \end{align}
   For $s=\epsilon + 4\pi i x$ with $\epsilon, x\ll X$, we have
    \begin{align}\label{eq: chi Gamma bound}
         \chi_{_M}(s;m) \, \frac{\Gamma(\frac{k}{2} - 2\pi i x)}{\Gamma(\frac{k}{2} + 2\pi i x)} \ = \ \frac{\varphi(M)}{sM} \delta\left(\frac{m}{(m,M^\infty)},1\right)\left(\frac{k}{2}\right)^{-4\pi i x} \, + \,  O\left( m^{\epsilon'} \right).
    \end{align}
Make a change of variables $x \to -x$ and use the evenness of $\phi$, we have
\begin{align}\label{smallasymp}
    &\int_{-X}^{X} \phi(x \log R)  \left( \frac{2\pi m}{\sqrt{Q}}\right)^{4\pi i x} \frac{\Gamma\left( \frac{k}{2} - 2\pi i x\right)}{\Gamma\left( \frac{k}{2} + 2\pi i x\right)} \, \chi_{_M}(\epsilon + 4\pi i x; m) \ \d x \nonumber\\
    & \hs{1} =  \   \frac{\varphi(M)}{M} \delta\left(\frac{m}{(m,M^\infty)},1\right) \int_{-X}^{X} \phi(x\log R)\left(\frac{k\sqrt{Q}}{4\pi m}\right)^{4\pi i x} \frac{dx}{\epsilon - 4\pi i x} \, + \, O\left( m^{\epsilon'}X  \right).
    \end{align}

We move on to bounding each of the three products over primes in (\ref{convEP}) for  $s = \epsilon + 4\pi i x$ with $|x| > X$.  

The last infinite product and $\zeta(2+2s)^{-1}$ in (\ref{EPcomp}) are  clearly $O(1)$. Using the bound $\zeta(1+s) \ll \log |x|$ for $|x|\ge 3$ (see \cite[Thm. 13.4]{Ap76}) and $\zeta(1+s)= O(|s|^{-1})$ as $s\to 0$,  we have
\begin{align}
 | \chi^{(1)}(s)| \ \ll \  |\zeta(1 + s)| \ \ll \ X^{-1} \log \, (3+|x|). 
\end{align}   
Next, observe the inequality
\begin{align}\label{absboundEP}
   \hspace{25pt}  \log\,  \prod_{\substack{p|M \\ p>2}} \, \left|1 + \frac{1}{(p-1)p^s}\right|^{-1}  \ \le \  - \, \sum_{\substack{p|M \\ p>2}} \, \log  \, \left(1 - \frac{1}{p-1}\right) \ = \  \sum_{\substack{p|M}} \, \frac{1}{p-1} \ + \ O(1).
\end{align}
From the identity $\zeta_{M}(1)=M/\varphi(M)$ and Taylor's expansion, one easily obtains
\begin{align}
    \log \, \frac{M}{\varphi(M)} \ = \  \sum_{p\mid M} \, \frac{1}{p - 1} \ + \ O(1). 
\end{align}
 As a result, we have
\begin{align}
    (\ref{absboundEP}) \ \le \ \log \frac{M}{\varphi(M)} \ + \ O(1) \ \le \ \log \log \log M \ + \ O(1)
\end{align}
using the  bound $\varphi(M) \gg M/\log \log M$ (\cite[Thm. 13.14]{Ap76}). We may deduce that
\begin{align}
   |\chi_{_M}^{(2)}(s)|  \ \ll \ \log \log M,
\end{align}
where the implicit constant is absolute. The estimate for $ \chi_{_M}^{(3)}(s; m)$ follows from a similar argument. So,  we have that 
\begin{align}\label{SmallExponentm}
    | \chi_{_M}(s;m)| \ \ll \ X^{-1} m^{\epsilon'} (\log \log M) \log \, (3+|x|)
\end{align}
for any $\epsilon'>0$. 

Equation \eqref{SmallExponentm} and  the decay of $\phi$  imply 
\begin{align}\label{decayest1}
        &   \hspace{-80pt}\left|\int_{X}^{\infty}\phi(x\log R) \chi_{_M}(\epsilon + 4\pi i x; m) \left(\frac{2\pi m}{\sqrt{Q}}\right)^{4\pi i x}\frac{\Gamma(\frac{k}{2} - 2\pi i x)}{\Gamma(\frac{k}{2} + 2\pi i x)} \, \d x\right| \notag \\
        \ &\ll  \ X^{-1}\,  m^{\epsilon'} (\log \log M )\, \int_{X}^{\infty} (x\log R)^{-A} \log\, (3+|x|) \ \d x \notag \\\iffalse
        \ll & \ m^{\log_5(4)} \int_{X}^{\infty}\left(\frac{1}{x\log R}\right)^{B + 1} \left(\frac{1}{4\pi X} + 3 + 4\pi x\right) dx \notag\\
        = & \ m^{\log_5(4)} \frac{1}{(\log R)^{B+1}} \left(\frac{1}{BX^B}\left(\frac{1}{4\pi X} + 3\right) + \frac{4\pi }{(B-1)X^{B-1}}\right)\fi
        \ &\ll \ m^{\epsilon'} (\log \log M ) (X \log R)^{-A},
    \end{align}
    as well as
\begin{align}\label{decayest2}
     \left|\int_{X}^{\infty} \phi(x\log R)\left(\frac{k\sqrt{Q}}{4\pi m}\right)^{4\pi i x} \frac{dx}{\epsilon - 4\pi i x}\right|
     \ &\ll \ \int_{X}^{\infty} \frac{1}{x} \, (x\log R)^{-A} \, \d x  \ \ll \ (X \log R)^{-A}
 \end{align}
 for any $A>0$.

Combining the estimates (\ref{decayest1}), (\ref{decayest2}) and (\ref{smallasymp}), it follows that
    \begin{align}
          (\ref{eq:A1}) \ = \ \delta\left(\frac{m}{(m,M^\infty)},1\right) \, & \frac{\varphi(M)}{M} \, \lim_{\epsilon \downarrow 0} \, \int_{-\infty}^{\infty}\phi(x\log R)  \left(\frac{k\sqrt{Q}}{4\pi m}\right)^{4\pi i x} \frac{dx}{\epsilon -4\pi ix} \nonumber\\
          \, &+\,  O\left( m^{\epsilon'} (\log \log M ) (X \log R)^{-A}\right) \ + \  O\left( m^{\epsilon'}X \right)
     \end{align}
for any $\epsilon'>0$, $A>0$, and $X \ll 1/((\log m) \log\log M)$.  

Taking $A = 1$ and $X^{-1}:= (\log m) (\log\log M)(\log R)^{1/2}$, and following the same argument of  \cite[pp. 100]{ILS}, we may now conclude 

 \begin{align}
      (\ref{eq:A1})  \ &= \ \delta\left(\frac{m}{(m,M^\infty)},1\right) \frac{\varphi(M)}{M} \left(-\frac{1}{2} \int_{-\infty}^{\infty} \phi(x) \sin \left( 2 \pi x\frac{\log \left(  k^2 Q/16\pi^2 m^2 \right) }{\log R} \right) \frac{dx}{2\pi x} + \frac{1}{4} \phi(0) \right) \nonumber \\
     &\hs{1} \, + \,  O\left(m^{\epsilon'} \, \frac{(\log \log M)^{2}}{(\log R)^{1/2}} \right). 
 \end{align}
The proof of Proposition \ref{lem:ils7} is complete.

%%%%%%%%%%%%%%%%%%%%%%%%%%%%%%%%%%%%%%%%%%%%%%%%%%%%%%%%%%%%%%%%%%%%%%%%%%%%%%%%%%%%%%%%%%%%%%%%%%%%%%%%%%%%%%%%%%%%%%%%%%%%%%%%%%%%%%%%%%%%%%%%%%%%%%%%%%%%%%%%%%%%%%%%%%%%%%%%%%%%%%%%%%%%%%%%%%%%%%%%%%%%%%%%%%%%%%%%%%%%%%%%%%%%%%%%%%%%%%%%%%%%%%%%%%%%%%%%%%%%%%%%%%%%%%%%%%%%%%%%%%%%%%%%%%%%%%%%%%%%%%%%%%%%%%%%%%%%%%%%%%%%%%%%%%%%%%%%%%%%%%%%%%%%%%%%%%%%%%%%%%%%%%%%%%%%%%%%%%%%%%%%%%%%%%%%%%%%%%%%%%%%%%%%%%%%%%%%%%%%%%%%%%%%%%%%%%%%%%%%%%%%%%%%%%%%%%%%%%%%%%%%%%%%%%%%%%%%%%%%%%%%%%%%%%%%%%%%%%%%%%%%%%%%%%%%%%%%%%%%%%%%%%%%%%%%%%%%%%%%%%%%%%%%%%%%%%%%%%%%%%%%%%%%%%%%%%%%

%Vanishing off the diagonal

\subsection{ Proof of Proposition~\ref{lem:nonmaintermcalc} }\label{sec: vanishing}
Suppose $n_j + m_j > 2$ for some $1 \le j \le \omega$ in \eqref{eq:E Ram}.  Firstly,  applying Proposition~\ref{lem:sums to integrals} to \eqref{eq:E Ram} with $c = p_1\cdots p_\alpha q_1^{m_1} \cdots q_\omega^{m_\omega}$, we obtain
\begin{equation}\label{eq:sums to integrals genterm simp}
    \begin{split}
      E(\vec{n}, \vec{m})  \ = \  & \qsumsimp{\omega} \prod_{j=1}^{\omega}  \testfn{q_j}^{n_j} \frac{\log^{n_j} q_j} {q_j^{(n_j + m_j)/2} \log^{n_j} R}  \psumsimp{\alpha + \delta} \prod_{i=1}^{\alpha + \delta} \testfn{p_i} \frac{\log p_j}{p_j \log R} \\
        &\hspace{1cm}\times  \msum \frac{1}{m^2} \bsumfull{ N p_{\alpha+1} \cdots p_{\alpha+ \delta} q_{\theta+1} \cdots q_\omega }{p_1, \ldots, p_{\alpha}, q_1, \ldots, q_\theta} \frac{R(m^2, b) R(p_1\cdots p_{\alpha}q_1^{m_1}\cdots q_{\theta}^{m_{\theta}}, b) }{\varphi(b)}  \\
        & \hspace{1cm}  \times\int_{0}^{\infty} J_{k-1}(x) \widehat{\Phi_{n-\alpha- \delta}} \left( \frac{2\log(bx\sqrt{N/(p_{1} \cdots p_{\alpha+ \delta} q_1^{m_1} \cdots q_{\omega}^{m_{\omega}})} /4\pi m)}{\log R} \right) \frac{dx}{\log R} 
    \end{split}
\end{equation}
up to an error  $\ONe$.  The condition $b < N^{2022}$ has been removed at a negligible cost. 

Secondly, we convert the sums over $q_j$ and $p_i$ in \eqref{eq:sums to integrals genterm simp} into sums over \emph{distinct} primes, requiring us to break up the sums depending on whether some of the primes in the sums are equal or not, as well as the order of the primes factors of $b$. It follows that $\Enm$ is  a sum of terms of the form
\begin{align}
\Fabcde \ & \defeq \ \qsumdist{\ell} \prod_{j=1}^{\ell}  \testfn{q_j}^{a_j} \frac{\log^{a_j} q_j}{q_j^{b_j} \log^{a_j} R}  \msum \frac{1}{m^2} \sum_{\substack{(b', Nq_{1}\ldots q_\ell) = 1 \\ b = b' q_1^{c_1} \cdots q_\kappa^{c_{\kappa}}}} \frac{R(m^2, b) R(q_1^{d_1} \cdots q_\kappa^{d_\kappa}, b) }{\varphi(b)} \nonumber \\
& \hspace{1cm}  \times\int_{x=0}^{\infty} J_{k-1}(x) \widehat{\Phi_{n-\nu}} \left( \frac{2\log(b'x\sqrt{N q_1^{e_1} \cdots  q_\ell^{e_\ell}} /4\pi m)}{\log R} \right) \frac{dx}{\log R} \label{eq: abcde}
\end{align}
up to an error term of size $\ONe$, where $a_i$, $b_i$, $c_i$, and $d_i$ are positive integers, and the $e_i$'s are integers. Additionally, we have that $\sum a_j = \nu$ and $b_i > 1$ for some $i$ (since $n_j + m_j > 2$ for some $j$), as well as $b_j \ge d_j$ for all $j$ (since $n_j \ge m_j$ in \eqref{eq:sums to integrals genterm simp}).

As a result, the proof of Proposition~\ref{lem:nonmaintermcalc} rests on the following lemma. In fact, this lemma will also be useful in Section~\ref{sec:maintermlaststep}.

\begin{lem}\label{lem:abcde}
    Let $\Fabcde$ be defined as in \eqref{eq: abcde} with $a_j, b_j, c_j, d_j$'s being positive integers, $e_j$'s being integers, $b_j \ge d_j$ for all $1 \le j \le \kappa$. If $b_i > 1$ or $d_i < c_i$ for some $i$, then $\Fabcde \ll 1/\log N$.
\end{lem}

\begin{proof}
Using the multiplicativity of the Ramanujan sums and the totient function $\varphi$, observe that
\begin{equation}\label{eq:Ram simp general}
    \frac{R(m^2, b) R(q_1^{d_1} \cdots q_\kappa^{d_\kappa}, b) }{\varphi(b)}\ =\ \frac{R(q_1^{d_1} \cdots q_\kappa^{d_\kappa}, q_1^{c_1} \cdots q_\kappa^{c_{\kappa}})}{ \varphi(q_1^{c_1} \cdots q_\kappa^{c_{\kappa}})} \cdot  \frac{R(m^2, b') R(m^2, q_1^{c_1} \cdots q_\kappa^{c_{\kappa}})  R(1, b')}{\varphi(b') },
\end{equation}
where  $b = b' q_1^{c_1} \cdots q_\kappa^{c_{\kappa}}$ and $(b', N q_1\cdots q_\ell ) = 1$.  Apply \eqref{eq:Ram simp general} to \eqref{eq: abcde}. This allows us to rewrite $\Fabcde$ (with $Q = N q_1^{e_1} \cdots  q_\ell^{e_\ell}$) as
\begin{align}\label{eq: abcde Ram simp}
         &\qsumdist{\ell} \,  \prod_{j=1}^{\ell}  \testfn{q_j}^{a_j} \frac{\log^{a_j} q_j}{q_j^{b_j} \log^{a_j} R}   \msum \, \frac{R(q_1^{d_1} \cdots q_\kappa^{d_\kappa}, q_1^{c_1} \cdots q_\kappa^{c_{\kappa}})}{ \varphi(q_1^{c_1} \cdots q_\kappa^{c_{\kappa}})}  \frac{R(m^2, q_1^{c_1} \cdots q_\kappa^{c_{\kappa}})}{m^2}  \nn 
        & \hs 1 \times \sum_{(b', Nq_{1}\ldots q_\ell) = 1 } \frac{R(m^2, b') R(1, b') }{\varphi(b')} \int_{x=0}^{\infty} J_{k-1}(x) \widehat{\Phi_{n-\nu}} \left( \frac{2\log(b'x\sqrt{Q} /4\pi m)}{\log R} \right) \frac{dx}{\log R}. 
    \end{align}
By Proposition~\ref{lem:ils7},  observes that the sum over $b'$ in \eqref{eq: abcde Ram simp} is $\ll m^{\epsilon}$, and hence,  \eqref{eq: abcde Ram simp} is 
\begin{equation}\label{eq: abcde Y bound}
     \ll \,    \qsumdist{\ell} \prod_{j=1}^{\ell} \,   \left|\testfn{q_j}^{a_j} \frac{\log^{a_j} q_j}{q_j^{b_j} \log^{a_j} R} \right| \msum \, \frac{|R(q_1^{d_1} \cdots q_\kappa^{d_\kappa}, q_1^{c_1} \cdots q_\kappa^{c_{\kappa}})|}{ \varphi(q_1^{c_1} \cdots q_\kappa^{c_{\kappa}})}  \frac{\left| R( m^2, q_1^{c_1} \cdots q_\kappa^{c_{\kappa}}) \right|}{m^{2 - \epsilon}} .
\end{equation}
The last $m$-sum is bounded by
\begin{equation}
    \left[ \sum_{(m', q_1 \cdots p_\kappa) = 1} \frac{1}{(m')^{2-\epsilon}} \right]  \left[ \prod_{i=1}^{\kappa} \sum_{t \ge 0} \frac{|R(q_i^{2t}, q_i^{c_i})||R(q_i^{d_i}, q_i^{c_i})|} {q_i^{(2- \epsilon) t} \varphi(q_i^{c_i})}\right] \ \ll \  \prod_{i=1}^{\kappa} \sum_{t \ge 0} \frac{|R(q_i^{2t}, q_i^{c_i})||R(q_i^{d_i}, q_i^{c_i})|} {q_i^{(2- \epsilon) t} \varphi(q_i^{c_i})}
\end{equation}
once again due to the multiplicativity of the Ramanujan sums.

 We now analyze the sum over $t$, primarily relying on \eqref{eq:vonsterneck} to bound the Ramanujan sums. When $2t < c_i - 1$, then $R(q_i^{2t}, q_i^{c_i}) = 0$. When $2t = c_i - 1$, then $R(q_i^{2t}, q_i^{c_i}) = q_i^{2t}$. When $2t \ge c_i$, we have that $R(q_i^{2t}, q_i^{c_i}) = \varphi(q_i^{c_i}) \le q_i^{c_i}$. 

 The sum over $t$ is $O(q_i^{\epsilon\lfloor c_i/2 \rfloor})$ when $d_i \ge c_i$, and is  $O(q_i^{-1 + \epsilon\lfloor c_i/2 \rfloor })$  when  $d_i < c_i$, where the bounds $|R(q_i^{d_i}, q_i^{c_i})| \le \varphi(q_i^{c_i})$ and $|R(q_i^{d_i}, q_i^{c_i})| \le q_i^{d_i}$ were applied respectively. Therefore,   
\begin{align}\label{eq: abcde factor}
        \Fabcde \ &\ll \ \qsumdist{\ell} \prod_{j=1}^{\ell} \, \left| \testfn{q_j}^{a_j} \right| \frac{\log^{a_j} q_j}{q_j^{b_j + \eta_j - \epsilon\lfloor c_j/2 \rfloor} \log^{a_j} R} \nonumber\\
        \ &\ll \ \prod_{j=1}^{\ell} \left[ \sum_{p} \, \left| \testfn{p}^{a_j} \right|\frac{\log^{a_j} p}{p^{b_j + \eta_j - \epsilon\lfloor c_j/2 \rfloor} \log^{a_j} R} \right],
\end{align}
where $\eta_j = \I_{d_j < c_j}$.  
Observe the following:
\begin{enumerate}
    \item Set $x_j := b_j + \eta_j - \epsilon\lfloor c_j/2 \rfloor$. The sum over $p$  in \eqref{eq: abcde factor} is $O(1/\log^{a_j} R)$ when $x_j > 1$,  and  is $O(1)$ when $x_j = 1$. 

    \item Suppose $d_j \ge c_j$. By assumption,  we have  $b_j \ge d_j$ and so $b_j \ge c_j$. If $b_j = 1$,  then $x_j = 1$,  and if $b_j > 1$,  then $x_j > 1$. Suppose $d_j < c_j$. Then  $x_j > 1$. 
\end{enumerate}
In particular,  $x_j \ge 1$ always holds and each of the $p$-sums in \eqref{eq: abcde factor} is  $O(1)$. 

By our assumption, there exists $i$ for which either $d_i < c_i$ or $b_j > 1$ hold. In either case, we have $x_i > 1$ and the $i$-th factor in \eqref{eq: abcde factor} is thus  $\ll 1/\log^{a_i} R \ll 1/\log N$ (since $a_i > 0$ by assumption). Taking the product over all $j$'s, we may now conclude that $\Fabcde \ll 1/\log N$. This completes the proof of the lemma. 
\end{proof}
\begin{proof}[Proof of Proposition~\ref{lem:nonmaintermcalc} ]
  By Lemma \ref{lem:abcde} and the arguments preceding \eqref{eq: abcde}, $\Enm$ can be written as the sum of finitely many terms of sizes $O(1/\log N)$ whenever $n_j + m_j > 2$ for some $j$. Note that the number of such terms is independent of $N$.   This completes the proof. 
\end{proof}

\subsection{Analytic simplification of the main contribution}\label{sec:maintermlaststep}

Recall the expression \eqref{eq:E Ram} for $E(\vec{n}, \vec{m})$ and the fact that  the main contribution comes from terms with $\omega = 0$  and  $n_j = m_j = 1$ for all $1 \leq j \leq n$, i.e., 
\begin{align}
        A \ &\defeq 
         \  -\sum_{\alpha = 0}^{a-1} \binom{n}{\alpha} 2^{n+1} \pi  \psumsimp{\alpha} \prod_{j=1}^\alpha \hat{\phi}\left(\frac{\log p_j}{\log R} \right) \frac{\log p_j}{\sqrt{p_j} \log R} \msum \frac{1}{m} \bsumdiv{N}{p_1,\ldots, p_\alpha}    \frac{1}{b\varphi(b)} R(m^2,b) R(p_1 \cdots p_\alpha,b) \nonumber   \\
        &\hspace{1cm} \times \rootn \sum_{p_{\alpha+1}, \ldots, p_n} \Jk \left(\frac{4\pi m \sqrt{p_{1}\cdots p_n}}{b\sqrt{N}} \right)  \prod_{j=\alpha + 1}^n \hat{\phi}\left(\frac{\log p_j}{\log R} \right) \frac{\chi_0 (p_j)\log p_j}{\sqrt{p_j} \log R} \, + \, \ONe. \label{eq:MTalphaCase}   
\end{align}

 \begin{rem}
 \ 
 \begin{enumerate}
     \item The coefficient  $\binom{n}{\alpha}$ comes from the choices of indices for the prime factors of $b$.

     \item We have truncated the $\alpha$-sum  in \eqref{eq:MTalphaCase} to $0\le \alpha\le a-1$ because the contribution from  $\alpha > a-1$ the term is $\ONe$, which follows from the bounds $\Jk(x) \ll x$, $R(m^2, b) \le m^4$ and $R(p_1\cdots p_\alpha, b) \le \varphi(b)$. 
 \end{enumerate}
      
 \end{rem}

Applying Proposition~\ref{lem:sums to integrals} to \eqref{eq:MTalphaCase}, we have, upon simplification,  
\begin{align}
     A\ &= \ -2^n\, \sum_{\alpha = 0}^{a-1} \binom{n}{\alpha} \sum_{\delta = 0}^{a - \alpha - 1} \binom{n - \alpha}{\delta} \sum_{i=0}^{a - \alpha - \delta - 1} (-1)^i \binom{n - \alpha - \delta}{i} H(\alpha, \delta) + \ONe, \label{eq:A to F}
\end{align}
\no where 
\begin{align}\label{eq:Gdef}
    H(\alpha, \delta) &\ \defeq \  \sum_{p_1, \ldots, p_{\alpha + \delta} } \prod_{j=1}^{\alpha + \delta} \hat{\phi}\left(\frac{\log p_j}{\log R} \right) \frac{\log p_j}{p_j \log R}  \msum \frac{1}{m^2}   \bsumfull{Np_{\alpha+1} \cdots p_{\alpha + \delta} }{p_1, \ldots, p_\alpha}  \frac{R(m^2, b) R(p_1 \cdots p_{\alpha}, b)} {\varphi(b)} \nonumber \\  
    & \times \int_{x=0}^{\infty} J_{k-1}(x) \widehat{\Phi_{n-\alpha- \delta}} \left( \frac{2\log(bx\sqrt{N} /(4\pi m\sqrt{p_1 \cdots p_{\alpha + \delta}}))}{\log R} \right) \frac{dx}{\log R }.
\end{align} 
The rest of this subsection is dedicated to proving the following lemma. 

\begin{lem}\label{lem:Gtosin}
We have that
    \begin{align}
        H(\alpha, \delta)\ &= \ -2^{-1-\alpha-\delta} (-1)^\alpha \intii \cdots \intii \hphi(x_2) \cdots \hphi(x_{\alpha+\delta+1})\nonumber \\
        &\hspace{1cm}\times \left[ \intii \phi^{n-\alpha-\delta}(x_1) \frac{\sin\left(2\pi x_1 (1 + |x_2| + \cdots + |x_{\alpha + 1}| - |x_{\alpha +2}| - \cdots - |x_{\alpha + \delta + 1}|)\right)}{2\pi x_1}dx_1\right.\nonumber \\
        &\hspace{1cm}- \left.\frac{1}{2}\phi^{n-\alpha-\delta}(0) \right]dx_2\cdots dx_{\alpha+\delta+1} + \converrorterm.
    \end{align}
\end{lem}

First, we transform the sum over the primes $p_1, \ldots, p_{\alpha +\delta}$ in $H(\alpha, \delta)$ to a sum over distinct primes. 
\begin{propt}
    A distinctness condition can be added to \eqref{eq:Gdef} at the cost of an error of size  $O(1/\log N)$.
\end{propt}
\begin{proof}
The distinctness condition can be imposed to the sums over primes of \eqref{eq:Gdef} by inclusion-exclusion,  depending on which primes are equal. If $p_i = p_j$ for some $1 \le i \le \alpha$ and $\alpha + 1\le j \le \alpha + \delta$, then the corresponding term of $H(\alpha, \delta)$ is zero due to the condition on the $b$-sum.

Without loss of generality, we let $p_1\cdots p_\alpha = q_1^{u_1} \cdots q_{\alpha'}^{u_{\alpha'}}$ and $p_{\alpha + 1} \cdots p_{\alpha + \delta'} = q_{\alpha' + 1}^{u_{\alpha'+1}} \cdots q_{\alpha' + \delta'}^{u_{\alpha' + \delta'}}$, where the primes $q_i$ are distinct and at least one $u_i > 1$. To add in a distinctness condition to $H(\alpha, \delta)$, our inclusion-exclusion introduced   terms of the form
\begin{align}
     &\sum_{\substack{q_1, \ldots, q_{\alpha' + \delta'}\\ q_i \text{ distinct}} } \prod_{j=1}^{\alpha' + \delta'} \hat{\phi}\left(\frac{\log q_j}{\log R} \right)^{u_j} \frac{\log^{u_j} q_j}{q_j^{u_j} \log^{u_j} R} \,  \sum_{m \le N^\epsilon} \frac{1}{m^2}   \sum_{\substack{(b,Nq_{\alpha' +1} \cdots q_{\alpha' + \delta'}) = 1  \\ q_1^{u_1}, \ldots, q_{\alpha'}^{u_{\alpha'}} \mid b}}  \frac{R(m^2, b) R(q_1^{u_1} \cdots q_{\alpha'}^{u_{\alpha'}}, b)} {\varphi(b)} \nonumber \\  
    &\hspace{1cm} \times \int_{x=0}^{\infty} J_{k-1}(x) \widehat{\Phi_{n-\alpha- \delta}} \left( \frac{2\log(bx\sqrt{N} /(4\pi m\sqrt{q_1^{u_1} \cdots q_{\alpha' + \delta'}^{u_{\alpha' + \delta'}}}))}{\log R} \right) \frac{dx}{\log R }.\label{eq:Gremainder}
\end{align}

After breaking up \eqref{eq:Gremainder} based on the multiplicities of the prime factors of $b$, we appeal to Lemma~\ref{lem:abcde} and find \eqref{eq:Gremainder} is $O(1/\log N)$ since at least one $u_i > 1$. This completes the proof. 
\end{proof}

Upon inserting the distinctness condition to $H(\alpha, \delta)$, we break up the terms based on the multiplicities of the prime factors of  $b$. Consider the $b$-sum of \eqref{eq:Gdef} over $b$'s  of the form  $b'p_1^{c_1} \cdots p_{\alpha}^{c_{\alpha}}$, where $(b', p_1 \cdots p_\alpha) = 1$. Due to Lemma~\ref{lem:abcde},  if any of the $c_i$'s is $>1$, then the corresponding term is $O(1/\log N)$. Thus, it suffices to consider the case when each of the $c_i$'s is  $1$, i.e.,
\begin{equation}\label{eq:G to H}
    \begin{split}
    H(\alpha, \delta)\ &=  \ \psumdist{\alpha + \delta} \prod_{j=1}^{\alpha + \delta} \hat{\phi}\left(\frac{\log p_j}{\log R} \right) \frac{\log p_j}{p_j \log R}  \msum \frac{1}{m^2}  \sum_{\substack{b = b'  p_1 \cdots  p_{\alpha} \\ (b',Np_{1} \cdots p_{\alpha + \delta}) = 1}}  \frac{R(m^2, b) R(p_1 \cdots p_{\alpha}, b)} {\varphi(b)}\\  
    & \times \int_{x=0}^{\infty} J_{k-1}(x) \widehat{\Phi_{n-\alpha- \delta}} \left( \frac{2\log(bx\sqrt{N} /(4\pi m\sqrt{p_1 \cdots p_{\alpha + \delta}}))}{\log R} \right) \frac{dx}{\log R } \, + \, O(1/\log N)  .
    \end{split}
\end{equation}

The $b$-sum can be simplified with the multiplicativity of $\varphi$ and the Ramanujan sums, $R(q,q) = \varphi(q)$, and the conditions  $b = b'p_1\cdots p_\alpha$ with $(b', N p_1 \cdots p_{\alpha +\delta} ) = 1$. In fact, 
\begin{equation}\label{eq:Ram simp}
\begin{split}
    \frac{R(m^2, b) R(p_1\cdots p_{\alpha}, b)} {\varphi(b)}\ &= \ \frac{R(m^2, b')R(m^2, p_1\cdots p_\alpha)R(1, b')}{\varphi(b')}. 
\end{split}
\end{equation}
Applying this to \eqref{eq:G to H}, we have
\begin{align}
    H(\alpha, \delta)\ &= \ \psumdist{\alpha + \delta} \prod_{j=1}^{\alpha + \delta} \hat{\phi}\left(\frac{\log p_j}{\log R} \right) \frac{\log p_j}{p_j \log R}  \msum \frac{1}{m^2} R(m^2, p_1\cdots p_\alpha)\nonumber \\  
    & \times \sum_{ (b',Np_{1} \cdots p_{\alpha + \delta}) = 1}  \frac{R(m^2, b')R(1, b')}{\varphi(b')} \int_{x=0}^{\infty} J_{k-1}(x) \widehat{\Phi_{n-\alpha- \delta}} \left( \frac{2\log(b'x\sqrt{Q} /(4\pi m))}{\log R} \right) \frac{dx}{\log R }\nonumber \\
    &+O(1/\log N),\label{eq:G to H simp}
    \end{align}
    where $Q = Np_1\cdots p_\alpha /(p_{\alpha+1} \cdots p_{\alpha+\delta})$. 

We are ready to apply Proposition~\ref{lem:ils7} to the sum over $b'$ in \eqref{eq:G to H simp}. We first show that the contribution of the error term of \eqref{eq:ils7} to $H(\alpha, \delta)$ is  $O\left( \frac{(\log \log N)^2}{(\log N)^{1/2}}\right)$. Indeed, using the multiplicativity of the Ramanujan sums,  the $m$-sum  in this case is
\begin{align}
   \ \ll \  \msum \frac{R(m^2, p_1 \cdots p_\alpha) }{ m^2} \cdot m^\epsilon \frac{(\log \log N)^2}{(\log N)^{1/2}} \ &\ll \ \frac{(\log \log N)^2}{(\log N)^{1/2}} \sum_{(m', \, p_1 \cdots p_\alpha) = 1} \frac{1}{(m') ^{2 - \epsilon}} \, \prod_{i=1}^{\alpha} \left[\sum_{t \ge 0} \frac{R(p_i^{2t}, p_i) }{ p_i^{ (2- \epsilon)t}}\right] \nonumber\\
   \ &\ll \ \frac{(\log \log N)^2}{(\log N)^{1/2}} \prod_{i=1}^{\alpha} \left[\sum_{t \ge 0} \frac{R(p_i^{2t}, p_i) }{ p_i^{ (2- \epsilon)t}}\right]. 
\end{align}
When $t = 0$, $R(p_i^{2t}, p_i) = R(1, p_i) = 1$. When $t > 0$, $R(p_i^{2t}, p_i) = \varphi(p_i)<p_i$. From this, it is clear that the sum over $t$ is $O(1)$ (independent of $p_i$) and the sum over $m$ is $O\left( \frac{(\log \log N)^2}{(\log N)^{1/2}} \right)$. Our claim follows by also considering the sums over primes and thus, 
\begin{equation}\label{eq:G ils7}
\begin{split}
    H&(\alpha, \delta)\ = \ \psumdist{\alpha + \delta}  \prod_{j=1}^{\alpha + \delta} \hat{\phi}\left(\frac{\log p_j}{\log R} \right) \frac{\log p_j}{p_j \log R}  \msum \frac{1}{m^2} R(m^2, p_1\cdots p_\alpha) \delta\left(\frac{m}{(m,M^\infty)},1\right)\frac{\varphi(M)}{M} \\  
    & \times \left(-\frac{1}{2}\int_{-\infty}^{\infty}\phi(x)^{n - \alpha - \delta} \sin\left(2\pi x\frac{\log(k^2Q/16\pi^2m^2)}{\log R}\right)\frac{dx}{2\pi x}
        + \frac{1}{4}\phi(0)^{n - \alpha -\delta} \right) +\converrorterm,
    \end{split}
\end{equation}
where $M = Np_1\cdots p_{\alpha + \delta}$. 

The factor $\delta\left(\frac{m}{(m,M^\infty)},1\right)$ and  the fact that $N \nmid m$ force $m$ in \eqref{eq:G ils7} to take the shape $p_1^{t_1} \cdots p_{\alpha + \delta}^{t_{\alpha + \delta}}$.   Additionally, the observations $Q = N p_1 \cdots p_\alpha /(p_{\alpha+1} \cdots p_{\alpha+\delta})$ and 
\begin{align*}
    \frac{\varphi(M)}{M} \  = \ \left(1-\frac{1}{N}\right)\prod_{j=1}^{\alpha+\delta} \, \left(1-\frac{1}{p_{j}}\right)
\end{align*}
permit us to  simplify \eqref{eq:G ils7} as
\begin{align}
    &H(\alpha, \delta)\ = \  \psumdist{\alpha + \delta} \prod_{j=1}^{\alpha + \delta} \hat{\phi}\left(\frac{\log p_j}{\log R} \right) \frac{\log p_j}{p_j \log R}\left(1 - \frac{1}{p_j} \right)  \sum_{0 \le t_1, \ldots , t_{\alpha + \delta} \le \epsilon\log N} \frac{R(p_1^{2t_1} \cdots p_\alpha^{2t_{\alpha}}, p_1 \cdots p_\alpha)}{p_1^{2t_1} \cdots p_{\alpha + \delta}^{2t_{\alpha + \delta}}}\nonumber\\
    &\times\left(-\frac{1}{2}\left. \int_{-\infty}^{\infty}\phi(x)^{n-\alpha-\delta}\sin\left(2\pi x\left(1 + \frac{\log p_1}{\log R} + \cdots + \frac{\log p_\alpha}{\log R}-  \frac{\log p_{\alpha + 1}}{\log R} - \cdots-  \frac{\log p_{\alpha + \delta}}{\log R} \right) \right)\frac{dx}{2\pi x} \right. \right.\nonumber\\
    & \hs{1}\left. + \frac{1}{4}\phi(0)^{n-\alpha-\delta}\right) + \converrorterm. \label{eq: Had0 simp}
    \end{align}
Next, we show that the contribution over the complement of $t_1 = \cdots = t_{\alpha + \delta} = 0$ in \eqref{eq: Had0 simp} is negligibly small. Indeed, upon inserting an absolute value,  the quantity of interest is  bounded by
\begin{equation}\label{eq: tbound}
    \begin{split}
      &\max_{1\le i\le \alpha+\delta} \, \sum_{\substack{p_1, \ldots, p_{\alpha + \delta} \\ p_j < R} } \prod_{j=1}^{\alpha + \delta}  \frac{\log p_j}{p_j \log R} \sum_{\substack{0 \le t_1, \ldots , t_{\alpha + \delta} \le \epsilon\log N \\ t_i \ne 0}} \frac{|R(p_1^{2t_1}, p_1)|\cdots  |R(p_\alpha^{2t_{\alpha}}, p_{\alpha})|}{p_1^{2t_1} \cdots p_{\alpha + \delta}^{2t_{\alpha + \delta}}},
    \end{split}
\end{equation}
which is equal to 
\begin{equation}\label{eq: tbound factored}
    \begin{split}
       \max_{1\le i\le \alpha+\delta} \, 
\prod_{j=1}^{\alpha + \delta} \left[ \sum_{p_j < R} \frac{\log p_j}{p_j \log R} \sum_{\delta(i, j) \le t_j \le \epsilon \log N} \frac{|R(p_j^{2t_j}, p_j^{s_j}) |}{p_j^{2t_j} }\right],
    \end{split}
\end{equation}
where  $s_j = 1$ for $1 \le j \le \alpha$ and 0 for $\alpha +1 \le j \le \alpha + \delta$.

We have $|R(1, p^{s_j})| = 1$ when $t_j = 0$,  whereas the sum over $t_j \ge 1$ is $O(1/p_j)$.  When $j \ne i$,   the sums over $t_j$'s and $p_{j}$'s are both $O(1)$.  When $j=i$, the sum over $t_i$'s is $O(1/p_i)$ as the term with $t_i = 0$ is absent. As a result,  we have
\begin{equation}
    \sum_{p_i < R} \frac{\log p_i}{p_i \log R} \sum_{1 \le t_i \le \epsilon \log N} \frac{|R(p_i^{2t_i}, p_i^{s_i})| }{p_i^{2t_i} }  \ \ll \ \sum_{p_i < R} \frac{\log p_i}{p_i^2 \log R} \ \ll \ 1/\log N.
\end{equation}
Upon taking the product over $j$, we deduce that \eqref{eq: tbound factored} is $O(1/\log N)$. Our claim follows.

It remains to consider the contribution when $t_1 = \cdots = t_{\alpha + \delta} = 0$ (and so $m = 1$). In this case, we approximate $\prod_{j} \, (1 - 1/p_j)$ in \eqref{eq: Had0 simp} by $1$ as the contribution from the rest of the terms in the expansion of such a product is negligibly small. More precisely, we have 
\begin{align}
    &H(\alpha, \delta)\ = \  (-1)^\alpha \psumdist{\alpha + \delta} \prod_{j=1}^{\alpha + \delta} \hat{\phi}\left(\frac{\log p_j}{\log R} \right) \frac{\log p_j}{p_j \log R}\nonumber  \\
    &\times\left(-\frac{1}{2}\left. \int_{-\infty}^{\infty}\phi(x)^{n-\alpha-\delta}\sin\left(2\pi x\left(1 + \frac{\log p_1}{\log R} + \cdots + \frac{\log p_\alpha}{\log R}-  \frac{\log p_{\alpha + 1}}{\log R} - \cdots-  \frac{\log p_{\alpha + \delta}}{\log R} \right) \right)\frac{dx}{2\pi x} \right. \right.\nonumber\\
    & \hs{1}\left. + \frac{1}{4}\phi(0)^{n-\alpha-\delta}\right) \, + \,  \converrorterm,\label{eq: Gad m 1}
    \end{align}
where we use the identity $R(1, p_{1}\cdots p_{\alpha}) = \mu(p_1 \cdots p_\alpha) = (-1)^\alpha$. 
\begin{propt}
    Equation \eqref{eq: Gad m 1} holds with the distinctness condition in the prime sum removed.
\end{propt}
\begin{proof}
    In the process of removing the distinctness condition, we apply inclusion-exclusion and we are left to show that the terms with  $p_i = p_j$ for some $i\neq j$ can be eliminated as they contribute negligibly. Indeed, suppose the condition in the sum of \eqref{eq: Gad m 1} is replaced by  $p_1 \cdots p_{\alpha + \delta} = q_1^{a_1} \cdots q_\ell^{a_\ell}$, where $q_i \ne q_j$ when $i \ne j$ and where $a_j > 1$ for some $j$. Now,  such a contribution  is bounded by
    \begin{equation}\label{eq: nondistinct bound}
        \sum_{\substack{q_1, \ldots, q_\ell \\ q_i \text{ distinct}\\ q_i < R}} \prod_{j=1}^\ell \frac{\log^{a_j}q_j}{q_j^{a_j} \log^{a_j} R} \ \ll \ \sum_{\substack{q_1, \ldots, q_\ell \\ q_i < R}} \prod_{j=1}^\ell \frac{\log^{a_j}q_j}{q_j^{a_j} \log^{a_j} R} \ \ll \  \prod_{j=1}^\ell \left[ \sum_{q_j < R} \frac{\log^{a_j}q_j}{q_j^{a_j} \log^{a_j} R} \right].
    \end{equation}
The last sum is $O(1)$ and $O( 1/\log^{a_j}R)$
when $a_j = 1$ and $a_j > 1$ respectively. Since $a_j > 1$ for some $j$,  we find \eqref{eq: nondistinct bound} is $O(1/\log^2 R)$. This completes the proof.
\end{proof}

To complete the proof of Lemma~\ref{lem:Gtosin}, we apply the Prime Number Theorem with partial summation to each of the sums over primes in \eqref{eq: Gad m 1} without the distinctness condition. We find that
\begin{align}
    H(\alpha, \delta)\ &= \ -2^{-1-\alpha-\delta} (-1)^\alpha \intii \cdots \intii \hphi(x_2) \cdots \hphi(x_{\alpha+\delta+1})\nonumber \\
    &\hspace{1cm}\times \left[ \intii \phi^{n-\alpha-\delta}(x_1) \frac{\sin\left(2\pi x_1 (1 + |x_2| + \cdots + |x_{\alpha + 1}| - |x_{\alpha +2}| - \cdots - |x_{\alpha + \delta + 1}|)\right)}{2\pi x_1}dx_1\right.\nonumber \\
    &\hspace{1cm}- \left.\frac{1}{2}\phi^{n-\alpha-\delta}(0) \right]dx_2\cdots dx_{\alpha+\delta+1} + \converrorterm \label{eq:Gtosin}
\end{align}
as desired. \qed

\subsection{Proof of Proposition~\ref{lem:maintermcalc}: Combinatorial simplification of the main contribution}\label{sec: main term combo}
In this section we finish the proof of Proposition~\ref{lem:maintermcalc} by applying Lemma~\ref{lem:Gtosin} to \eqref{eq:A to F} and simplifying. This step is mostly combinatorial, although we need the following lemma.

\begin{lem}\label{lem:3.49gen}
We have
    \begin{equation}
        \intii \hphi(y)\left(\sin(z+ 2\pi x|y|) + \sin(z- 2\pi x|y|) \right) dy\ = \ 2\sin (z) \phi(x).
    \end{equation}
\end{lem}
\begin{proof}
    Using that $\sin(z+ 2\pi x|y|) + \sin(z- 2\pi x|y|) =  2\sin(z)\cos(2\pi xy)$ we have that
    \begin{align}
        \intii \hphi(y)\left(\sin(z+ 2\pi x|y|) + \sin(z- 2\pi x|y|) \right) dy\ &= \ 2\sin(z) \intii \hphi(y)  \cos(2\pi xy) dy\nonumber \\
        &=\ 2\sin(z) \intii \hphi(y) \Re( \exp{2\pi i xy})   dy\nonumber \\
        &=\ 2\sin(z) \phi(x)
    \end{align}
as desired.
\end{proof}

Applying Lemma~\ref{lem:Gtosin} to \eqref{eq:A to F} gives
\begin{align}
    A \ &= \ \sum_{\alpha = 0}^{a-1} \sum_{\delta = 0}^{a - \alpha - 1} \binom{n}{\alpha + \delta} \binom{\alpha + \delta}{\alpha} \sum_{i=0}^{a - \alpha - \delta - 1} (-1)^i \binom{n - \alpha - \delta}{i} 2^{n-1-\alpha-\delta} (-1)^\alpha\nonumber \\
    &\hspace{1cm}\times \intii \cdots \intii \hphi(x_2) \cdots \hphi(x_{\alpha+\delta+1})\nonumber \\
    &\hspace{1cm}\times \left[ \intii \phi^{n-\alpha-\delta}(x_1) \frac{\sin\left(2\pi x_1 (1 + |x_2| + \cdots + |x_{\alpha + 1}| - |x_{\alpha +2}| -\cdots  - |x_{\alpha + \delta + 1}|)\right)}{2\pi x_1}dx_1\right.\nonumber \\
    &\hspace{2cm}- \left.\frac{1}{2}\phi^{n-\alpha-\delta}(0) \right]dx_2\cdots dx_{\alpha+\delta+1} + \converrorterm . \label{eq:mainterm with sin}
    \end{align}
Our first step is to eliminate the integral over $\phi^{n - \alpha - \delta} (0)$ in \eqref{eq:mainterm with sin} when $\alpha + \delta > 0$. We fix some $\nu \leq a$ and collect the terms of \eqref{eq:mainterm with sin} for which $\alpha + \delta = \nu$:
\begin{align}
    \iffalse&\sum_{\alpha = 0}^\nu \binom{n}{\nu} \binom{\nu}{\alpha} \sum_{i=0}^{a - \nu -1} (-1)^i \binom{n - \nu}{i} 2^{n-1-\nu} (-1)^\alpha\nonumber \\
    & \hspace{1cm}\times\int_{-\infty}^\infty \cdots \intii \hphi(x_2)\cdots \hphi(x_{\nu+1}) \left[-\frac{1}{2}\phi^{n-\nu}(0)\right] dx_2\cdots dx_{\nu+1}\nonumber\\
    =\ \fi&  \left[\sum_{\alpha = 0}^\nu \binom{\nu}{\alpha} (-1)^\alpha\right] \binom{n}{\nu}  \sum_{i=0}^{a - \nu -1} (-1)^i \binom{n - \nu}{i} 2^{n-1-\nu}\nonumber  \\
    & \hspace{1cm}\times\int_{-\infty}^\infty \cdots \intii \hphi(x_2)\cdots \hphi(x_{\nu+1}) \left[-\frac{1}{2}\phi^{n-\nu}(0)\right] dx_2\cdots dx_{\nu+1}. \label{eq: nu ints}
\end{align}
By the Binomial Theorem, we have $\sum_{\alpha = 0}^\nu \binom{\nu}{\alpha} (-1)^\alpha = (1-1)^{\nu} = 0$ for $\nu > 0$, so the sum over $\alpha$ in \eqref{eq: nu ints} is 0 unless $\nu = 0$. Thus, the terms where $\alpha + \delta = \nu$ cancel when $\nu > 0$. When $\nu = 0$, we pull out the $-\frac{1}{2} \phi^n(0)$ term and find that

\begin{align}
    A \ &= \ \sum_{\alpha = 0}^{a-1} \sum_{\delta = 0}^{a - \alpha - 1} \binom{n}{\alpha + \delta} \binom{\alpha + \delta}{\alpha} \sum_{i=0}^{a - \alpha - \delta - 1} (-1)^i \binom{n - \alpha - \delta}{i} 2^{n-1-\alpha-\delta} (-1)^\alpha I(\alpha, \delta)\nonumber \\
    &\hspace{1cm} -2^{n-2}\phi^n(0)\sum_{i=0}^{a - 1} (-1)^i \binom{n}{i}  \,  + \converrorterm  \label{eq:AtoI}
\end{align}

\no where

\begin{align}
    I(\alpha, \delta) \ &= \ \intii \cdots \intii \hphi(x_2) \cdots \hphi(x_{\alpha+\delta+1}) \intii \phi^{n-\alpha-\delta}(x_1) \nonumber\\
    &\times   \frac{\sin\left(2\pi x_1 (1 + |x_2| + \cdots + |x_{\alpha + 1}| - |x_{\alpha +2}| - \cdots - |x_{\alpha + \delta + 1}|)\right)}{2\pi x_1}dx_1\cdots dx_{\alpha+\delta+1}. \label{eq:I def}
    \end{align}
We simplify the first sum over $\alpha$ in \eqref{eq:AtoI}, which we denote by $A'$. By Lemma \ref{lem:3.49gen} we have $I(\alpha, \delta) = 2 I(\alpha, \delta - 1) - I(\alpha + 1, \delta - 1)$. We express $A'$ in terms of $I(\alpha, 0)$ via the following result:
\begin{lem}\label{lem: I to 0}
    Let $I(\alpha, \delta)$ be defined as above. Then 
    \begin{equation}
        I(\alpha, \delta) \ = \ \sum_{j=0}^\delta 2^{\delta - j}(-1)^j \binom{\delta}{j} I(\alpha + j, 0).
    \end{equation}
\end{lem}
\begin{proof}
    We prove the following claim holds by induction, after which setting $k = \delta$ completes the proof of the lemma:
    \begin{equation}
        I(\alpha, \delta) \ = \ \sum_{j=0}^k 2^{k - j}(-1)^j \binom{k}{j} I(\alpha + j, \delta - k).
    \end{equation}
    The base case $k = 0$ holds immediately. Suppose the result holds up to $k$. Then using that $I(\alpha, \delta) = 2 I(\alpha, \delta - 1) - I(\alpha + 1, \delta - 1)$ we have that
    \begin{align}
        I(\alpha, \delta) \ &= \  \sum_{j=0}^k 2^{k - j}(-1)^j \binom{k}{j} (2I(\alpha + j, \delta - k - 1) - I(\alpha + j + 1, \delta - k -1))\nonumber \\
        &= \ \sum_{j=0}^{k+1} 2^{k+1 - j}(-1)^j \left[\binom{k+1}{j+1} -\binom{k}{j}\right]I(\alpha + j, \delta - k -1)\nonumber \\
        &= \ \sum_{j=0}^{k+1} 2^{k+1 - j}(-1)^j \binom{k+1}{j}I(\alpha + j, \delta - k -1)
    \end{align}
    completing the inductive hypothesis and the proof of the lemma.
\end{proof}

\begin{proof}[Proof of Proposition~\ref{lem:maintermcalc}.]
    Apply Lemma \ref{lem: I to 0} to \eqref{eq:AtoI}, we have
\begin{equation}\label{eq:AtoI0}
    \begin{split}
    A' \ &= \ \sum_{\alpha = 0}^{a-1} \sum_{\delta = 0}^{a - \alpha - 1}\sum_{j=0}^\delta \binom{n}{\alpha + \delta} \binom{\alpha + \delta}{\alpha} \sum_{i=0}^{a - \alpha - \delta - 1} (-1)^{\alpha +j+i} \binom{n - \alpha - \delta}{i} 2^{n-1-\alpha-j}    \binom{\delta}{j} I(\alpha + j, 0).
    \end{split}
\end{equation}
The terms are collected according to the values of   $\omega := \alpha + j$. Upon simplification, it follows that
\begin{equation}\label{eq:toomega}
    \begin{split}
    A' \ &= \ 2^{n-1}\sum_{\omega = 0}^{a-1} 2^{-\omega} (-1)^{\omega}I(\omega, 0) \sum_{\alpha = 0}^\omega \sum_{\delta = \omega - \alpha}^{a - 1 - \alpha} \sum_{i= 0}^{a - \delta - \alpha - 1} (-1)^i \binom{n}{\delta +\alpha} \binom{\delta + \alpha}{\alpha} \binom{n - \delta - \alpha}{i} \binom{\delta}{\omega - \alpha}.
    \end{split}
\end{equation}
We then make a change of variables $\delta = \ell + \omega - \alpha$ and rewrite the sums above as 
\begin{align}
    A' \ &= \ 2^{n-1}\sum_{\omega = 0}^{a-1} 2^{-\omega} (-1)^{\omega}I(\omega, 0) \sum_{\alpha = 0}^\omega \sum_{\ell = 0}^{a - 1 - \omega} \sum_{i= 0}^{a - \ell - \omega - 1} (-1)^i \binom{n}{\ell + \omega} \binom{\ell + \omega}{\alpha} \binom{n - \ell - \omega}{i} \binom{\ell + \omega - \alpha}{\omega - \alpha} \nonumber\\
    \ &= \ 2^{n-1}\sum_{\omega = 0}^{a-1} 2^{-\omega} (-1)^{\omega}I(\omega, 0) \sum_{\alpha = 0}^\omega \sum_{\ell = 0}^{a - 1 - \omega} \sum_{i= 0}^{a - \ell - \omega - 1} (-1)^i \binom{n}{\ell + \omega +i} \binom{\ell + \omega +i}{\ell + i} \binom{\omega}{\alpha} \binom{\ell + i }{i}.
\end{align}
Grouping terms based on the values $m = \ell + i$ and rearranging, we have
\begin{equation}\label{eq:groupm}
    \begin{split}
    A' \ &= \ 2^{n-1}\sum_{\omega = 0}^{a-1} 2^{-\omega} (-1)^{\omega}I(\omega, 0) \sum_{\alpha = 0}^\omega \binom{\omega}{\alpha} \sum_{m = 0}^{a - 1 - \omega} \binom{n}{m +\omega} \binom{m+\omega}{m} \sum_{i=0}^m  (-1)^i   \binom{m}{i} .
    \end{split}
\end{equation}
As a consequence of the Binomial Theorem, the sum over $i$ is zero unless $m = 0$.  Thus, summing over $\alpha$ yields $\binom{n}{\omega}2^\omega$ and so, 
\begin{equation}
\begin{split}
    A' \ &= \  2^{n-1}\sum_{\omega=0}^{a-1} \binom{n}{\omega}(-1)^\omega I(\omega, 0).
\end{split}
\end{equation}

Applying this to \eqref{eq:AtoI}, we find that $A = (-1)^{n+1} \mathcal{R}(n, a; \phi) + \converrorterm $, where $\mathcal{R}(n, a; \phi)$ was defined in \eqref{eq: R def}. This completes the proof of Proposition~\ref{lem:maintermcalc}.

\end{proof}

%%%%%%%%%%%%%%%%%%%%%%%%%%%%%%%%%%%%%%%%%%%%%%%%%%%%%%%%%%%%%%%%%%%%%%%%%%%%%%%%%%%%%%%%%%%%%%%%%%%%%%%%%%%%%%%%%%%%%%%%%%%%%%%%%%%%%
%%%%%%%%%%%%%%%%%%%%%%%%%%%%%%%%%%%%%%%%%%%%%%%%%%%%%%%%%%%%%%%%%%%%%%%%%%%%%%%%%%%%%%%%%%%%%%%%%%%%%%%%%%%%%%%%%%%%%%%%%%%%%%%%%%%%%
%%%%%%%%%%%%%%%%%%%%%%%%%%%%%%%%%%%%%%%%%%%%%%%%%%%%%%%%%%%%%%%%%%%%%%%%%%%%%%%%%%%%%%%%%%%%%%%%%%%%%%%%%%%%%%%%%%%%%%%%%%%%%%%%%%%%%
%%%%%%%%%%%%%%%%%%%%%%%%%%%%%%%%%%%%%%%%%%%%%%%%%%%%%%%%%%%%%%%%%%%%%%%%%%%%%%%%%%%%%%%%%%%%%%%%%%%%%%%%%%%%%%%%%%%%%%%%%%%%%%%%%%%%%

\section{Extending support for random matrix theory: Proof of Theorem \ref{thm:fullRMT}}\label{sec:rmt calc}
%%%%%%%%%%%%%%%%%%%%%%%%%%%%%%%%%%%%%%%%%%%%%%%%%%%%
In this section, we prove Theorem \ref{thm:fullRMT}. We focus on the case where $n \ge 3$ as \cite[Theorem 1.7]{HM} have proved the $n = 2$ case.

In Section \ref{subsec:RMTprelims}, we use results from \cite{HR} and \cite{HM} to reduce the proof of Theorem \ref{thm:fullRMT} to proving Proposition \ref{thm:ComputeQnPhi}, which gives a closed form expression for the quantity $Q_n(\phi)$ defined in \eqref{eqn:QnPhi}. 

The rest of the section is dedicated to evaluating $Q_n(\phi)$. In Section \ref{subsec: Qnphi prelims}, we define the notion of a \emph{system of parameters} and of a \emph{$t$-class}, and express $Q_n(\phi)$ as a sum over these two objects in Lemma \ref{thm:Qn_end}. Lemma~\ref{thm:Qn_end} splits $Q_n(\phi)$ into a combinatorial piece and an integral piece. We evaluate the combinatorial piece in Section \ref{subsec: combo part}. In particular, we calculate the contribution to $Q_n(\phi)$ from $t$-classes with $t = 1$ using Lemma \ref{cor:nonoverlapping final coef}, and show that the $t$-classes with $t \ge 2$ do not contribute in Lemma \ref{cor:classes_cancel}. Then, in Section \ref{subsec: final integral}, we evaluate the integral piece, completing the proof of Proposition \ref{thm:ComputeQnPhi}.

\subsection{Proof of Theorem \ref{thm:fullRMT} assuming Proposition \ref{thm:ComputeQnPhi}}\label{subsec:RMTprelims}

We calculate the $n$th-centered moments of $Z(U; \phi)$, denoted in Section \ref{sec: rmt setup} by $\mathcal{Z}_{n}(M; \phi)$, using the method of cumulants. Weyl's explicit representation of Haar measure would allow us to compute the higher moments directly. However, to facilitate the comparison with number theory, we use the cumulants as in \cite{HR} and \cite{HM}. The cumulants $C_{\ell}^{+}(\phi)$ and $C_{\ell}^{-}(\phi)$ are defined to satisfy the following equality of formal power series:
\begin{align}
    \sum_{\ell  = 1}^{\infty}C_{\ell}^{+}(\phi)\frac{\lambda^{\ell}}{\ell !} \ &= \ \lim_{\substack{M \text{ even}\\ M \to \infty}} \log \E_{\text{SO}(M)}[\exp{\lambda Z(U; \phi)}],\\
    \sum_{\ell = 1}^{\infty}C_{\ell}^{-}(\phi)\frac{\lambda^{\ell}}{\ell !} \ &= \ \lim_{\substack{M \text{ odd}\\ M \to \infty}} \log \E_{\text{SO}(M)}[\exp{\lambda Z(U; \phi)}].
\end{align}

Given the first $n$ cumulants, one can compute the first $n$ moments. In particular, for $n \ge 2$, we have that
\begin{align} \label{eqn:CumulantsToMoments}
    \sideset{}{^\pm}{\lim}_{M\to \infty} \  \mathcal{Z}_{n}(M; \phi) \ = \ \sum_{\substack{2k_2 + 3k_3 + \cdots + nk_n = n\\
    k_{j}\geq 0}} \left(\frac{C^\pm_{2} (\phi) }{2!}\right)^{k_{2}} \cdots \left(\frac{C_{n}^\pm (\phi)}{n!}\right)^{k_{n}}\frac{n!}{k_{2}! \cdots k_{n}!}\,.
\end{align}
Now, set $S(x) \coloneqq \frac{\sin(\pi x)}{\pi x}$ and define
\begin{align}
    Q_n(\phi)\ & \coloneqq \ 2^{n - 1} \sum_{m=1}^{n}  \sum_{\substack{\lambda_{1}+\cdots+\lambda_{m} = n\nonumber\\
    \lambda_{j}\geq1}}\frac{(-1)^{m+1}}{m}\frac{n!}{\lambda_{1}!\cdots\lambda_{m}!} \intii \cdots \intii \phi(x_1)^{\lambda_1} \cdots \phi(x_m)^{\lambda_m} \\
     & \hs{1} \times S(x_{1}-x_{2}) S(x_2 - x_3) \cdots S(x_{m-1} - x_m) S(x_m + x_1) dx_1 \cdots dx_m.\label{eqn:QnPhi}
 \end{align}
We have the following result due to \cite{HR}.
\begin{lem}[\cite{HR}, Section 2.1] \label{thm:nthCumulant}
     Let $\phi\in \mathcal{S}_{ec}(\R)$ with $\supp(\hphi) \subseteq \twonsquare$.
     For $n \geq 3$,
     \begin{align}
     C_{n}^{\soe}(\phi) &\ = \  Q_{n}(\phi) \nonumber \\
     C_{n}^{\soo}(\phi) &\ = \  -Q_{n}(\phi).
     \end{align}
     Moreover, for $n\geq 4$,
     \begin{align}\label{2ndCumulant}
     C_{2}^{\soe} \ = \  C_{2}^{\soo} \ = \ 2 \intii|y|\hphi(y)^{2} dy \ = \ \sigma^2_\phi
     \end{align}
     where $\sigma^2_\phi$ is defined as in \eqref{eq:fullNT_var}.
\end{lem}
\no Thus, in order to prove Theorem~\ref{thm:fullRMT}, it suffices to calculate $Q_n(\phi)$, which we do in the following proposition.

\begin{prop}\label{thm:ComputeQnPhi}
    Let $\phi\in \mathcal{S}_{ec}(\R)$ with $\supp (\hphi) \subseteq \nasquare $ for some integer $1\le a\le \lceil n/2 \rceil$.  Let $R(n, a; \phi)$ be as in \eqref{eq: R def}. Then
    \begin{align}
        Q_n(\phi) \ =\ \mathcal{R}(n, a; \phi).
    \end{align}
\end{prop}
We complete the proof of Proposition \ref{thm:ComputeQnPhi} in Section \ref{subsec: final integral}.
Assuming Proposition~\ref{thm:ComputeQnPhi}, we now prove Theorem~\ref{thm:fullRMT}.
\begin{proof}[Proof of Theorem \ref{thm:fullRMT}]
By \cite[Theorem~1.4]{HM}, since $\supp(\hphi) \subseteq \left[ -\frac{1}{n-a}, \frac{1}{n-a} \right]$, then $C_j^\pm(\phi) = 0$ for $3 \le j \le n-a$, as in this case the cumulants are the same as those of the Gaussian. Hence, restricting the sum in \eqref{eqn:CumulantsToMoments} to those terms with $k_3=\cdots=k_{n-a}=0$ does not change its value. Moreover, $a \le \lceil n/2\rceil$ and $\sum_{\ell=2}^{n}\ell k_{\ell} = n$ imply that $ k_{n}, k_{n-1}, \dots, k_{n-a+1} \in\{0,1\}$ and at most one of $k_{n}, k_{n-1}, \dots, k_{n-a+1}$ is equal to $1$. Thus, we can rewrite \eqref{eqn:CumulantsToMoments} as
\begin{equation} 
     \sideset{}{^\pm}{\lim}_{M\to \infty} \  \mathcal{Z}_{n}(M; \phi) \ =\ \mathbbm{1}_{\{n \text{ even}\}} \left(\frac{C^\pm_2(\phi)}{2}\right)^{n/2}\frac{n!}{(n/2)!} + 
     \sum_{\substack{k_2, \ell \\ 2k_2+(n-\ell)=n \\ 0 \leq \ell \leq a-1}} \left(\frac{C_{2}^\pm (\phi) }{2!}\right)^{k_{2}}\left(\frac{C_{n-\l}^\pm (\phi)}{(n-\l)!}\right)\frac{n!}{k_{2}!}
\end{equation}
where the first term is from when $k_2 = n/2$.
Observing that $2k_2 + (n-\ell) = n$ forces $\ell=2k_2$, we have
\begin{align}
    \sideset{}{^\pm}{\lim}_{M\to \infty} \  \mathcal{Z}_{n}(M; \phi) &\ = \
    \mathbbm{1}_{\{n \text{ even}\}} \left(C_{2}^{\pm}(\phi)\right)^{n/2}(2n-1)!! + \sum_{k_2=0}^{\lfloor\frac{a-1}{2}\rfloor} \frac{n! C_{n-2k_2}^{\pm}(\phi)}{k_2!(n-2k_2)!}\left( \frac{C_2^\pm(\phi)}{2}\right)^{k_2} .\label{eqn:nthMoment}
\end{align}
Now, applying Lemma~\ref{thm:nthCumulant} and Proposition~\ref{thm:ComputeQnPhi} to the right hand side of \eqref{eqn:nthMoment} and simplifying completes the proof of Theorem~\ref{thm:fullRMT} after comparing with \eqref{eq: Snaphi def}. 
\end{proof}

%%%%%%%%%%%%%%%%%%%%%%%%%%%%%%%%%%%%%%%%%%%%%%%%%%%%%%%%%%%%%%%%%%%%%%%%%%%%%%%%%%%%%%%%%%%%%%%%%%%%%%%%%%%%%%%%%%%%%%%%%%%%%%%%%%%%%

\subsection{Decomposition of \texorpdfstring{$Q_n(\phi)$}{Q}}\label{subsec: Qnphi prelims}

In this section, we work towards Proposition~\ref{thm:ComputeQnPhi} by evaluating $Q_n(\phi)$, as defined in \eqref{eqn:QnPhi}, when $\supp(\hphi)\subseteq \nasquare$ and $a \le \lceil n/2 \rceil$. The main result of this subsection is Lemma~\ref{thm:Qn_end}, which splits $Q_n(\phi)$ into a combinatorial term and an integral term which we will then evaluate separately. 

Equation (5.27) of \cite{HM} gives (independent of the choice of support) that
\begin{equation}\label{eq:Qn}
	Q_n(\phi)\ =\  2^{n-2} \int_0^\infty \cdots \int_0^\infty \hat{\phi}(y_1)\cdots\hat{\phi}(y_n)K(y_1,\ldots,y_n)dy_1\cdots dy_n,
\end{equation}
where
\begin{equation}\label{eq:Ky}
	K(y_{1},\ldots,y_{n})\ \coloneqq\ \sum_{m=1}^{n}\sum_{\substack{\lambda_{1}+\ldots+\lambda_{m}=n\\
			\lambda_{j}\geq1}
	}\frac{(-1)^{m+1}}{m}\frac{n!}{\lambda_{1}!\cdots\lambda_{m}!}\sum_{\epsilon_{1},\ldots,\epsilon_{n} \in \{\pm 1\}}\prod_{\ell=1}^{m}\indicator_{\left\{ \left|\sum_{j=1}^{n}\eta(\ell,j)\epsilon_{j}y_{j}\right|\leq1\right\} }
\end{equation}
and
\begin{equation}\label{eq:eta}
	\eta(\ell,j) \ \coloneqq\  \begin{cases} +1 &\text{if } j\leq\sum_{k=1}^{\ell}\lambda_{k} \\
	-1 &\text{if } j>\sum_{k=1}^{\ell}\lambda_{k}. \end{cases}
\end{equation}

%%%%%%%%%%%%%%%%%%%%%%%%%%%%%%%%%%%%%%%%%%%%%%%%%%%%%%%%%%%%%%%%%%%%%%%%%%%%%%%%%%%%%%%%%%%%%%%%%%%%%%%%%%%%%%%%%%%%%%%%%%%%%%%%%%%%%%%%%%%%%%%%%%%%%%%%%%%%%%%%%%%%%%%%%%%%%%%%%%%%%%%%%%%%%%%%%%%%%%%%%%%%%%%%%%%%%%%%%%%%%%%%%%%%%%%%%%%%%%%%%%%%%%%%%%%%%%%%%%%%%%%%%%%%%%%%%%%%%%%%%%%%%%%%%%%%%%%%%%

\subsubsection{Simplifying $K(y_1,\dots,y_n)$}
To evaluate $Q_n(\phi)$, we first discuss how we will interpret the expression $K(y_1,\dots,y_n)$ for $y_1,\dots,y_n \in \zerona$. 

\begin{defn}

Throughout this section, if $I \subseteq \{1,\dots,n\}$, we write 
\begin{equation}\label{eq: chi I def}
	\indicator_{I} \ =\  \indicator_{\{y_1+\dots+y_n > 1 + 2\sum_{i\in I} y_i\}}.
\end{equation}

\end{defn}

\begin{defn}
	A \textbf{\emph{system of parameters}} (or \textbf{\emph{\sop{}}}) is an ordered tuple $(m,\la_1,\dots,\la_m,\epsilon_1,\dots,\epsilon_n)$ with $1\le m \le n$, $\la_1+\dots+\la_m = n$, $\la_i \ge 1$ for all $1\le i\le m$, and $\epsilon_j = \pm 1$ for each $1\le j \le n$. 
\end{defn}

Given a system of parameters $S$, we may use $\eta_S(\ell, j)$ to denote the function $\eta(\ell, j)$ where the $\la_k$ are taken from $S$. When it is clear from context that the $\la_k$ are taken from the \sop{} $S$, we simply denote this function $\eta(\ell, j)$. Fix $n \ge 2a$ and a \sop{} $S =(m,\la_1,\dots,\la_m,\epsilon_1,\dots,\epsilon_n)$. Consider the product
\begin{equation}\label{eq:theproduct}
	\prod_{\ell=1}^{m}\indicator_{\{|\sum_{j=1}^n \eta(\ell,j)\epsilon_jy_j| \le 1\}}
\end{equation}
from \eqref{eq:Ky}. Fix $1\le \ell_0 \le m$. In order to study \eqref{eq:theproduct}, we study the complement of the indicator functions in \eqref{eq:theproduct}, given by
\begin{equation}\label{eq:vanishing_chi}
	\indicator_{\{|\sum_{j=1}^n \eta(\ell_0,j)\epsilon_jy_j| > 1\}}.
\end{equation}

For $y_1, \ldots, y_n \in \zerona$, if $\sum_{j=1}^n \eta(\ell_0,j)\epsilon_jy_j > 1$ then we cannot find $y'_1, \ldots, y'_n \in \zerona$ such that $\sum_{j=1}^n \eta(\ell_0,j)\epsilon_jy_j < -1$ because $a \le \lceil n/2 \rceil$. Thus the indicator function \eqref{eq:vanishing_chi} is identical to \eqref{eq: chi I def} for a particular choice of $I$. Moreover, there exists $y_i \in \zerona$ such that \eqref{eq:vanishing_chi} is nonzero if and only if one of the following (mutually exclusive) conditions holds:
\begin{enumerate}
	\item[(i)] $|\{1\le j \le n : \eta (\ell_0, j) \epsilon_j = +1 \}| \le a-1$, or
	\item[(ii)] $|\{1\le j \le n : \eta (\ell_0, j) \epsilon_j = -1 \}| \le a-1$\,.
\end{enumerate}
If case (i) holds, we define
\begin{equation}
J_{\ell_0} \coloneqq \{1\le j \le n : \eta(\ell_0,j)\epsilon_j = +1\}\,
\end{equation}
and say that $J_{\ell_0}$ has sign $\zeta_{\ell_0} = +1$.\\
If case (ii) holds, we define
\begin{equation}
J_{\ell_0} \coloneqq \{1\le j \le n : \eta(\ell_0,j)\epsilon_j = -1\}\,
\end{equation}
and say that $J_{\ell_0}$ has sign $\zeta_{\ell_0} = -1$.\\
If neither case holds, then $J_{\ell_0}$ is undefined. 

\begin{lem}\label{lem:there_is_only_one}
	If $S =(m,\la_1,\dots,\la_m,\epsilon_1,\dots,\epsilon_n)$ is a system of parameters and $J \subseteq [1,n]$ is any subset,  then there is at most one $\ell_0 \in [1,m]$ and $\zeta \in \{\pm 1\}$ such that $\eta(\ell_0,i)\epsilon_i = \zeta$ for $i \in J$ and $\eta(\ell_0,j)\epsilon_j = -\zeta$ for $j\notin J$.
\end{lem}
\begin{proof}
	Suppose $\ell_1 > \ell_0$ and that both $\ell_0$ and $\ell_1$ have this property for some $\zeta_0$ and $\zeta_1$. Without loss of generality, we assume that $J = \{i : \eta(\ell_0,i)\epsilon_i = -1\}$. It is clear that we cannot also have $I = \{i : \eta(\ell_1,i)\epsilon_i = -1\}$, so we may assume that $I = \{i : \eta(\ell_1,i)\epsilon_i = +1\}$, but then we must have $\eta(\ell_0,j) = -\eta(\ell_1,j)$ for all $j$, and this is clearly impossible. 
\end{proof}

In particular, if $J_{\ell_{0}}$ and $J_{\ell_{1}}$ are both defined, then $J_{\ell_0} \neq J_{\ell_1}$.

\begin{defn}
	For a \sop{} $S =(m,\la_1,\dots,\la_m,\epsilon_1,\dots,\epsilon_n)$, let  $\{\ell_1,\dots,\ell_t\}\subseteq \{1,\dots,m\}$ be the set of indices for which $I_{\ell_j}$ is defined. Define
	\begin{equation}
	    J(S) \ \coloneqq\  \{{J_{\ell_1}},\dots,{J_{\ell_t}}\}\,.
	\end{equation}
	Define 
	\begin{equation}
		I(S) \ \coloneqq\  \{{I_1},\dots,{I_r}\}
	\end{equation}
	to be the subset of elements of $J(S)$ which are minimal with respect to inclusion. That is, $I(S)$ consists of those elements of $J(S)$ which do not strictly contain any other elements of $J(S)$. By Lemma~\ref{lem:there_is_only_one}, for each $i \in [1,r]$ there is a unique $\ell_{i}$ such that $I_i = J_{\ell_i}$.
	Finally, define	the function
	\begin{equation}\label{eq: sigma def}
		\sigma^S(y_1,\dots,y_n)\  \coloneqq\  \sum_{i=1}^{r} \sum_{1\le j_1<\dots<j_i\le r} (-1)^i (\indicator_{I_{j_1}}\cdots \indicator_{I_{j_i}})(y_1,\dots,y_n),
	\end{equation}
	and the quantity
	\begin{equation}\label{eq: AS def}
		A(S) \ \coloneqq\  \frac{(-1)^{m+1}}{m}\frac{n!}{\lambda_{1}!\cdots\lambda_{m}!}.
	\end{equation}
\end{defn}

% Then the product \eqref{eq:theproduct} vanishes at $(y_1,\dots,y_n)$ if and only if at least one of $\indicator_{I_{\ell_1}},\dots,\indicator_{I_{\ell_t}}$ is supported at $(y_1,\dots,y_n)$. Moreover, the product \eqref{eq:theproduct} vanishes at $(y_1,\dots,y_n)$ if and only if at least one of the \textit{minimal elements} of $\indicator_{I_{\ell_1}},\dots,\indicator_{I_{\ell_t}}$ in $\Omega$ is supported at $(y_1,\dots,y_n)$. Let $\{\indicator_{I_1},\dots,\indicator_{I_r}\}$ denote the \textit{multiset} of those minimal elements, each appearing as many times as they appear above.
% Lemma \ref{lem:there_is_only_one} implies for an \sop{} $S$, each element of $I(S)$ occurs once, so $I(S)$ is simply a set. 
The next lemma for $\sigma^S$ resembles the   M\"{o}bius inversion formula from elementary number theory. 

\begin{lem}\label{lem:sigma}
	For any \sop{} $S$, we have 
	\begin{equation}
		\sigma^S(y_1,\dots,y_n) \ =\  \begin{cases}
			-1 &\mbox{\emph{ if }} \indicator_I(y_1,\dots,y_n) = 1 \mbox{\emph{ for some }} I\in I(S) \\
			0 &\mbox{\emph{ otherwise.}}
		\end{cases}
	\end{equation}
\end{lem}
\begin{proof}
	Fix $(y_1,\dots,y_n)$. Suppose there are $k$ elements in $I(S)$ whose support contains $(y_1, \dots, y_n)$. If $k= 0$ the result is immediate. Now, for $k \ge 1$ and $1\le i \le k$, there are $\binom{k}{i}$ terms in the $i$th summand of with coefficient $(-1)^i$ and all the other terms vanish. Thus we have
	\begin{align}
		\sigma^S(y_1,\dots,y_n) &\ =\  \sum_{i=1}^k \binom{k}{i} (-1)^i   = (1-1)^k-1 = -1. 
	\end{align}
	\iffalse
	Without loss of generality, assume that the first $k$ elements of $I(S)$ have this property; that is, $\{I\in I(S) : \indicator_I(y_1,\dots,y_n) = 1 \} = \{I_1,\dots,I_k\}$. If $k=0$, then the claim is clear. If $k\ge 1$, then  
	\begin{align}
		\sigma^S(y_1,\dots,y_n) &\ =\  \sum_{i=1}^k \sum_{1\le j_1<\dots<j_i \le k} (-1)^i (\indicator_{I_{j_1}}\cdots\indicator_{I_{j_i}})(y_1,\dots,y_n) \nonumber \\
		&\ =\  \sum_{i=1}^k \binom{k}{i} (-1)^i   = (1-1)^k-1 = -1. 
	\end{align}
	\fi
\end{proof}
\iffalse
\no Note that Lemma \ref{lem:sigma} and the preceding discussion imply that \eqref{eq:theproduct} vanishes exactly when $\sigma^S(y_1,\dots,y_n) =-1$.
\fi

\no We now have the following reformulation of the quantity from \ref{eq:Ky} in terms of $A(S)$, defined in \ref{eq: AS def}.
\begin{lem}\label{prop:K_simplified}
	For $(y_1,\dots,y_n) \in \zerona^n$,
	\begin{equation}\label{eq:K_simplified}
		K(y_1,\dots,y_n)\  =\  \sum_{t=1}^{n} (-1)^t \sum_{\substack{(I_1,\dots,I_t) \\ \mbox{\emph{\tiny valid}}}} (\indicator_{I_1}\cdots \indicator_{I_t})(y_1,\dots,y_n) \sum_{\substack{\mbox{\emph{\tiny \sop{} }} S \mbox{\emph{ \tiny with}} \\ {I_1},\dots,{I_t} \in I(S)}} A(S).
	\end{equation}
\end{lem}
\begin{proof}
    The product \eqref{eq:theproduct} vanishes at $(y_1,\dots,y_n)$ if and only if there is some $J \in J(S)$ such that $\chi_{J}$ is supported at $(y_1,\dots,y_n)$ if and only if there is some $I \in I(S)$ such that $\chi_{I}$ is supported at $(y_1,\dots,y_n)$. So, by Lemma~\ref{lem:sigma},
    \begin{equation}\label{eq:indicator-to-sigma}
        \prod_{\ell=1}^{m}\indicator_{\{|\sum_{j=1}^n \eta(\ell,j)\epsilon_jy_j| \le 1\}}(y_1,\dots,y_n) = 1 +  \sigma^S(y_1, \cdots, y_n)\,.
    \end{equation}
    Substituting \eqref{eq:indicator-to-sigma} into \eqref{eq:Ky}, 
%	Fixing some $m, \lambda_1, \ldots, \lambda_m$ in the summand in \eqref{eq:Ky}, there are $2^n$ total choices of $\epsilon_1, \ldots, \epsilon_n$. For each choice of $\epsilon_1, \ldots, \epsilon_n$, the summand will be 1 if $\sigma^S(y_1, \ldots, y_n) = 0$ and 0 if $\sigma^S = -1$ by Lemma \ref{lem:sigma}. Thus 
    we have that
	\begin{align}\label{eq: Ky pre sosh}
		K(y_1,\dots,y_n) &\ =\  \sum_{m=1}^{n}\sum_{\substack{\lambda_{1}+\ldots+\lambda_{m}=n\\ \lambda_{j}\geq1}}\frac{(-1)^{m+1}}{m}\frac{n!}{\lambda_{1}!\cdots\lambda_{m}!} \cdot 2^n + \sum_{\mbox{\tiny \sop{}'s } S} A(S) \sigma^S(y_1,\dots,y_n).
	\end{align}
    Now, we apply the following identity given by Soshnikov \cite{So}:
\begin{equation}\label{eq: sosh id}
    z \ =\  \log(1 + (e^z-1))\  =\ \sum_{n=1}^\infty z^n \sum_{m=1}^n \sum_{\substack{\lambda_1 + \cdots + \lambda_m = n \\ \lambda_j \ge 1}} \frac{(-1)^{m + 1}}{m} \frac{1}{\lambda_1 ! \cdots \lambda_m !}
\end{equation}
which gives that the first sum in \eqref{eq: Ky pre sosh} is 0. Expanding the second sum using the definition of $\sigma^S$ from \eqref{eq: sigma def} and rearranging completes the proof.
	\iffalse
	\begin{align}
		K(y_1,\dots,y_n) &\ =\  \sum_{\mbox{\tiny \sop{}'s } S} A(S) \sum_{i=1}^{|I(S)|} \sum_{{I_1},\dots,{I_i}\in I(S)} (-1)^i (\indicator_{I_1}\cdots\indicator_{I_i})(y_1,\dots,y_n) \nonumber \\
		&\ =\  \sum_{t=1}^{n} (-1)^t \sum_{\substack{(I_1,\dots,I_t) \\ \mbox{{\tiny valid}}}} (\indicator_{I_1}\cdots \indicator_{I_t})(y_1,\dots,y_n) \sum_{\substack{\mbox{{\tiny \sop{}'s }} S \mbox{{ \tiny with}} \\ {I_1},\dots,{I_t} \in I(S)}} A(S)
	\end{align}
	as desired.
	\fi
\end{proof}

%%%%%%%%%%%%%%%%%%%%%%%%%%%%%%%%%%%%%%%%%%%%%%%%%%%%%%%%%%%%%%%%%%%%%%%%%%%%%%%%%%%%%%%%%%%%%%%%%%%%%%%%%%%%%%%%%%%%%%%%%%%%%%%%%%%%%%%%%%%%%%%%%%%%%%%%%%%%%%%%%%%%%%%%%%%%%%%%%%%%%%%%%%%%%%%%%%%%%%%%%%%%%%%%%%%%%%%%%%%%%%%%%%%%%%%%%%%%%%%%%%%%%%%%%%%%%%%%%%%%%%%%%%%%%%%%%%%%%%%%%%%%%%%%%%%%%%%%%%

\subsubsection{Simplifying $Q_n(\phi)$}
In this section, we simplify $Q_n(\phi)$ by applying Lemma \ref{prop:K_simplified} to the quantity $Q_n(\phi)$ as in \eqref{eq:Qn}. First, we define further notation which allows us to express $Q_n(\phi)$ (through Lemma \ref{thm:Qn_end}) in terms of combinatorial quantities which we then compute in Section \ref{subsec: t1 class} and \ref{subsec: t2 class}.

The symmetric group $S_n$ acts naturally on sets of (unordered) $t$-tuples of subsets of $[1,n]$ by permuting the elements in each subset of each tuple. Take such a $t$-tuple $(I_1, \ldots, I_t)$ and some $I_j = \{i_1, \ldots, i_k\}$. Given some permutation $\tau \in S_n$, we have that $\tau(I_j) = \{\tau(i_1), \ldots, \tau(i_k) \}$.
\iffalse
For instance, if $\tau = (1 2) \in S_3$, then $\tau\left(\{1, 2\}, \{2, 3\}\right) = \left(\{2, 1\}, \{1, 3\}\right)$. 
\fi
Let $\indicator_{I_1}\cdots\indicator_{I_t}$ and $\indicator_{J_1}\cdots\indicator_{J_t}$ be elements of $\Omega$ such that there exists a permutation $\tau \in S_n$ so that for each $1\le \ell \le t$, $\tau(I_\ell) = J_\ell$. Then,
\begin{equation}\label{eq: C is well defined}
	\int_0^\infty \cdots \int_0^\infty \hat{\phi}(y_1)\cdots\hat{\phi}(y_n)(\indicator_{I_1}\cdots\indicator_{I_t} - \indicator_{J_1}\cdots\indicator_{J_t})(y_1, \dots, y_n)dy_1\cdots dy_n = 0.
\end{equation}

\no This motivates the following definition.
\begin{defn}
	The symmetric group $S_n$ acts naturally on sets of (unordered) $t$-tuples of subsets of $[1,n]$, as described above. An orbit of this action is called a \textbf{\emph{$t$-class}}.  %\st{A \emph{$t$-class}} 
\end{defn}
\iffalse
\begin{ex}
	The set $\{{(\{i,j\})} : 1\le i < j \le n\}$ is a 1-class, but $\{ {(\{i\})} : 1\le i \le n-1 \}$ is not because $(\{n\})$ is in the same orbit as $(\{1\})$ but not in the set. Furthermore, the set $C = \{(\{i,j\},\{j,k\}) : 1\le i<j<k \le n\}$ is a $2$-class, but $\{(\{i,j\},\{j,k\}), (\{i,j\},\{k,\ell \}) : 1\le i<j<k<\ell \le n\}$ is not because it contains $C$ as a proper subset.
	% , and $\{ {(\{i\})},\,{(\{i,j\})} : 1\le i < j \le n \}$ is not because $(\{1\})$ and $(\{1,2\})$ are not in the same orbit.
\end{ex}
\fi

For a $t$-class $C$, we define
\begin{equation}\label{eq: int C def}
	\int\,  C \, dy \ \coloneqq \ \int_0^\infty \cdots \int_0^\infty \hat{\phi}(y_1)\cdots\hat{\phi}(y_n)\indicator(y_1, \dots, y_n) \; dy_1\cdots dy_n,
\end{equation}
where $\indicator \coloneqq \indicator_{I_1}\cdots\indicator_{I_t}$ with $(I_1,\dots,I_t) \in C$. Equation \eqref{eq: C is well defined} shows that the integral $\int\,  C \, dy$ is well-defined. %\st{$(I_1,\dots,I_t)$ any element of $C$}. 

\begin{defn}
	We call an unordered tuple $(I_1,\dots,I_t)$ of subsets of $\{1,\dots,n\}$ \textbf{\emph{valid}} if $I_1,\dots,I_t \in I(S)$ for some \sop{} $S$ and $\indicator_{I_1}\cdots\indicator_{I_t}$ is supported at some point in $\zerona^n$.
\end{defn}

We can extend this definition to a $t$-class.

\begin{defn}
	We call a $t$-class \textbf{\emph{valid}} if it contains at least one valid tuple.
\end{defn}

%For an \sop{} $S = (m,\la_1,\dots,\la_m,\epsilon_1,\dots,\epsilon_n)$, we define $\la(S) = m$. 

\no We are now ready to prove the main result of the section.
\begin{lem}\label{thm:Qn_end} 
For a \sop{} $S$ and a $t$-class $C$, set
	\begin{equation}\label{eq: TSC def}
		T(S,C) \ \coloneqq \  \#\left\{(I_1,\dots,I_t) \in C : I_1,\dots,I_t \in I(S) \right\}.
	\end{equation}
We have
	\begin{equation}\label{eq:Qn_end}
		Q_n(\phi) \ =\ 2^{n-2} \sum_{t=1}^n (-1)^t \sum_{\substack{ \mbox{\emph{\tiny valid $t$-classes }} C} } \left( \sum_{\substack{ \mbox{\emph{\tiny \sop{}'s }} S}}  T(S,C) \, A(S)  \right) \int C dy.
	\end{equation}
\end{lem}
\begin{proof}
Given a valid $t$-class $C$, there is a valid tuple $(I_1,\dots,I_t) \in C$ for which $\indicator_{I_1} \cdots \indicator_{I_t}$ is supported at some point $(y_1,\dots,y_n)\in \nasquare^n$. Therefore, if $\tau\in S_n$, then $\indicator_{\tau(I_1)}\cdots\indicator_{\tau(I_t)}$ is supported at $(y_{\tau(1)},\dots,y_{\tau(n)})$. Since $S_n$ acts transitively on $C$, this means that every tuple in $C$ is valid. Now, applying Lemma~\ref{prop:K_simplified} to \eqref{eq:Qn} and grouping tuples into $t$-classes completes the proof.
\end{proof}

%%%%%%%%%%%%%%%%%%%%%%%%%%%%%%%%%%%%%%%%%%%%%%%%%%%%%%%%%%%%%%%%%%%%%%%%%%%%%%%%%%%%%%%%%%%%%%%%%%%%%%%%%%%%%%%%%%%%%%%%%%%%%%%%%%%%%%%%%%%%%%%%%%%%%%%%%%%%%%%%%%%%%%%%%%%%%%%%%%%%%%%%%%%%%%%%%%%%%%%%%%%%%%%%%%%%%%%%%%%%%%%%%%%%%%%%%%%%%%%%%%%%%%%%%%%%%%%%%%%%%%%%%%%%%%%%%%%%%%%%%%%%%%%%%%%%%%%%%%

%%%%%%%%%%%%%%%%%%%%%%%%%%%%%%%%%%%%%%%%%%%%%%
%%%%%%%%%%%%%%%%%%%%%%%%%%%%%%%%%%%%%%%%%%%%%%
%%%%%%%%%%%%%%%%%%%%%%%%%%%%%%%%%%%%%%%%%%%%%%
%%%%%%%%%%%%%%%%%%%%%%%%%%%%%%%%%%%%%%%%%%%%%%
%%%%%%%%%%%%%%%%%%%%%%%%%%%%%%%%%%%%%%%%%%%%%%

\subsection{Computing the combinatorial piece}\label{subsec: combo part}

In this section, we calculate $ \sum T(S,C) \, A(S) $ for valid $t$-classes $C$, where $T(S, C)$ and $A(S)$ are defined as in \eqref{eq: TSC def} and \eqref{eq: AS def}, respectively. In Section \ref{subsec: t1 class} we find a closed form for the case $t = 1$, and then in Section \ref{subsec: t2 class} we show that when $t \ge 2$ the quantity vanishes.

\subsubsection{Computing for valid $1$-classes}\label{subsec: t1 class}

In this section, we compute the terms in \eqref{eq:Qn_end} for which $t=1$. We first classify the valid $1$-tuples.

\begin{lem}\label{lem:2_sets_support}
	If $I$ and $J$ are subsets of $[1,n]$ such that $|I\cup J| \ge a$, then $\indicator_I\cdot\indicator_J$ is identically zero on $\zerona^n$.
\end{lem}
\begin{proof}
	Let $I$ and $J$ be as in the hypotheses, and assume for contradiction that both $y_1 + \dots + y_n \ > \ 1 + 2\sum_{i\in I}y_i $ \ \ and\  \ $
		y_1 + \dots + y_n \ > \ 1 + 2\sum_{j\in J}y_j$
		\iffalse
	\begin{gather}
		y_1 + \dots + y_n \ > \ 1 + 2\sum_{i\in I}y_i     \hspace{10pt}\mbox{and} \label{eq:subtract I} \\
		y_1 + \dots + y_n \ > \ 1 + 2\sum_{j\in J}y_j, \label{eq:subtract J}
	\end{gather}
	\fi
	\ \ for some $(y_1,\dots,y_n) \in \zerona^n$. Since $y_h \le \frac{1}{n-a}$ for every $h$, $\sum_{h\notin I\cup J} y_h \le 1$, so we must have that $\sum_{i\in I} y_i  < \sum_{j\in J \setminus I} y_j$  and similarly $\sum_{j\in J} y_j < \sum_{i \in I \setminus J} y_i$ by our assumptions. Adding these inequalities gives 
	\begin{equation}
		\sum_{i\in I} y_i + \sum_{j\in J} y_j\  <\  \sum_{i \in I \setminus J} y_i + \sum_{j \in J \setminus I} y_j,
	\end{equation}
	which is a contradiction, as all the $y_h$'s are nonnegative and the terms on the right are a subset of those on the left.
\end{proof}

\begin{lem}\label{lem:classify_1-tuples}
	If $I$ is a subset of $[1,n]$, then the $1$-tuple $(I)$ is valid if and only if $|I|\le a-1$. 
\end{lem}
\begin{proof}
	If $|I| > a-1$, then $(I)$ is not valid by Lemma \ref{lem:2_sets_support}, taking both subsets to be $I$. 
	
	Now suppose $|I| \le a - 1$.  Let $y_j = 1/(n-a)$ for each $j\notin I$ and let $y_i  = 0$ for each $i \in I$. It is clear that $\indicator_I(y_1,\dots,y_n) = 1$. Now consider the system of parameters $S=(m,\la_1,\dots,\la_m,\epsilon_1,\dots,\epsilon_n)$, where $m = 1$, $\la_1 = n$, and $\epsilon_i = -1$ if and only if $i\in I$. Clearly, $I\in I(S)$. Therefore, $(I)$ is valid. 
\end{proof}

\no It follows from Lemma \ref{lem:classify_1-tuples} that the valid $1$-classes are exactly the classes 
\begin{equation}
	C_f \coloneqq \{(I) : I \subseteq [1,n], |I| = f\}
\end{equation}
with $0\le f \le a-1$.

\begin{lem}\label{lem:good_placements t = 1}
	Let $1\le f \le a-1$. Let $S=(m,\la_1,\dots,\la_m,\epsilon_1,\dots,\epsilon_n)$ be a system of parameters with $m\ge 2$ and suppose $(I)\in C_f$ is such that, for some $1 \leq \ell \leq m$, we have $I = J_{\ell} \in J(S)$. %$\eta(\ell,i)\epsilon_i = \pm 1$ for $i \in I$ and $\eta(\ell,j)\epsilon_j = \mp 1$ for $j\notin I$. 
	Define $\Lambda_\ell \coloneqq \la_1+\dots+\la_{\ell}$. Then $I\in I(S)$ if and only if $[\Lambda_{\ell - 1}+1, \Lambda_\ell] \not\subseteq I$ and  $[\Lambda_{\ell}+1, \Lambda_{\ell+1}] \not\subseteq I$. If $\ell = m$, then we set $[\Lambda_{m}+1, \Lambda_{m+1}]$ to $[1, \Lambda_1] = [1, \la_1]$.
	
	%$[x_0 - \la_{\ell_0}+1, x_0]$ and $[x_0+1, x_0+\la_{\ell_0+1}]$ consists entirely of elements of $I$, where if $\ell_0 = m$, then we take $[x_0+1, x_0+\la_{\ell_0+1}]$ to mean $[1,\la_1]$. 
\end{lem}
\begin{proof}
    Assume without loss of generality that $J_{\ell}$ has sign $\zeta_\ell = -1$, i.e. $J_{\ell}= \{j : \eta(\ell, j) \epsilon_j = -1\}$. %$I = \{j : \eta(\ell, j) \epsilon_j = -1\}$. 
    
    For any $\ell' < \ell$, we have
    \begin{equation} \label{eq:eta-epsilon}
        \eta(\ell,j)\epsilon_j = \begin{cases} -\eta(\ell',j) \epsilon_j & \text{ if } j \in [\Lambda_{\ell'} + 1, \Lambda_{\ell}]\,, \\
        \eta(\ell',j) \epsilon_j & \text{ if } j \notin [\Lambda_{\ell'} + 1, \Lambda_{\ell}]\,.
        \end{cases}
    \end{equation}
    
%    \begin{equation}
%        \{ j : \eta(\ell',j) \epsilon_j = -1\} = \{ j : \eta(\ell,j) \epsilon_j = -1\} \smallsetminus [\Lambda_{\ell'} + 1, \Lambda_{\ell}] \cup [\Lambda_{\ell'} + 1, \Lambda_{\ell}]\smallsetminus \{ j : \eta(\ell,j) \epsilon_j = -1\}
%    \end{equation}
    
    If $[\Lambda_{\ell - 1}+1, \Lambda_\ell] \subseteq J_{\ell}$, then $J_{\ell-1} = J_{\ell} \smallsetminus [\Lambda_{\ell - 1}+1, \Lambda_\ell] \subsetneq J_{\ell}$. In particular, $J_{\ell}$ is not minimal, so $J_{\ell} \notin I(S)$. Similarly, if $[\Lambda_{\ell}+1, \Lambda_{\ell+1}] \subseteq J_{\ell}$ then $J_{\ell+1} = J_{\ell} \smallsetminus [\Lambda_{\ell}+1, \Lambda_{\ell+1}]$ so $J_{\ell}$ is not minimal.

    %First suppose $[\Lambda_{\ell - 1}+1, \Lambda_\ell] \subseteq I$. Then $I' \coloneqq \{j : \eta(\ell - 1, j) \epsilon_j = -1\}$ is a subset of $I$ so $I' \le I$ so $I$ is not minimal and $I \notin I(S)$. Similarly, if $[\Lambda_{\ell}+1, \Lambda_{\ell+1}] \subseteq I$ then $I'' \coloneqq \{j : \eta(\ell +1, j) \epsilon_j = -1\} \subset I$ so $I$ is not minimal again.

	Now assume $J_{\ell}$ is not minimal, so there exists some $J_{\ell'} \subsetneq J_{\ell}$. 
	
	%Now assume $I$ is not minimal, so there exists some $I' \subset I$ and some $\ell'$ such that $\eta(\ell', j)\epsilon_j = \pm 1 $ when $j \in I'$ and $\mp 1$ otherwise. 
	
	First, suppose the sign of $J_{\ell'}$ is $\zeta_{\ell'} = -1$. Suppose that $\ell' < \ell$.
	By \eqref{eq:eta-epsilon}, $J_{\ell'} \smallsetminus [\Lambda_{\ell'}+1, \Lambda_\ell] = J_{\ell} \smallsetminus [\Lambda_{\ell'}+1, \Lambda_\ell]$, while $J_{\ell'} \cap [\Lambda_{\ell'}+1, \Lambda_\ell]$ and $J_{\ell} \cap [\Lambda_{\ell'}+1, \Lambda_\ell]$ are disjoint with union $[\Lambda_{\ell'}+1, \Lambda_\ell]$. So,
	so $J_{\ell'} \subsetneq J_{\ell}$ implies $[\Lambda_{\ell - 1} + 1, \Lambda_{\ell}] \subseteq [\Lambda_{\ell'} + 1, \Lambda_{\ell}] \subset J_{\ell}$. Similarly, if $\ell' > \ell$, we have $[\Lambda_{\ell}+1, \Lambda_{\ell+1}] \subset J_{\ell}\,.$
    
	%We will first cover the case when $\eta(\ell', j)\epsilon_j = -1$ when $j \in I'$. Suppose that $\ell' < \ell$. Then $\eta(\ell', j)\epsilon_j = -\eta(\ell, j)\epsilon_j$ when $j \in [\Lambda_{\ell'} +1, \Lambda_\ell]$. Thus $I'$ and $I$ are disjoint in this range. But since $I' \subset I$, this means that $[\Lambda_{\ell - 1} +1, \Lambda_\ell] \subseteq [\Lambda_{\ell'} +1, \Lambda_\ell] \subseteq I$. When $\ell' > \ell$, we similarly have that $\eta(\ell', j)\epsilon_j = -\eta(\ell, j)\epsilon_j$ when $j \in [\Lambda_{\ell} +1, \Lambda_{\ell'}]$. Thus $[\Lambda_{\ell} +1, \Lambda_{\ell+1}] \subseteq [\Lambda_{\ell} +1, \Lambda_{\ell'}] \subseteq I$.
	
	Next suppose that the sign $\zeta_{\ell'} = 1$. Suppose that $\ell' < \ell$. By \eqref{eq:eta-epsilon}, $J_{\ell'} \cap [\Lambda_{\ell'}+1, \Lambda_\ell] = J_{\ell} \cap [\Lambda_{\ell'}+1, \Lambda_\ell]$, while $J_{\ell'} \smallsetminus [\Lambda_{\ell'}+1, \Lambda_\ell]$ and $J_{\ell} \smallsetminus [\Lambda_{\ell'}+1, \Lambda_\ell]$ are disjoint with union $[1,n] \smallsetminus [\Lambda_{\ell'}+1, \Lambda_\ell]$. Since $J_{\ell}$ is not minimal, we must have $[\Lambda_{\ell} +1, \Lambda_{\ell+1}] \subset [1,n] \smallsetminus [\Lambda_{\ell - 1}+1, \Lambda_\ell] \subset J_{\ell}$. When $\ell' > \ell$, by the same reasoning we have that $[\Lambda_{\ell - 1} +1, \Lambda_\ell] \subseteq J_{\ell}$.
	
	%Next suppose that $\eta(\ell', j) \epsilon_j = 1 $ when $j \in I'$. When $\ell' < \ell$, we have that $\eta(\ell', j) \epsilon_j = \eta(\ell, j) \epsilon_j$ when $j \le \Lambda_{\ell'}$ or $j > \Lambda_\ell$. Since $I' \subset I$, this means that $\eta(\ell, j) \epsilon_j = -1$ when $j \le \Lambda_{\ell'}$ or $j > \Lambda_\ell$. Thus $[\Lambda_{\ell} +1, \Lambda_{\ell+1}] \subseteq I$. When $\ell' > \ell$, by the same reasoning we have that $[\Lambda_{\ell - 1} +1, \Lambda_\ell] \subseteq I$.

\end{proof}

%Nicholas - Replaced $\lambda(a)$ with $m$ since we don't define $\lambda(a)$.
\begin{lem}\label{prop:formula_single_>=2 part}
	Fix $1\le f\le a-1$. We have
	\begin{equation}\label{eq:formula_single_>=2 parts}
		\sum_{\substack{ \mbox{\emph{\tiny \sop{}'s }} S \\ \mbox{\emph{\tiny with }} m\ge 2}}  T(S,C_f) \, A(S)\  =\  2n!(-1)^n \sum_{\substack{c+d\le n \\ c,d\ge 0}} (-1)^{c+d+1} G(n,f,c,d) \frac{1}{(n-c-d)!c!d!},
	\end{equation}
	where 
	\begin{equation}\label{eq:single_combin}
		G(n,f,c,d)\ \coloneqq\  \binom{n}{f} - \binom{n-c}{f-c} - \binom{n-d}{f-d} + \binom{n-c-d}{f-c-d}.
	\end{equation} 
\end{lem}
\begin{proof}
	Let $S = (m, \la_1, \dots, \la_m, \epsilon_1, \dots, \epsilon_n)$ denote a variable system of parameters. By Lemma~\ref{lem:there_is_only_one}, we can rewrite $T(S,C_f)$ as
	\begin{equation}
	    T(S,C_f) = \sum_{\ell=1}^{m} \mathbbm{1}_{\{J_{\ell} \in I(S) \text{ and } \#J_{\ell} = f\}}\,.
	\end{equation}
	We sum over systems of parameters by first summing over all values of $m$, then summing over all possible values of $\ell$, then summing over all possible values of $c = \lambda_\ell$ and $d = \lambda_{\ell + 1}$, then summing over all possible values of $\lambda_1, \ldots, \lambda_m$ and finally summing over all possible choices of $\epsilon_1, \ldots, \epsilon_n$.
	For fixed $m, \lambda_1, \dots, \lambda_m$, the $\epsilon_1, \dots, \epsilon_n$ and $J_{\ell}, \zeta_{\ell}$ uniquely determine each other, so we may rewrite the innermost sum as
	\begin{equation} \label{eq:sum-over-epsilons}
	 \sum_{(\epsilon_j) \in \{\pm 1\}^n} A(S)\mathbbm{1}_{\{J_{\ell} \in I(S) \text{ and } J_{\ell}: \#J_{\ell} = f\}} = A(S)\sum_{\zeta_{\ell} \in \{\pm 1\}} \sum_{\#J_{\ell} = f}  \mathbbm{1}_{\{J_{\ell} \in I(S)\}}\,.
	\end{equation}
	%The summand will be $A(S)$ multiplied by the number of subsets $(I) \in C_f$ such that $I \in I(S)$ multiplied by all possible choice of $\epsilon_1, \ldots, \epsilon_n$ such that $\eta(\ell, i) \epsilon_i = \pm 1$ when $i \in I$ and $\eta(\ell, i) \epsilon_i = \mp 1$ when $i \notin I$. 
	By Lemma~\ref{lem:good_placements t = 1}, the sum over $J_{\ell}$ is $G(n,f,c,d)$, since we can choose a general $f$ element subset in $\binom{n}{f}$ ways, and we need to subtract off when the $c$ element subset $[\lambda_1 + \dots + \lambda_{\ell -1} + 1, \lambda_1 + \dots + \lambda_{\ell} ] \subseteq I$ or when the $d$ element subset $[\lambda_1 + \dots + \lambda_{\ell} + 1, \lambda_1 + \dots + \lambda_{\ell + 1} ] \subseteq I$. Then, we add back in the case when both subsets are contained in $J_{\ell}$ since we have double counted it. Finally, there are 2 choices for $\zeta_{\ell}$. We have 
	\begin{equation}
	\begin{split}
	    \sum_{\substack{ \mbox{{\tiny \sop{}'s }} S \\ \mbox{{\tiny with }} m\ge 2}}  T(S,C_f) \, A(S)  \ &=\   \sum_{m=2}^n \sum_{\ell = 1}^m \sum_{\substack{c,d\ge 1 \\ c+d\le n }}  \sum_{\substack{\la_1+\dots+\la_m = n \\ \la_i \ge 1, \la_{\ell} = c, \la_{\ell+1} = d}}  \frac{(-1)^{m+1}}{m} \frac{n!}{\la_1! \cdots \la_m!} 2\, G(n,f,c,d).
	   \end{split}
	  \end{equation}
	  
	\no Noting that for each value of $\ell$ the inner summand is the same, we can set $\ell = m-1$ and write
	 \begin{equation}\label{eq: 1 class m-2}
	\begin{split}
	    \sum_{\substack{ \mbox{{\tiny \sop{}'s }} S \\ \mbox{{\tiny with }} m\ge 2}}  T(S,C_f)\,A(S)  \ &=\   2n! \sum_{\substack{c,d\ge 1 \\ c+d\le n}}  \frac{G(n, f, c, d)}{c! d!}\sum_{m=2}^n m  \sum_{\substack{\la_1+\dots+\la_{m-2} = n - c -d }}  \frac{(-1)^{m+1}}{m} \frac{1}{\la_1! \cdots \la_{m-2}!}.
	   \end{split}
	  \end{equation}
    \no The sum over $m$ equals $(-1)^{n+c +d + 1}/(n-c-d)!$, which follows from evaluating the coefficient of $z^n$ in
    \begin{equation}\label{eq: gen function ez}
    \begin{split}
    \sum_{n=0}^\infty \frac{(-1)^n}{n!}z^n = e^{-z} = \frac{1}{1 + (e^z - 1)} = \sum_{n=0}^\infty z^n \sum_{m=1}^n \sum_{\substack{ \lambda_1 + \dots + \lambda_m = n \\ \lambda_j \ge 1}} \frac{(-1)^m}{\lambda_1 ! \cdots \la_{m}!}.
    \end{split}
    \end{equation}
    Applying this to \eqref{eq: 1 class m-2} gives
    	\begin{equation}
	\begin{split}
	    \sum_{\substack{ \mbox{{\tiny \sop{}'s }} S \\ \mbox{{\tiny with }} m\ge 2}}  T(S,C_f) \, A(S)  \ &=\   2n!(-1)^n \sum_{\substack{c+d\le n \\ c,d\ge 1}} (-1)^{c+d+1} G(n,f,c,d) \frac{1}{(n-c-d)!c!d!}.
	   \end{split}
	  \end{equation}

	Now, we can extend the sum to include when $c = 0$ or $ d = 0$ to complete the proof as in this case $G(n, f,c,d) = 0$.
\end{proof}

We complete our evaluation of the case when $m \ge 2$ with the following lemma, proven in Appendix~\ref{app_single_simp}.
\begin{lem}\label{lem_single_simp}
    Fix $1 \le f \le a -1$. We have
    \begin{equation}\label{eq: lem_single_simp}
        \sum_{\substack{ \mbox{\emph{\tiny \sop{}'s }} S \\ \mbox{\emph{\tiny with }} m\ge 2}}  T(S,C_f) \, A(S)\  =\ 2 \binom{n}{f}\left((-1)^{n+f+1} - 1 \right).
    \end{equation}
\end{lem}

Now we evaluate the case when $m = 1$.
\begin{lem}
	Fix $1\le f\le a-1$. We have
	\begin{equation}\label{eq:formula_single_1 part}
		\sum_{\substack{ \mbox{\emph{\tiny \sop{}'s }} S \\ \mbox{\emph{\tiny with }} m = 1}}  T(S,C_f) \, A(S) \ = \ 2\binom{n}{f}.
	\end{equation}
\end{lem}
\begin{proof}
	We let $S = (1,\la_1,\epsilon_1,\dots,\epsilon_n)$ denote a variable system of parameters. Since $m = 1$, we have $\lambda_{1} = n$ and $A(S) = \frac{(-1)^2}{1} \frac{n!}{n!} = 1$ for all $S$. Now, as in \eqref{eq:sum-over-epsilons}, %in the proof of Lemma~\ref{prop:formula_single_>=2 part}, 
	we may rewrite the sum over $\epsilon_1, \dots, \epsilon_n$ as a sum over $J_{1}, \zeta_{1}$. Since $m=1$, any $f$-element $J_1 \in J(S)$ will be minimal. So,
	\begin{equation}
        \sum_{\substack{ \mbox{\emph{\tiny \sop{}'s }} S \\ \mbox{\emph{\tiny with }} m\ge 2}}  T(S,C_f) \, A(S)\  =
         A(S) \sum_{\zeta_{1} \in \{\pm 1\}} \sum_{\#J_{1} = f}  \mathbbm{1}_{\{J_{1} \in I(S)\}} \  =
        \sum_{\zeta_{1} \in \{\pm 1\}} \sum_{\#J_{1} = f} 1 = \ 2 \binom{n}{f}\,.
    \end{equation}
%	$A(S)$ is constant as $S$ varies. We may count the left hand side of \eqref{eq:formula_single_1 part} by multiplying $A(S) = \frac{(-1)^2}{1} \frac{n!}{n!} = 1$ by the number of subsets $I\subseteq [1,n]$ such that $(I) \in C_f$, and finally multiplying by the number of choices of $\epsilon_1,\dots,\epsilon_n$ such that $\eta(1,i)\epsilon_{i} = \pm 1$ for $i \in I$ and $\eta(1,j)\epsilon_j = \mp 1$ for $j\notin I$ and $I\in I(S)$. This quantity is simply $2\binom{n}{f}$, since for any such $I$, of which there are $\binom{n}{f}$, there are two such choices of $\epsilon_i$. 
\end{proof}

\no Adding equations \eqref{eq: lem_single_simp} and \eqref{eq:formula_single_1 part} gives the main result of the section.
\begin{lem}\label{cor:nonoverlapping final coef}
	Fix $1\le f\le a-1$. Then
	\begin{equation}\label{eq:nonoverlapping final coef}
		\sum_{\substack{ \mbox{\emph{\tiny \sop{}'s }} S }}  T(S,C_f) \, A(S) \ = \ 2(-1)^{n+f+1}\binom{n}{f}.
	\end{equation}
\end{lem}

%%%%%%%%%%%%%%%%%%%%%%%%%%%%%%%%%%%%%%%%%%%%%%
%%%%%%%%%%%%%%%%%%%%%%%%%%%%%%%%%%%%%%%%%%%%%%
%%%%%%%%%%%%%%%%%%%%%%%%%%%%%%%%%%%%%%%%%%%%%%
%%%%%%%%%%%%%%%%%%%%%%%%%%%%%%%%%%%%%%%%%%%%%%
%%%%%%%%%%%%%%%%%%%%%%%%%%%%%%%%%%%%%%%%%%%%%%

\subsubsection{The vanishing of valid $t$-classes for $t\ge 2$} \label{subsec: t2 class}

In this section, we show that all terms with $t \ge 2$ in \eqref{eq:Qn_end} vanish. Our main result is the following.

\begin{lem}\label{cor:classes_cancel}
	Let $C$ be a valid $t$-class with $t\ge 2$. Then 
	\begin{equation}
		\sum_{\substack{ \mbox{\emph{\tiny \sop{}'s }} S}}  T(S,C) \, A(S) \ = \ 0.
	\end{equation}
\end{lem}

Throughout this section, let $S =(m, \la_1, \dots, \la_m, \epsilon_1, \dots, \epsilon_n)$ be a system of parameters, $C$ a valid class, and $(I_1,\dots,I_t) \in C$ a tuple of subsets of $[1,n]$ such that for each $1 \le i \le t$, there is some $\ell_i$ and $\zeta_{\ell_{i}} \in \{\pm 1\}$ such that $I_i = \{j: \eta(\ell_i, j) \epsilon_j = \zeta_{\ell_{i}}\}$. I.e., $I_i = J_{\ell_{i}}$ with sign $\zeta_{\ell_{i}}$. Reorder the $I_i$ so that $\ell_1 < \ell_2 < \cdots < \ell_t$ and set $I_i' \coloneqq I_i - \bigcap_{k=1}^t I_k$ and $j_i = \Lambda_{\ell_{i}} = \sum_{k=1}^{\ell_i} \lambda_k$. To begin, we prove lemmas which characterize $(I_1, \ldots, I_t)$. 

\begin{lem}\label{lem: sign switch}
    Set $I_1 = J_{\ell_{1}}$ with sign $\zeta_{\ell_{1}}$ and suppose there is some minimal $T$ such that $I_T = J_{\ell_{T}}$ with sign $\zeta_{\ell_{T}} = -\zeta_{\ell_{1}}$. Then, for all $i \geq T$, we have $I_{i} = \zeta_{\ell_{i}}$ with sign $\zeta_{\ell_{i}} = - \zeta_{\ell_{1}}$.
%    Set $I_1 = \{j: \eta(\ell_1, j) \epsilon_j = \pm 1\}$ and suppose there is some minimal $T$ such that $I_T = \{j: \eta(\ell_T, j) \epsilon_j = \mp 1\}$. Then $\forall i \ge T$, $I_i = \{j: \eta(\ell_i, j) \epsilon_j = \mp 1\}$. 
\end{lem}
\begin{proof}
    Assume WLOG that $\zeta_{\ell_{1}} = -1$ so $I_1 = \{j: \eta(\ell_1, j) \epsilon_j = - 1\}$ and let $T$ be the smallest value such that $I_T = \{j: \eta(\ell_T, j) \epsilon_j =  1\}$. Suppose there exists some $s> T$ such that  $I_s = \{j: \eta(\ell_s, j)\epsilon_j = -1 \}$. If $j \le j_{T-1}$ or $j > j_s$, then $\eta(\ell_T, j) = \eta(\ell_s, j)$, so $j \in I_T \cup I_s$ so $[1, j_{T-1}] \cup [j_s + 1, n] \subseteq I_T \cup I_s$.
    Similarly, if $j \in [j_{T-1} + 1, j_s]$, then $\eta(\ell_{T-1}, j) = -\eta(\ell_s, j)$, so $j \in I_{T-1} \cup I_s$ so $[j_{T-1} + 1, j_s] \subseteq I_{T-1} \cup I_s$. Since $[1, j_{T-1}] \cup [j_s + 1, n] \cup [j_{T-1} + 1, j_s] = [1,n]$ and $a \le n/2$, we must have that either $|I_T \cup I_s| \ge a$ or $|I_{T-1} \cup I_s| \ge a$. Then, by Lemma \ref{lem:2_sets_support}, $C$ is not valid, a contradiction. Thus such an $s$ cannot exist so $I_i = \{ j : \eta(\ell_i,j) \epsilon_j = +1\}$ for all $i\ge T$.
\end{proof}
The above lemma shows that the sign of $(I_1, I_2, \ldots, I_t)$ can switch at most once. This motivates the following definition.
\begin{defn}
    The \textbf{\emph{transition point}} of $(I_1, \ldots, I_t)$ is the smallest $T$ such that $\zeta_{\ell_{T}} = -\zeta_{\ell_{1}}$. If $I_1, I_2, \ldots, I_t$ all have the same sign (so that no such $T$ exists), then we set $T = 1$.
\end{defn}

\begin{lem}\label{prop:necessary_valid}
	Let $T$ be the transition point of $(I_1, \ldots, I_t)$. Then
	\begin{align}
	    \bigcup_{i=1}^t I'_i \  = \ I'_{T-1} \cup I'_T \ &= \ [1,n] \smallsetminus [j_{T-1}+1, j_{T}]\,, \\%[1, j_{T-1}] \cup [j_{T} + 1, n]\,. 
%	    \mbox{\quad{and}}\\
	     I'_{T-1} \cap I'_{T} &= \emptyset\,, \mbox{\quad{and}} \\
	     \bigcap_{i=1}^t I_i \  = \ I_{T-1} \cap I_T \ &\subseteq \  [j_{T-1}+1, j_{T}]\,, 
	\end{align}
	taking indices cyclically in $[1,t]$ and intervals cyclically in $[1,n]$ so that $I'_{0} \coloneqq I'_{t}$ and $I_{0} \coloneqq I_{t}$ and $[j_0 + 1, j_1] = [j_t + 1, n]\cup [1,j_1]$.
	%with the conventions $I'_{0} \coloneqq I'_t$ and $[1,j_0] = \emptyset$. %if $T = 0$ then $\bigcup_{i=1}^t I'_i = I'_1\cup I'_{t} =  [j_1 + 1, j_t]$. 
	Additionally, if $1 \le i \le t$ with $i \ne T - 1$, then 
	\begin{align}
	    (I_i \cap [j_i +1, j_{i+1}]) \cup (I_{i+1} \cap [j_i +1, j_{i+1}]) \ &= \ [j_i +1, j_{i+1}] \mbox{\quad and} \\
	    (I_i \cap [j_i +1, j_{i+1}]) \cap (I_{i+1} \cap [j_i +1, j_{i+1}]) \ &= \ \emptyset.
	\end{align}
	In other words, the restriction of $I_i$ and $I_{i+1}$ to the interval $[j_i +1, j_{i+1}]$ forms a partition of the interval. If $i = t$ and $T \ne 1$, again taking indices and intervals cyclically, we set $I_{t+1} = I_1$ and $[j_t +1, j_{1}]  = [j_t + 1, n] \cup [1, j_1]$.

\end{lem}
\begin{proof}
    We consider indices and intervals cyclically in $[1,t]$ and $[1,n]$ respectively, as in Lemma~\ref{prop:necessary_valid}.
%   We consider two cases. %In both, we assume WLOG that $I_1 = \{j: \eta(\ell_1, j) \epsilon_j = - 1\}$. 
%   First, suppose $T\neq 1$. 
   
   For $j \in [j_{T-1}+1,j_T]$, the value $\eta(\ell_i, j) \epsilon_{j} \zeta_{\ell_{i}}$ is independent of $i$ since for any $i, i'$ either $\zeta_{\ell_{i}}/\zeta_{\ell_{i'}}$ and $\eta(\ell_i,j)/\eta(\ell_{i'},j)$ are both $1$ or both $-1$. 
   So, for any $j \in [j_{T-1}+1, J_{T}]$ either $j \in I_{i}$ for all $i$ of $j \notin I_{i}$ for all $i$. So, $\bigcup_{i=1}^{t} I_{i}'\subset [1,n] \smallsetminus [j_{T-1}+1, j_{T}]$ and $\bigcap_{i=1}^{t} I_{i} \cap [j_{T-1}+1,j_T] = I_{T-1} \cap I_{T} \cap [j_{T-1}+1,j_T]$.
   
   For $j \notin [j_{T-1}+1,j_T]$, we have $\eta(\ell_{T-1}, j) = \eta(\ell_{T},j)$ and so $\eta(\ell_{T-1}, j) \epsilon_{j}, \zeta_{\ell_{T-1}} = -\eta(\ell_{T}, j)\epsilon_{j}, \zeta_{\ell_{T}}$. So, every $j \notin [j_{T-1}+1,j_T]$ belongs to exactly one of $I'_{T-1}$ and $I'_{T}$. We conclude that  $\bigcup_{i=1}^{t} I_{i}' = I'_{T-1} \cap I'_{T} \subset [1,n] \smallsetminus [j_{T-1}+1, j_{T}]$ and $I_{T-1} \cap I_{T} \subset [j_{T-1}+1, j_{T}]$. So, $I_{T-1} \cap I_{T} \cap [j_{T-1}+1, j_{T}] = I_{T-1} \cap I_{T} = \bigcap_{i=1}^{t} I_{i}$.
   
   For $i \neq T-1$, we have $\eta(\ell_{i}, j) = -\eta(\ell_{i+1},j)$ if and only if $j \in [j_{i}+1,j_{i+1}]\,.$ Since $\zeta_{\ell_{i}} = \zeta_{\ell_{i+1}}$, this means 
   $\eta(\ell_{i},j) \epsilon_{j} \zeta_{\ell_{i}} = -\eta(\ell_{i+1},j) \epsilon_{j} \zeta_{\ell_{i+1}}$ if and only if $j \in [j_{i}+1,j_{i+1}]\,.$ Hence, each $j \in [j_{i}+1, j_{i+1}]$ is contained in exactly one of $I_{i}$ and $I_{i+1}$, as desired. 
\end{proof}

\begin{defn}
    For each $1 \le i \le t$, set
    \begin{align}
        r_i\ &\coloneqq \ |I_i \cap [j_{i} + 1, j_{i+1}]| \mbox{\quad and} \\
        s_i\ &\coloneqq \ |I_{i+1} \cap [j_{i} + 1, j_{i+1}]|.
    \end{align}
    We call the ordered tuple $(T, r_1, s_1, \ldots, r_{t}, s_{t})$ the \textbf{\emph{structure}} of $(I_1, \ldots, I_t)$ in $S$ where $T$ is the transition point of $(I_1, \ldots, I_t)$. If $(T, r_1, s_1, \ldots, r_{t}, s_{t})$ is a structure for some $(I_1, \ldots, I_t) \in C$, we call it a \textbf{\emph{valid structure}} for $C$.
\end{defn}
By Lemma~\ref{prop:necessary_valid}, $r_{T-1} = s_{T-1} = \left|\bigcap_{k=1}^t I_k \right|$.  Lemma~\ref{prop:necessary_valid} also shows that when $i \ne T-1$, $r_i + s_i = |[j_{i} + 1, j_{i+1}]| = j_{i+1} - j_i = \lambda_{\ell_i + 1} + \cdots + \lambda_{\ell_{i+1}}$. The following lemma shows that the two tuples with the same structure are in the same $t$-class.

\begin{lem}\label{lem: same structure}
    Let $C$ be a valid $t$-class and let $(I_1, \ldots, I_t) \in C$ such that $(I_1, \ldots, I_t) \in I(S)$ for some \sop{} $S$. Let $(J_1, \ldots, J_t)$ be another tuple such that $(J_1, \ldots, J_t) \in I(P)$ for some \sop{} $P$. If the structure of $(I_1, \ldots, I_t)$ in $S$ is the same as the structure of $(J_1, \ldots, J_t)$ in $P$, then $(J_1, \ldots, J_t) \in C$.
\end{lem}
\begin{proof}
    %The key idea is that for fixed $m, \lambda_{1}, \dots, \lambda_{m}$, one can exactly recover the $\epsilon_{j}$ in a \sop{}. from the sets $I_{i} \cap [j_{i}+1,j_{i+1}]$ and vice-versa.
    %The proof is simple after defining all necessary quantities.
    We first set notation. Set $S = (m, \lambda_1, \ldots, \lambda_m, \epsilon_1, \ldots, \epsilon_n)$ and $P = (m, \lambda'_1, \ldots, \lambda'_m, \epsilon'_1, \ldots, \epsilon'_n)$. Set $\ell_1 < \cdots < \ell_t$ and $\ell'_1 < \cdots < \ell'_t$ such that $I_i = \{j: \eta_S(\ell_i, j) = \zeta_{\ell_{i}}\}$ and $J_i = \{j: \eta_P(\ell'_i, j) = \zeta_{\ell'_{i}}\}$. Lastly, define $j_i = \sum_{k=1}^{\ell_i} \lambda_k$ and $j'_i = \sum_{k=1}^{\ell'_i} \lambda'_k$. 
    
    Without loss of generality, we may assume $\zeta_{\ell_{1}} = \zeta_{\ell_{1}'}\,$ or else we may replace each $\epsilon_{i}$ with $-\epsilon_{i}$. Since $(I_1, \ldots, I_t)$ and $(J_1, \ldots, J_t)$ have the same structure, for each $i$, we have 
    \begin{align}
        |I_i \cap [j_i+1, j_{i+1}]|\ &= \ |J_i \cap [j'_i+1, j'_{i+1}]| \mbox{\quad and} \\
        |I_{i+1} \cap [j_i+1, j_{i+1}]|\ &= \ |J_{i+1} \cap [j'_i+1, j'_{i+1}]|.
    \end{align}
    Let $\tau \in S_n$ be the permutation which maps the $k$th smallest element of $|I_i \cap [j_i+1, j_{i+1}]|$ to the $k$th smallest element of $|J_i \cap [j'_i+1, j'_{i+1}]|$ and the $k$th smallest element of $|I_{i+1} \cap [j_i+1, j_{i+1}]|$ to the $k$th smallest element of $|J_{i+1} \cap [j'_i+1, j'_{i+1}]|$. 
    
    Since $\tau([j_{i}+1,j_{i+1}]) = [j'_{i}+1, j_{i+1}]$, for all $i \in [1,t]$ and $j \in [1,n]$ we have $\eta(\ell_{i}, j) = \eta(\ell'_{i}, \tau(j))$. Since $(I_1,\dots,I_t)$ and $(J_1, \dots, J_t)$ have the same transition value $T$ and we assumed $\zeta_{\ell_{1}} = \zeta_{\ell'_{1}}$, we have $\zeta_{\ell_{i}} = \zeta_{\ell'_{i}}$ for all $i$.
    Moreover, for $j \in [j_{i}+1, j_{i+1}]$ we have $\eta(\ell_{i}, j) \epsilon_{j} \zeta_{\ell_{i}} =  \eta(\ell'_{i}, \tau(j)) \epsilon'_{\tau(j)} \zeta_{\ell'_{i}}$ so that $\epsilon_{j} =\epsilon'_{\tau(j)}\,.$ But this is true for all $i$, so in fact $\epsilon_{j} =\epsilon'_{\tau(j)}$ for all $j \in [1,n]$. 
    So, for all $i \in [1,t]$ and $j \in [1,n]$ we have $\eta(\ell_{i}, j) \epsilon_{j} \zeta_{\ell_{i}} =  \eta(\ell'_{i}, \tau(j)) \epsilon'_{\tau(j)} \zeta_{\ell'_{i}}$. It follows that $\tau(I_1, \ldots, I_t) = (J_1, \ldots, J_t)$ so $(J_1, \ldots, J_t) \in C$.
\end{proof}
Lemma \ref{lem: same structure} shows that if a structure is valid for $C$, then all tuples with that structure are in $C$. Thus in order to calculate $\sum T(S,C) \, A(S)$, we can first sum over all valid structures for $C$ and then count tuples and \sop{}s with that structure.
All that remains is to determine when $(I_1, \ldots, I_t) \in I(S)$.

\begin{lem}\label{lem: t2 I in I(S)}
    Suppose $I_1, \ldots, I_t \in J(S)$. Then, $I_1, \ldots, I_t \in I(S)$ if and only if for each $1 \le i \le t$, $[j_i - \lambda_{\ell_i} +1, j_i] \not\subseteq I_i$ and $[j_i +1, j_i + \lambda_{\ell_i + 1}] \not\subseteq I_i$.
\end{lem}
\begin{proof}
    Note that $I_i \in J(S)$ implies $\#I_{i} \leq a-1$. So, this is an immediate corollary of Lemma~\ref{lem:good_placements t = 1} which says $I_{i} \in I(S)$ if and only if $[j_i - \lambda_{\ell_i} +1, j_i] \not\subseteq I_i$ and $[j_i +1, j_i + \lambda_{\ell_i + 1}] \not\subseteq I_i$.
\end{proof}

Now we are ready to calculate $\sum T(S,C) \, A(S)$.

\begin{lem}\label{prop:t>=2 expression}
	Let $C$ be a valid $t$-class with $t\ge 2$. Then 
	\begin{equation}\label{eq:T(S,C) in general}
		\sum_{\substack{ \mbox{\emph{\tiny \sop{}'s }} S}}  T(S,C) \, A(S)
	\end{equation}
	is a sum of terms of the form 
	\begin{equation}
		\sum_{d=1}^f \sum_{\substack{\mu_1+\dots+\mu_d = f \\ \mu_i \ge 1}}  \frac{(-1)^d}{\mu_1!\cdots\mu_d!}  H(f,g,\mu_1,\mu_d),
	\end{equation}
	for some $f$ and $g$, where
	\begin{equation}\label{eq: Hhf def}
		 H(f,g,\mu_1,\mu_d)\ \coloneqq\ \binom{f}{g} - \binom{f-\mu_1}{g-\mu_1} - \binom{f- \mu_d}{g} + \binom{f- \mu_1- \mu_d}{ g -\mu_1}.
	\end{equation}
\end{lem}

\begin{proof}
    Let $C$ be a valid $t$-class. By Lemma~\ref{lem: same structure}, when summing over all \sop{}s, we can first sum over all valid structures, and then over all \sop{}s and tuples with that structure. To do this, we can sum over all $m$, then over all possible values of $\ell_1, \ldots, \ell_t$, then over all $\lambda_1, \ldots, \lambda_m$ such that $\lambda_1 + \cdots + \lambda_m = n$ and $\lambda_{\ell_i + 1} + \cdots + \lambda_{\ell_{i+1}} = r_i + s_i$ for each $i \ne T-1$. Now we use Lemma \ref{lem: t2 I in I(S)} to determine the summand. We can pick the elements of $\bigcap_{k=1}^t I_k$, which by Lemma~\ref{prop:necessary_valid} is a subset of $[j_{T-1} + 1, j_T]$, in $G(j_T - j_{T-1}, r_T, \lambda_{\ell_{T-1}}, \lambda_{\ell_T})$ ways, where $G$ is defined as in \eqref{eq:single_combin}. Next, we choose the $r_i$ elements of $I_i$ contained in the interval $[j_i + 1, j_{i+1}]$ in $H(r_i + s_i, r_i, \lambda_{\ell_i + 1}, \lambda_{\ell_{i+1}})$ ways. Then, there are two possible choices for the sign $\zeta_{\ell_{1}}$ of $I_1$ and then the signs for the rest of the $I_i$'s follow because the point of transition $T$ is fixed. The choice of $\zeta_{\ell_{1}}$ and each $r_i$-element set $[j_{i}+1, j_{i+1}] \cap I_{i}$ determines all $\epsilon_{j}$, so they determine exactly the same data as the $I_{i}$. 

    Lastly, we multiply by $A(S)$. We have that
    \begin{equation}
    \begin{split}
         \sum_{\substack{ \mbox{\emph{\tiny \sop{}'s }} S}}  T(S,C) \cdot A(S) \ &= \ \sum_{\substack{(T, r_1, s_1, \ldots, r_t, s_t) \\ \mbox{\tiny a valid structure for $C$}}} \sum_{m=1}^n \sum_{1\le \ell_1 < \dots <\ell_t \le m} \sum_{\substack{\la_1+\dots+\la_m = n \\ \la_{\ell_i+1}+\dots+\la_{\ell_{i+1}} = r_{i+1}+s_{i+1} \mbox{ \tiny for each $i \ne T$} \\ \la_i \ge 1   }} \\ 
		& \times \ 2 G(u,v,\la_{\ell_1},\la_{\ell_t+1}) \prod_{\substack{1 \le i \le t \\ i \ne T - 1}}  H(r_i+s_i,r_i,\la_{\ell_i+1},\la_{\ell_{i+1}})  \frac{(-1)^{m+1}}{m}\frac{n!}{\la_1!\cdots\la_m!}.
    \end{split}
    \end{equation}
    For each structure, we can fix some $i \ne T-1$, which exists since $t \ge 2$, to see that this is a sum of terms of the form
	\begin{equation}
		\sum_{d = 1}^{r_i+s_i}  \sum_{\substack{\mu_1+\dots+\mu_{d} = r_i+s_i \\ \mu_i \ge 1}} \frac{(-1)^{d}}{\mu_1!\cdots\mu_{d}!}  H(r_i + s_i, r_i, \mu_1, \mu_{d}). 
	\end{equation}
\end{proof}
We finish the calculation with the following combinatorial lemma, proven in Appendix~\ref{app_disj_vanishes}.

\begin{lem}\label{lem:prod_disj_vanishes}
    Fix $f, g$ and let $H(f, g, \mu_1, \mu_d)$ be as in \eqref{eq: Hhf def}. Then
    \begin{equation}
		\sum_{d = 1}^{f}  \sum_{\substack{\mu_1+\dots+\mu_{d} = f \\ \mu_i \ge 1}} \frac{(-1)^{d}}{\mu_1! \cdots \mu_{d}!}  H(f, g, \mu_1, \mu_{d})\ = \ 0 .
	\end{equation}
\end{lem}
Combining Lemmas \ref{prop:t>=2 expression} and \ref{lem:prod_disj_vanishes} completes the proof of Lemma \ref{cor:classes_cancel}. \qed

%%%%%%%%%%%%%%%%%%%%%%%%%%%%%%%%%%%%%%%%%%%%%%
%%%%%%%%%%%%%%%%%%%%%%%%%%%%%%%%%%%%%%%%%%%%%%
%%%%%%%%%%%%%%%%%%%%%%%%%%%%%%%%%%%%%%%%%%%%%%
%%%%%%%%%%%%%%%%%%%%%%%%%%%%%%%%%%%%%%%%%%%%%%
%%%%%%%%%%%%%%%%%%%%%%%%%%%%%%%%%%%%%%%%%%%%%%

\subsection{Computing the integral piece}\label{subsec: final integral}

In this section we complete the proof of Proposition \ref{thm:ComputeQnPhi} by calculating the integral $\int C dy$ appearing in \eqref{eq:Qn_end}. %{eq_q_expression}. 
Applying Lemmas \ref{cor:classes_cancel} and \ref{cor:nonoverlapping final coef} to \eqref{thm:Qn_end} gives
\begin{equation}\label{eq_q_expression}
		Q_n(\phi)\ = \ 2^{n-1}(-1)^{n} \sum_{\l =0}^{a-1} (-1)^{\ell} \binom{n}{\l} \int_0^\infty \cdots  \int_0^\infty \hphi(y_1) \cdots  \hphi(y_n) \indicator_{\{n-\l+1,\ldots,n\}}dy_1 \cdots  dy_n.
\end{equation}

Next we define
\begin{equation}\label{eq: xxi def}
    \xxi{\ell} \ :=\ \int_0^\infty \cdots \int_0^\infty \hat{\phi}(y_1)\cdots \hat{\phi}(y_n)\indicator_{\{y_1+\cdots+y_{n-\ell}-y_{n-\l+1}-\cdots-y_n>1\}} dy_1 \cdots dy_n
\end{equation}
and
\begin{equation}\label{eq: bxi def}
    \bxi{\ell}\ := \ \int_{-\infty}^\infty \cdots \int_{-\infty}^\infty \hat{\phi}(y_1)\cdots \hat{\phi}(y_n)\indicator_{\{|y_1+\cdots+y_{n-\ell}|-|y_{n-\l+1}|-\cdots-|y_n|>1\}} dy_1 \cdots dy_n.
\end{equation}

\no We have that \eqref{eq_q_expression} equals
\begin{equation}\label{eqn_qn_xxi}
Q_n(\phi)\ = \ 2^{n-1} (-1)^n \sum_{\l=0}^{a-1} (-1)^\l\binom{n}{\l}\xxi{\l}.
\end{equation}

\no We express $Q_n(\phi)$ in terms of $\bxi{\ell}$ with the following lemma.

\begin{lem}\label{prop_qn_bxi}
Let $\phi\in \mathcal{S}_{ec}(\R)$ with $\supp(\widehat{\phi})\subseteq \left[-\frac{1}{n-a},\frac{1}{n-a}\right]$. Then
\begin{equation}\label{eq: Qn bxi}
    Q_n(\phi)\ =\ 2^{n-2} (-1)^{n} \sum_{t=0}^{a-1} (-1)^{t} \binom{n}{t} \bxi{t}.
\end{equation}
\end{lem}

\begin{proof}
Given $\suppgen$ and $t \leq \t$, if $\abs{y_1+\cdots+y_{n-
t}}-\abs{y_{n-t+1}}-\cdots-\abs{y_n} > 1$, then either at most $i \leq \t-t$ of the $y_j$'s in the first absolute value are nonnegative and the rest are negative or at most $i \leq \t-t$ of the $y_j$'s in the first absolute value are nonpositive or zero and the rest are positive. Moreover, the sign of $y_1 + \cdots y_{n-t}$ matches the second group. 
% have the same sign and the rest have the opposite sign. 
 There are $\binom{n-t}{i}$ ways to choose these indices and we introduce a factor of 2 from choosing the sign of $y_1+\cdots+y_{n-
t}$. Lastly, since $\hphi$ is even, we multiply by a factor of $2^t$ to account for changing the limits of integration over $y_{n-t+1}, \ldots , y_n$. Thus we have
\iffalse For every such choice of $i$, since $\hphi$ is even and each $y_j$ has the same contribution in the integral, we have
\fi
\begin{align}
\bxi{t}\ &=\ 2^{t+1}\int_0^\infty \cdots \int_0^\infty \hat{\phi}(y_1)\cdots \hat{\phi}(y_n)\left[\sum_{i=0}^{\t-t}\binom{n-t}{i}\indicator_{\{y_1+\cdots+y_{n-i-t}-y_{n-i-t+1}-\cdots-y_n>1\}} \right]dy_1 \cdots dy_n \nonumber \\
&=\ 2^{t+1} \sum_{i=0}^{\t-t} \binom{n-t}{i}\xxi{i+t}. \label{eqn_xxi_bxi}
\end{align}

\no Applying the identity
\begin{equation}
    \sum_{t=0}^\l (-2)^{t} \binom{n}{t} \binom{n-t}{\l-t} \ =\ \binom{n}{\l} \sum_{t=0}^\l (-2)^{t} \binom{\l}{t} \ =\ \binom{n}{\l} (1-2)^\l\ = \ \binom{n}{\l}(-1)^\l
\end{equation}
to \eqref{eqn_qn_xxi} gives
\begin{equation}
     Q_n(\phi)\ = \  2^{n-1} (-1)^n \sum_{\l=0}^\t\xxi{\l} \sum_{t=0}^\l (-2)^t \binom{n}{t}\binom{n-t}{\l-t}.
\end{equation}

\no Switching the order of summation and setting $i = \l - t$ gives
\begin{align}
     %Q_n(\phi)&= \frac{(-1)^n}{2} \sum_{t=0}^\t (-2)^t\binom{n}{t} \sum_{\l=t}^\t \binom{n-t}{\l-t} \xxi{\l} \\ 
     Q_n(\phi) \ &= \ 2^{n-2} (-1)^{n} \sum_{t=0}^\t (-1)^{t} \binom{n}{t} \left[2^{t + 1} \sum_{i=0}^{\t-t} \binom{n-t}{i} \xxi{i+t} \right].
\end{align}
Applying \eqref{eqn_xxi_bxi} gives the desired result.
\end{proof}
We complete the evaluation of $Q_n(\phi)$ by computing $\bxi{t}$.

\begin{lem}\label{prop_eval_bxi}
	Let $\phi\in \mathcal{S}_{ec}(\R)$ with $\supp(\widehat{\phi})\subseteq \left[-\frac{1}{n-a},\frac{1}{n-a}\right]$. Then we have %$\text{Supp}(\phi) \subset \left[ -\frac{1}{n-(\t+1)},\frac{1}{n-(\t+1)}\right]$
	\begin{align}
	\bxi{\l} \ &= \ \phi^{n}(0)\nonumber \\
	&  - 2\int_{-\infty}^\infty \cdots \int_{-\infty}^\infty \hphi(x_{\l+1}) \cdots \hphi(x_2)\int_{-\infty}^\infty \phi^{n-\l}(x_1)\frac{\sin(2\pi x_1 (1+\abs{x_2}+\cdots+\abs{x_{\l+1}}))}{2\pi x_1} dx_1\cdots dx_{\l+1}
	\end{align}
 for $\ell \leq \t$. 
\end{lem}
\begin{proof}
    We apply a change of variables given by
%    \begin{eqnarray}
%    x_{i} = \sum_{j=1}^{i} y_{j} & \qquad \text{for} \qquad i \in [1,n-\ell]\hspace{18pt}\\
%    x_{i} = y_{i} \hspace{18pt} & \qquad \text{for} \qquad i \in [n-\ell+1,n]
%    \end{eqnarray}
%    so that
%    \begin{eqnarray}
%    y_{i} = x_{i} - x_{i-1} & \qquad \text{for} \qquad i \in [2,n-\ell]\hspace{50pt}\\
%    x_{i} = y_{i} \hspace{34pt} & \qquad \text{for} \qquad i \in \{1\} \cup [n-\ell+1,n]
%    \end{eqnarray}
    
	\begin{equation}\begin{matrix}
	x_{1} = y_{1} & \space & y_{1} = x_{1}\\
	x_{2} = y_{1} + y_{2} & \space & y_{2} = x_{2} - x_{1}\\
	\vdots & \space & \vdots\\
	x_{n - \l} = \sum_{j = 1}^{n - \l}y_{j} & \space & y_{n - \l} = x_{n - \l} - x_{n - \l - 1}\\
	x_{n - \l + 1} = y_{n - \l + 1} & \space & y_{n - \l + 1} = x_{n - \l + 1}\\
	\vdots & \space & \vdots\\
	x_{n} = y_{n} & \space & y_{n} = x_{n}
	\end{matrix}\end{equation}
	to \eqref{eq: bxi def}, giving
	\begin{align}\label{eq: bxi change of var}
	\bxi{\ell}\ &= \  \int_{-\infty}^{\infty}\cdots \int_{-\infty}^{\infty}\hat{\phi}(x_1)\hat{\phi}(x_{2} - x_{1})\cdots \hat{\phi}(x_{n - \l} - x_{n - \l - 1})\nonumber \\
	& \ \hs{1} \times \hat{\phi}(x_{n - \l + 1}) \cdots \hat{\phi}(x_{n})\indicator_{\{|x_{n - \l}| - (|x_{n - \l + 1}| + \cdots + |x_{n}|) > 1  \}}dx_1\cdots  dx_n.
	\end{align}
	
	\no Repeatedly applying the identity $\intii \hat{f}(v) \hat{g}(u-v) dv =  \widehat{fg}(u)$ (which arises from the convolution theorem) to \eqref{eq: bxi change of var} gives
	\begin{equation}
	\bxi{\ell}\ =  \ \int_{-\infty}^{\infty}\cdots\int_{-\infty}^{\infty} \widehat{\phi^{n - \l}}(x_{n - \l})
	\hat{\phi}(x_{n - \l + 1}) \cdots \hat{\phi}(x_{n})\indicator_{\{|x_{n - \l}| - (|x_{n - \l + 1}| + \cdots + |x_{n}|) > 1  \}}dx_{n - \l}\cdots dx_n.
	\end{equation}
	
	\no We rename $x_{n - \l}$ to $x_{1}$, $x_{n - \l + 1}$ to $x_{2}$, and so on until $x_{n}$ to $x_{\l + 1}$. This and the identity
	\begin{equation}
	\indicator_{\{|x_{1}| - (|x_{2}| + \cdots + |x_{\l + 1}|) > 1  \}}\  = \  1 - \indicator_{\{|x_{1}| \le 1 + |x_{2}| + \cdots + |x_{\l + 1}|  \}}
	\end{equation}
	\no gives
	
	\begin{equation}
	\bxi{\ell}\ = \ \intii \cdots \intii \widehat{\phi^{n - \l}} (x_{1})
	\hat{\phi} (x_{2}) \cdots \hat{\phi}(x_{\l + 1})(1 - \indicator_{\{|x_{1}| \le 1 + |x_{2}| + \cdots + |x_{\l + 1}|  \}}) dx_{1} \cdots dx_{\l + 1}.
	\end{equation}
	
	\no Distributing and using the identity $\phi (0) = \intii \hphi(x) dx $, we have that
	\begin{equation}
	\bxi{\ell}\ = \  \phi^{n}(0) - \int_{-\infty}^{\infty}\cdots\int_{-\infty}^{\infty} \widehat{\phi^{n - \l}}(x_{1})
	\hat{\phi}(x_{2}) \cdots \hat{\phi}(x_{\l + 1})
	\indicator_{\{|x_{1}|  \le 1 + |x_{2}| + \cdots + |x_{\l + 1}|  \}}dx_{1}\cdots dx_{\l + 1}.
	\end{equation}
	Fix $x_2, \ldots, x_{\ell + 1}$ and set $S_\ell(x_1) = \sin(2\pi x_1(1 + |x_2| + \cdots + |x_{\l + 1}|))/(2\pi x_1) $. We have the identity 
	\begin{equation}
	    \indicator_{\{|x_{1}|  \le 1 + |x_{2}| + \cdots + |x_{\l + 1}|  \}} (x_1) = 2 \widehat{S_\ell}(x_1),
	\end{equation}
	which follows from the Fourier pair
	
	\begin{equation}
	    \frac{\sin(2 \pi Ax)}{2 \pi x}\ =  \ \intii \frac{1}{2} \indicator_{\{|u| \le A\}} e^{2\pi i x u} du.
	\end{equation}
	
	\no Thus Plancherel's theorem gives us that
	\begin{align}
	&\bxi{\ell}\ =\ \phi^{n}(0) \nonumber \\
	&- 2\int_{-\infty}^{\infty} \cdots \int_{-\infty}^{\infty}\hat{\phi}(x_{\l + 1}) \cdots \hat{\phi}(x_{2}) \int_{-\infty}^{\infty}\phi^{n - \l}(x_{1})\frac{\sin\left(2\pi x_{1}(1 + |x_{2}| + \cdots + |x_{\l + 1}|)\right)}{2\pi x_{1}} dx_{1} \cdots dx_{\l+1} 
	\end{align}
	as desired.
\end{proof}

Applying Lemma \ref{prop_eval_bxi} to \eqref{eq: Qn bxi} and comparing with \eqref{eq: R def} completes the proof of Proposition \ref{thm:ComputeQnPhi}.

%%%%%%%

%%%%%%%%%%%%%%%%%%%%%%%%%%%%%%%%%%%%%%%%%%%%%%%%%%%%%%%%%%%%%%%%%%%%%%%%%%%%%%%%%%%%%%%%%%%%%%%%%%%%%%%%%%%%%%%%%%%%%%%%%%%%%%%%%%%%%
%%%%%%%%%%%%%%%%%%%%%%%%%%%%%%%%%%%%%%%%%%%%%%%%%%%%%%%%%%%%%%%%%%%%%%%%%%%%%%%%%%%%%%%%%%%%%%%%%%%%%%%%%%%%%%%%%%%%%%%%%%%%%%%%%%%%%
%%%%%%%%%%%%%%%%%%%%%%%%%%%%%%%%%%%%%%%%%%%%%%%%%%%%%%%%%%%%%%%%%%%%%%%%%%%%%%%%%%%%%%%%%%%%%%%%%%%%%%%%%%%%%%%%%%%%%%%%%%%%%%%%%%%%%
%%%%%%%%%%%%%%%%%%%%%%%%%%%%%%%%%%%%%%%%%%%%%%%%%%%%%%%%%%%%%%%%%%%%%%%%%%%%%%%%%%%%%%%%%%%%%%%%%%%%%%%%%%%%%%%%%%%%%%%%%%%%%%%%%%%%%

%%%%%%%%%%%%%%%%%%%%%%%%%%%%%%%%%%%%%%%%%%%%%%%%%%%%%%%%%%%%%%%%%%%%%%%%%%%%%%%%%%%%%%%%%%%%%%%%%%%%%%%%%%%%%%%%%%%%%%%%%%%%%%%%%%%%%
%%%%%%%%%%%%%%%%%%%%%%%%%%%%%%%%%%%%%%%%%%%%%%%%%%%%%%%%%%%%%%%%%%%%%%%%%%%%%%%%%%%%%%%%%%%%%%%%%%%%%%%%%%%%%%%%%%%%%%%%%%%%%%%%%%%%%
%%%%%%%%%%%%%%%%%%%%%%%%%%%%%%%%%%%%%%%%%%%%%%%%%%%%%%%%%%%%%%%%%%%%%%%%%%%%%%%%%%%%%%%%%%%%%%%%%%%%%%%%%%%%%%%%%%%%%%%%%%%%%%%%%%%%%
%%%%%%%%%%%%%%%%%%%%%%%%%%%%%%%%%%%%%%%%%%%%%%%%%%%%%%%%%%%%%%%%%%%%%%%%%%%%%%%%%%%%%%%%%%%%%%%%%%%%%%%%%%%%%%%%%%%%%%%%%%%%%%%%%%%%%

\appendix

\section{Proof of Lemma \ref{lem:coeffs are 0}}\label{sec:pfs nt lemmas}
%%%%%%%%%%%%%%%%%%%%%%%%%%%%%%%%%%%%%%%%%%%%%%%%%%%%

\begin{proof}
    We begin by imposing a distinctness condition on the sum over primes in \eqref{eq:S2n}. Fix a partition $\vec n = (n_1, \ldots, n_\ell)$ of $n$, where $1 \le n_1 \le \cdots \le n_\ell$ and $n_1 + \cdots + n_\ell = n$. We want to write $p_1\cdots p_n = q_1^{n_1} \cdots q_\ell^{n_\ell}$. If the parts of our partition are distinguishable, then by the multinomial theorem there are
    \begin{equation}
        \binom{n}{n_1, \ldots, n_\ell} = \frac{n!}{(n_1!)\cdots (n_\ell !)}
    \end{equation}
    ways to do this. Let $r_x(\vec n)$ denote the number of values of $i$ for which $n_i = x$ in the partition $\vec n$. Because the parts of our partition which are the same size are \emph{indistinguishable}, the number of ways to write $p_1\cdots p_n = q_1^{n_1} \cdots q_\ell^{n_\ell}$ is
    \begin{equation}
         \sigma(\vec n) \defeq \binom{n}{n_1, \ldots, n_\ell}  \prod_{j=1}^n \frac{1}{r_j!}.
    \end{equation}
    Thus we have that
    \begin{align}\label{eq: partition expression}
    \frac{S_2^{(n)}(N)}{\i^k \sqrt{N}} \ &= \ \sum_{1\le \ell \le n} \ \sum_{\substack{\vec{n}:=(n_1,\dots,n_{\ell}) \\ 1 \le n_1 \le \cdots \le n_\ell \\ n_{1}+\cdots +n_{\ell}=n}} \   \sigma(\vec n) \qsum{\ell} \ \prod_{j=1}^{\ell} \left( \testfn{q_j}^{n_j}  \testcomp{q_j}^{n_j} \right) \nonumber\\
    & \hs{1} \times \left< \lambda_{f}(N) \lambda_{f}(q_{1})^{n_{1}}\cdots \lambda_{f}(q_{\ell})^{n_{\ell}}\right>_\ast.
\end{align}
Now that we have a distinct sum over primes, we want to apply the multiplicative properties of Fourier coefficients given in Lemma \ref{lem:multiplicativity of fourier coeffs}. We do some careful bookkeeping in the process. We will sum over tuples $\vec m = (m_1, \ldots, m_\ell)$ which are \emph{admissible} to a given partition $\vec n$. This means that $m_i \le n_i$, $n_i \equiv m_i \pmod{2}$, and if $n_i = n_j$ with $i< j$ then $m_i \le m_j$. This last condition means that we order the $m_i$s for each fixed value of $n_i$. 

For some fixed $\vec n, \vec m$ with $\vec m$ admissible to $\vec n$, let $s_{x, y}(\vec n, \vec m)$ be the number of values of $i$ for which $(n_i, m_i) = (x, y)$.  For each value of $x$, the number of ways to order the indices $i$ for which $n_i =x$ is 
\begin{equation}
    \frac{r_x(\vec n)!}{\prod_{j=1}^x s_{x, j}(\vec n, \vec m)!}.
\end{equation}
Define the auxiliary function
\begin{equation}
    \tau(\vec n, \vec m) \defeq \prod_{i=1}^n \frac{r_i(\vec n)!}{\prod_{j=1}^i s_{i, j}(\vec n, \vec m)!}
\end{equation}
and let $t_{n, m}$ be the coefficient of $\lambda_f(p^m)$ in the expansion of $\lambda_f(p)^n$ (see \eqref{eq:GuyLambdaExpansion}), so that
\begin{equation}
    \lambda_f(p)^n = \sum_{m=0}^n t_{n, m} \lambda_f(p^m).
\end{equation}
Note that $t_{n, n} = 1$ for all $n$. Expanding the Fourier coefficients with \eqref{eq:GuyLambdaExpansion} and using \eqref{eqn multHec}, we have that
\begin{align}\label{eq: partition ms}
    \frac{S_2^{(n)}(N)}{\i^k \sqrt{N}} \ &= \ \sum_{1\le \ell \le n} \ \sum_{\substack{\vec{n}:=(n_1,\dots,n_{\ell}) \\ 1 \le n_1 \le \cdots \le n_\ell \\ n_{1}+\cdots +n_{\ell}=n}} \ \sum_{\vec m \text{ admissible to } \vec n}  \sigma(\vec n) \tau(\vec n, \vec m) \qsum{\ell} \  \nonumber\\
    & \hs{1} \times \prod_{j=1}^{\ell}  \left( t_{n_j, m_j} \testfn{q_j}^{n_j}  \testcomp{q_j}^{n_j} \right) \left< \lambda_{f}(N q_1^{m_1} \cdots q_\ell^{m_\ell})\right>_\ast
\end{align}
We want to separate $(\vec n, \vec m)$ into a part with $n_i = m_i$ and a part with $n_i < m_i$. Let $\vec{n}^\flat = (n^\flat_1, n^\flat_2, \ldots, n^\flat_\omega)$ denote the sub partition of $\vec n$ with $n_i > m_i$, and $\vec n^\sharp = (n^\sharp_1, n^\sharp_2, \ldots, n^\sharp_{\ell - \omega})$ denote the sub partition of $\vec n$ with $n_i = m_i$. The analogous notation holds for $\vec m$. Set $\sum n_i^\flat = n'$, so that $\sum n_i^\sharp = n - n'$. We have that
\begin{align}
    \sigma(\vec n) \tau( \vec n, \vec m) &= \frac{n!}{n_1! \cdots n_\ell !} \prod_{i=1}^n \prod_{j=1}^i \frac{1}{s_{i, j} (\vec n, \vec m)!} \nn 
    &= \frac{n!}{n'! (n-n')!} \times \left[\frac{n'!}{n_1^\flat! \cdots n_\omega^\flat!} \prod_{i=1}^n \prod_{j=1}^{i-1} \frac{1}{s_{i, j} (\vec n, \vec m)!}\right] \times \left[\frac{(n-n')!}{n_1^\sharp! \cdots n_{\ell - \omega}^\sharp!} \prod_{i=1}^n  \frac{1}{s_{i, i} (\vec n, \vec m)!}\right] \nn 
    &= \binom{n}{n'} \times \left[\frac{n'!}{n_1^\flat! \cdots n_\omega^\flat!} \prod_{i=1}^{n'} \prod_{j=1}^{i} \frac{1}{s_{i, j} (\vec n^\flat, \vec m^\flat)!}\right] \times \left[\frac{(n-n')!}{n_1^\sharp! \cdots n_{\ell - \omega}^\sharp!} \prod_{i=1}^{n - n'}   \frac{1}{r_i (\vec n^\sharp)!}\right] &\nn 
    &= \binom{n}{n'} \sigma(\vec n^\flat) \tau(\vec n^\flat, \vec m^\flat) \sigma(\vec n^\sharp).
\end{align}
Applying this identity to \eqref{eq: partition ms} and using $t_{n, n} = 1$ and $n^\sharp_i = m^\sharp_i$ gives
\begin{align}\label{eq: four liner}
    \frac{S_2^{(n)}(N)}{\i^k \sqrt{N}} \ &= \ \sum_{0 \le \omega \le n} \sum_{0 \le n' \le n} \binom{n}{n'} \sum_{\substack{\vec{n}^\flat :=(n_1^\flat,\dots,n_{\omega}^\flat) \\ 1 \le n_1^\flat \le \ldots \le n_\omega^\flat \\ n_{1}^\flat+\cdots +n_{\omega}^\flat=n'}} \ \sum_{\substack{\vec m^\flat  \text{ admissible to } \vec n^\flat \\ n_j^\flat < m_j^\flat }}  \sigma(\vec n^\flat) \tau(\vec n^\flat, \vec m^\flat) t_{n_1^\flat, m_1^\flat} \cdots t_{n_\omega^\flat, m_\omega^\flat}   \nonumber\\
    & \hs{1} \times  \qsum{\omega}  \prod_{j=1}^{\omega}  \left( \testfn{q_j}^{n^\flat_j}  \testcomp{q_j}^{n^\flat_j} \right)  \nn
    &\hs{1} \times \sum_{\substack{\ell \\ 0\le \ell - \omega \le n-n'}} \sum_{\substack{\vec{n}^\sharp :=(n_1^\sharp,\dots,n_{\ell - \omega}^\sharp) \\ 1 \le n_1^\sharp \le \ldots \le n_{\ell - \omega}^\sharp \\ n_{1}^\sharp+\cdots +n_{\ell - \omega}^\sharp=n - n'}} \sigma(\vec n^\sharp) \psumextra{\ell - \omega} \prod_{i=1}^{\ell - \omega}  \left( \testfn{p_i}^{n^\sharp_i}  \testcomp{p_i}^{n^\sharp_i} \right) \nn
    & \hs{1} \times \left< \lambda_{f}(N q_1^{m^\flat_1} \cdots q_\omega^{m^\flat_\omega} p_1^{n_1^\sharp} \cdots p_{\ell - \omega}^{n_{\ell - \omega}^\sharp})\right>_\ast.
\end{align}
Arguing as in \eqref{eq: partition expression}, we have that
\begin{align}\label{eq: regroup sharps}
    & \sum_{\substack{\ell \\ 0\le \ell - \omega \le n-n'}} \sum_{\substack{\vec{n}^\sharp :=(n_1^\sharp,\dots,n_{\ell - \omega}^\sharp) \\ 1 \le n_1^\sharp \le \ldots \le n_{\ell - \omega}^\sharp \\ n_{1}^\sharp+\cdots +n_{\ell - \omega}^\sharp=n - n'}} \sigma(\vec n^\sharp) \psumextra{\ell - \omega} \prod_{i=1}^{\ell - \omega}  \left( \testfn{p_i}^{n^\sharp_i}  \testcomp{p_i}^{n^\sharp_i} \right)  \nn 
    & \times \left< \lambda_{f}(N q_1^{m^\flat_1} \cdots q_\omega^{m^\flat_\omega} p_1^{n_1^\sharp} \cdots p_{\ell - \omega}^{n_{\ell - \omega}^\sharp})\right>_\ast \nn
    &= \sum_{\substack{p_1 \nmid N, \ldots, p_{n- n'} \nmid N \\ p_i \ne q_j}} \prod_{i=1}^{n - n'}  \left( \testfn{p_i} \testcomp{p_i} \right) \left< \lambda_{f}(N q_1^{m^\flat_1} \cdots q_\omega^{m^\flat_\omega} p_1 \cdots p_{n - n'})\right>_\ast.
\end{align}
We apply \eqref{eq: regroup sharps} to \eqref{eq: four liner}. In doing so, we remove the condition that $1 \le n_1^\flat \cdots \le n^\flat_\omega$ as we are no longer concerned with the ordering of our partition. We also relax the condition that $\vec m^\flat$ is admissible to $\vec n^\flat$. We also suppress the $\flat$ notation (as there are no more $\sharp$s). Lastly, the first line of \eqref{eq: four liner} is purely combinatorial, so we combine the combinatorial factors appearing in \eqref{eq: four liner} into the coefficients $C''_{\vec n, \vec m}$ appearing below. Thus we have that
\begin{equation}\label{eq: S2n E''}
    S_2^{(n)}(N) = \sum_{0 \le \omega \le n} \sum_{0 \le n' \le n} \sum_{\substack{\vec{n} :=(n_1,\dots,n_{\omega}) \\ n_j > 1 \\ n_{1}+\cdots +n_{\omega}=n'}} \sum_{\substack{\vec{m} :=(m_1,\dots,m_{\omega}) \\ m_j \equiv n_j \pmod{2} \\ 0\le  m_j < n_j}} C''_{\vec n, \vec m} E''(\vec n, \vec m)
\end{equation}
where each $C''_{\vec n, \vec m}$ is some explicit constant only dependent on $\vec n, \vec m$ and 
\begin{align}\label{eq: E'' def}
    E''(\vec n, \vec m) &\defeq   \qsum{\omega}  \prod_{j=1}^{\omega}  \left( \testfn{q_j}^{n_j}  \testcomp{q_j}^{n_j} \right)  \nn 
    & \hs{1} \times \sum_{\substack{p_1 \nmid N, \ldots, p_{n- n'} \nmid N \\ p_i \ne q_j}} \prod_{i=1}^{n - n'}  \left( \testfn{p_i} \testcomp{p_i} \right) \left< \lambda_{f}(N q_1^{m_1} \cdots q_\omega^{m_\omega} p_1 \cdots p_{n - n'})\right>_\ast.
\end{align}
We want to remove the condition $p_i \ne q_j$ from the sum over $p_1, \ldots, p_{n-n'}$ in \eqref{eq: E'' def}. We apply an inclusion-exclusion process, subtracting off the terms where some $p_i = q_j$. We have that
\begin{equation}\label{eq: E'' inc-exc}
    E''(\vec n, \vec m) = E'(\vec n, \vec m) - \sum_{a = 1}^{n-n'} \sum_{\substack{\vec a = (a_1, \ldots, a_\omega) \\ a_j \ge 0 \\ a_1 + \cdots + a_\omega = a}} C''_{\vec n, \vec m, \vec a} E''(\vec n + \vec a, \vec m + \vec a)
\end{equation}
where each $C''_{\vec n, \vec m, \vec a}$ is some explicit constant only dependent on $\vec n, \vec m, \vec a$ and
\begin{align}\label{eq: E' def}
    E'(\vec n, \vec m) &\defeq   \qsum{\omega}  \prod_{j=1}^{\omega}  \left( \testfn{q_j}^{n_j}  \testcomp{q_j}^{n_j} \right)  \nn 
    & \hs{1} \times \psumn{n-n'} \prod_{i=1}^{n - n'}  \left( \testfn{p_i} \testcomp{p_i} \right) \left< \lambda_{f}(N q_1^{m_1} \cdots q_\omega^{m_\omega} p_1 \cdots p_{n - n'})\right>_\ast.
\end{align}
The addition of vectors is taken component-wise, so that $\vec n + \vec a = (n_1 + a_1, \ldots, n_\omega + a_\omega)$. We apply the inclusion-exclusion identity \eqref{eq: E'' inc-exc} to each term $E''(\vec n + \vec a, \vec m + \vec a)$ appearing on the right hand side of \eqref{eq: E'' inc-exc}. Repeating this process $n-n'$ times, we find that
\begin{equation}\label{eq: E'' to E'}
    E''(\vec n, \vec m) = \sum_{a = 0}^{n-n'} \sum_{\substack{\vec a = (a_1, \ldots, a_\omega) \\ a_j \ge 0 \\ a_1 + \cdots + a_\omega = a}} C'_{\vec n, \vec m, \vec a} E'(\vec n + \vec a, \vec m + \vec a)
\end{equation}
where each $C'_{\vec n, \vec m, \vec a}$ is some explicit constant only dependent on $\vec n, \vec m, \vec a$. In particular, we have that $C'_{\vec n, \vec m, \vec 0} = 1$. Applying \eqref{eq: E'' to E'} to \eqref{eq: S2n E''} gives
\begin{equation}\label{eq: S2n E'}
    S_2^{(n)}(N) = \sum_{0 \le \omega \le n} \sum_{0 \le n' \le n} \sum_{\substack{\vec{n} :=(n_1,\dots,n_{\omega}) \\ n_j > 1 \\ n_{1}+\cdots +n_{\omega}=n'}} \sum_{\substack{\vec{m} :=(m_1,\dots,m_{\omega}) \\ m_j \equiv n_j \pmod{2} \\ 0\le  m_j < n_j}} C'_{\vec n, \vec m} E'(\vec n, \vec m)
\end{equation}
where each $C'_{\vec n, \vec m}$ is some explicit constant only dependent on $\vec n, \vec m$.

We want to remove the distinctness condition from the sum over $q_1, \ldots, q_\omega$ in \eqref{eq: E' def}. We again apply inclusion-exclusion, subtracting off terms where some of the $q_j$'s are equal. First, we define a \emph{partition of a set} $S$ to be a set $\vec \pi = \{\pi_1, \ldots, \pi_\alpha\}$ where each $\pi_i \subset S$ is a nonempty subset of $S$, $\pi_i \cap \pi_j = \emptyset$ for every $i \ne j$, and $\bigcap_{1 \le i \le \alpha} \pi_i = S$. We have that
\begin{equation}\label{eq: E' inc-exc}
    E'(\vec n, \vec m) = E(\vec n, \vec m) - \sum_{1 \le \alpha \le \omega - 1}\  \sum_{\substack{\vec\pi \text{ partitions } \{1, \ldots, \omega\} \\ \vec\pi = \{\pi_1, \ldots, \pi_\alpha\} }}  E'(\vec x, \vec y)
\end{equation}
where $E(\vec n, \vec m)$ is as in \eqref{eq:E with omega} and for each $\vec \pi$ we have
\begin{equation}
    \vec x = (x_1, \ldots, x_\alpha), \quad x_i = \sum_{j \in \pi_i} n_j
\end{equation}
and 
\begin{equation}
    \vec y = (y_1, \ldots, y_\alpha), \quad y_i = \sum_{j \in \pi_i} m_j.
\end{equation}
We apply the inclusion-exclusion identity \eqref{eq: E' inc-exc} to each term $E'(\vec x, \vec y)$ appearing on the right hand side of \eqref{eq: E' inc-exc}. Repeating this process $\omega$ times, we have that
\begin{equation}\label{eq: E' to E}
    E'(\vec n, \vec m) = \sum_{1 \le \alpha \le \omega}\  \sum_{\substack{\vec\pi \text{ partitions } \{1, \ldots, \omega\} \\ \vec\pi = (\pi_1, \ldots, \pi_\alpha) }} C_{\vec \pi} E(\vec x, \vec y)
\end{equation}
where each $C_{\vec \pi}$ is some explicit constant only depending on $\vec \pi$. Applying \eqref{eq: E' to E} to \eqref{eq: S2n E'} gives the desired result.
\end{proof}

%%%%%%%%%%%%%%%%%%%%%%%%%%%%%%%%%%%%%%%%%%%%%%%%%%%%%%%%%%%%%%%%%%%%%%%%%%%%%%%%%%%%%%%%%%%%%%%%%%%%%%%%%%%%%%%%%%%%%%%%%%%%%%%%%%%%%
%%%%%%%%%%%%%%%%%%%%%%%%%%%%%%%%%%%%%%%%%%%%%%%%%%%%%%%%%%%%%%%%%%%%%%%%%%%%%%%%%%%%%%%%%%%%%%%%%%%%%%%%%%%%%%%%%%%%%%%%%%%%%%%%%%%%%
%%%%%%%%%%%%%%%%%%%%%%%%%%%%%%%%%%%%%%%%%%%%%%%%%%%%%%%%%%%%%%%%%%%%%%%%%%%%%%%%%%%%%%%%%%%%%%%%%%%%%%%%%%%%%%%%%%%%%%%%%%%%%%%%%%%%%
%%%%%%%%%%%%%%%%%%%%%%%%%%%%%%%%%%%%%%%%%%%%%%%%%%%%%%%%%%%%%%%%%%%%%%%%%%%%%%%%%%%%%%%%%%%%%%%%%%%%%%%%%%%%%%%%%%%%%%%%%%%%%%%%%%%%%

\section{Increasing support for the non-split family}\label{sec: completefamily}

In this section, we prove Theorem~\ref{thm:mock-Gaussian for D}. Arguing as in Appendix E of \cite{HM}, we need to bound terms of the form

\begin{align}\label{eq: Ecomp def}
    \overline{E} (\vec n, \vec m) \ &\defeq \ 2\pi i^k \qsumdist{\ell}  \prod_{j=1}^{\ell} \left( \testfn{q_j}^{n_j}  \testcomp{q_j}^{n_j} \right) \msum \frac{1}{m} \nn 
    &\hs{1} \times \ \sum_{b=1}^\infty \frac{S(m^2, Q; Nb)}{Nb} \Jk\left( \frac{4\pi m \sqrt{Q}}{Nb} \right)
\end{align}
where $Q = q_1^{m_1} \cdots q_\ell^{m_\ell}$ and $n_j \equiv m_j \pmod 2$ for all $j$. Showing that these terms vanish as $N \to \infty$ for $\phi$ with $\supp \hphi \subset \twoovern$ completes the proof of Theorem \ref{thm:mock-Gaussian for D}. These terms are very similar to the $\Enm$ terms introduced in Section~\ref{sec:number theory section} (see \eqref{eq:E Petersson}, for example), and we are able to evaluate them in a similar fashion. We omit proofs as they are analogous to the proofs of the corresponding lemmas in Section~\ref{sec:number theory section}, which we refer to. We will eventually prove the following lemma.
\begin{lem}\label{lem: E comp final}
    Let $\overline{E} (\vec n, \vec m)$ be defined as in \eqref{eq: Ecomp def}. Under GRH for Dirichlet $L$-functions, if $\supp(\hphi) \subset \twoovern$, then $\overline{E} (\vec n, \vec m) \ll N^{-\epsilon}$ and thus does not contribute in the limit.
\end{lem}

First we restrict the sum over $b$ as in Lemmas \ref{lem: b,N coprime} and \ref{lem: b > N2006}. 

\begin{lem} \label{lem: b,N coprime comp}
    Suppose supp$(\hat{\phi}) \subseteq \left( -\frac{7}{2n}, \frac{7}{2n}\right)$. Then the subterms of $\overline{E} (\vec n, \vec m)$ in \eqref{eq: Ecomp def} for which $(b,N)>1$ are $\ONe$.
\end{lem}

\begin{lem} \label{lem: b > N2006 comp}
    Suppose supp$(\hat{\phi}) \subset \left( -\frac{1000}{n}, \frac{1000}{n}\right)$. Then the subterms of $\overline{E} (\vec n, \vec m)$ in \eqref{eq: Ecomp def} for which $b \geq N^{2022}$ are $O(N^{-12})$.
\end{lem}
Applying Lemmas \ref{lem: b,N coprime comp} and \ref{lem: b > N2006 comp} to \eqref{eq: Ecomp def} gives

\begin{align}
    \overline{E} (\vec n, \vec m) \ &= \ 2\pi i^k \qsumdist{\ell}  \prod_{j=1}^{\ell} \left( \testfn{q_j}^{n_j}  \testcomp{q_j}^{n_j} \right) \msum \frac{1}{m} \nonumber \\
    & \hs1 \times \bsum{N} \frac{S(m^2, Q; Nb)}{Nb} \Jk\left( \frac{4\pi m\sqrt{Q}}{Nb} \right) + \ONe. \label{eq: Ecomp bsum}
\end{align}
We convert the Kloosterman sums to sums over Gauss sums as in Lemma~\ref{lem:genkloos2}. 

\begin{lem}\label{lem:genkloos comp}
    Let $N$ be a prime not dividing $b, Q, m$. Then
    \begin{equation}
        S(m^2, NQ; Nb) \ = \ -\frac{1}{\varphi(Nb)}\sum_{\chi (Nb)} G_\chi (m^2) G_\chi((Q, b^\infty)) \overline{\chi} \left(\frac{Q}{(Q, b^\infty)}\right) .
    \end{equation}
\end{lem}
Applying Lemma \ref{lem:genkloos comp} to \eqref{eq: Ecomp bsum} gives

\begin{align}
    \overline{E} (\vec n, \vec m) \ &= \ -2\pi i^k \qsumdist{\ell}  \prod_{j=1}^{\ell} \left( \testfn{q_j}^{n_j}  \testcomp{q_j}^{n_j} \right) \msum \frac{1}{m}\nonumber \\
    &\hs{1} \times \bsum{N} \frac{1}{Nb \varphi(Nb)} \sum_{\chi (Nb)} G_\chi (m^2) G_\chi((Q, b^\infty)) \overline{\chi} \left(\frac{Q}{(Q, b^\infty)}\right)  \Jk\left( \frac{4\pi m\sqrt{Q}}{Nb} \right) +\ONe. \label{eq:Ecomp Gauss}
\end{align}

Next, it holds that subterms involving non-principal characters in \eqref{eq:Ecomp Gauss} are negligible in the limit. This leaves only subterms involving $\chi_0 = \overline{\chi_0} \pmod{Nb}$ for each $b$. It holds that $G_{\chi_0}(x) = R(x, Nb)$, a Ramanujan sum.

\begin{lem}\label{lem:char comp}
    Assume GRH for Dirichlet $L$-functions and suppose that $\supp(\hphi) \subset \twoovern$. Then the sum over all non-principal characters in \eqref{eq:Ecomp Gauss} is $\ONe$.
\end{lem}
This lemma corresponds to Lemma~\ref{lem:char}. Applying Lemma \ref{lem:char comp} to \eqref{eq:Ecomp Gauss} gives
\begin{align}
    \overline{E} (\vec n, \vec m) \ &= \ -2\pi i^k \qsumdist{\ell}  \prod_{j=1}^{\ell} \left( \testfn{q_j}^{n_j}  \testcomp{q_j}^{n_j} \right) \msum \frac{1}{m}\nonumber \\
    &\hs{1} \times \bsum{N} \frac{R(m^2, Nb) R((Q, b^\infty), Nb)}{Nb \varphi(Nb)} \chi_0 \left(\frac{Q}{(Q, b^\infty)}\right)  \Jk\left( \frac{4\pi m \sqrt{Q}}{Nb} \right) +\ONe. \label{eq:Ecomp Ram}
\end{align}
Now, applying the bounds $R(m^2, Nb) \le m^4$, $R(x, Nb) \le \varphi(Nb)$, and $\Jk(x) \ll x$ to \eqref{eq:Ecomp Ram} and using the fact that $\supp \hphi \subset \twoovern$, we find that the main term is absorbed by the error term, completing the proof of Lemma \ref{lem: E comp final}.

%%%%%%%%%%%%%%%%%%%%%%%%%%%%%%%%%%%%%%%%%%%%%%%%%%%%%%%%%%%%%%%%%%%%%%%%%%%%%%%%%%%%%%%%%%%%%%%%%%%%%%%%%%%%%%%%%%%%%%%%%%%%%%%%%%%%%
%%%%%%%%%%%%%%%%%%%%%%%%%%%%%%%%%%%%%%%%%%%%%%%%%%%%%%%%%%%%%%%%%%%%%%%%%%%%%%%%%%%%%%%%%%%%%%%%%%%%%%%%%%%%%%%%%%%%%%%%%%%%%%%%%%%%%
%%%%%%%%%%%%%%%%%%%%%%%%%%%%%%%%%%%%%%%%%%%%%%%%%%%%%%%%%%%%%%%%%%%%%%%%%%%%%%%%%%%%%%%%%%%%%%%%%%%%%%%%%%%%%%%%%%%%%%%%%%%%%%%%%%%%%
%%%%%%%%%%%%%%%%%%%%%%%%%%%%%%%%%%%%%%%%%%%%%%%%%%%%%%%%%%%%%%%%%%%%%%%%%%%%%%%%%%%%%%%%%%%%%%%%%%%%%%%%%%%%%%%%%%%%%%%%%%%%%%%%%%%%%

%%%%%%%%%%%%%%%%%%%%%%%%%%%%%%%%%%%%%%%%%%%%%%%%%%%%%

\section{Proofs of Lemmas in Section \ref{sec:rmt calc}}\label{sec:RMT Lemmas}

\subsection{Proof of Lemma \ref{lem_single_simp}}\label{app_single_simp}

\begin{proof}
We will consider each term appearing in \eqref{eq:single_combin} separately. First, define
\begin{alignat}{2}
		g_1(n,f,c,d) & \coloneqq\binom{n}{f}, \qquad && g_2(n,f,c,d) \coloneqq\binom{n-c}{f-c} \nonumber \\
		g_3(n,f,c,d) & \coloneqq\binom{n-d}{f-d}, \qquad && g_4(n,f,c,d) \coloneqq\binom{n-c-d}{f-c-d} 
	\end{alignat}
	and
	\begin{equation}\label{eqn_sepasing}
		G_i(n,f) \coloneqq2(-1)^n n!\sum_{\substack{c+d\le n \\ c,d \ge 0}} (-1)^{c+d+1}\frac{ g_i(n,f,c,d)}{(n-c-d)!c!d!}. 
	\end{equation}
 We want to evaluate $G_1(n,f) - G_2(n,f) - G_3(n,f) + G_4(n,f)$. We set $\ell=c+d$ to rewrite \eqref{eqn_sepasing} as 
\begin{equation}
    G_i(n,f)=2(-1)^{n+1} n!\sum_{\ell=0}^n (-1)^{\ell} \sum_{c=0}^{\ell} \frac{g_i(n,f,c,\ell-c)}{(n-\ell)!(\ell-c)!c!}.
\end{equation}
To evaluate $G_1(n, f)$, we group the binomial coefficients to find that
\begin{align}
G_1(n, f) \ &= \ 2(-1)^{n+1} \binom{n}{f} \sum_{\ell = 0}^n (-1)^\ell \binom{n}{\ell} \sum_{c = 0}^\ell \binom{\ell}{c}\nonumber \\
& = \ 2(-1)^{n+1} \binom{n}{f} \sum_{\ell = 0}^n (-2)^\ell \binom{n}{\ell} = -2 \binom{n}{f}. \label{eq: G1 eval}
\end{align}

Next, we note that $G_2(n,f) = G_3(n,f)$. We have that
\begin{align}
G_2(n, f) \ &= \ 2(-1)^{n+1} \binom{n}{f} \sum_{\ell = 0}^n (-1)^\ell  \sum_{c = 0}^\ell \binom{f}{c} \binom{n - c}{n - \ell} \nonumber \\
&= \ 2(-1)^{n+1} \binom{n}{f} \sum_{c = 0}^n \binom{f }{c} \sum_{\ell = c}^n (-1)^\ell \binom{n - c}{n - \ell}. 
\end{align}
We reindex the sum by setting $\ell' = \ell - c$. Doing so, we see that sum over $\ell '$ is zero unless $n - c = 0$. However, in this case we have that $\binom{f}{c} = 0$ since $f \le n/2 < n$. Thus each term vanishes and 
\begin{equation}\label{eq: G2 eval}
    G_2(n,f) \ = \ G_3(n, f) \ = \ 0.
\end{equation}

Lastly, again grouping terms into binomial coefficients gives
\begin{align}
    G_4(n,f) \ &= \ 2 (-1)^{n+1} \binom{n}{f} \sum_{\ell = 0}^n (-1)^\ell \binom{f}{\ell} \sum_{c = 0 }^\ell  \binom{\ell}{c}\nonumber  \\
    &= \ 2 (-1)^{n+1} \binom{n}{f} \sum_{\ell = 0}^n (-2)^\ell \binom{f}{\ell}.
\end{align}
We may restrict the sum in the last line to $0 \le \ell \le f$ since $f \le n$ and $\binom{f}{\ell} = 0$ when $\ell > f$. Doing so, we find the sum over $\ell$ is $(-1)^f$ so
\begin{equation}\label{eq: G4 eval}
    G_4(n,f) \ = \ 2(-1)^{n+ f + 1} \binom{n}{f}.
\end{equation}
Combining \eqref{eq: G1 eval}, \eqref{eq: G2 eval} and \eqref{eq: G4 eval} completes the proof of the lemma.
\end{proof}

\subsection{Proof of Lemma \ref{lem:prod_disj_vanishes}} \label{app_disj_vanishes}
\begin{proof}
We consider each term appearing in \eqref{eq: Hhf def} separately. First, define
\begin{alignat}{2}
		h_1(f, g, \mu_1, \mu_d) \ &\coloneqq \ \binom{f}{g}, \qquad && h_2(f, g, \mu_1, \mu_d) \ \coloneqq \ \binom{f - \mu_1}{g - \mu_1} \nonumber \\
		h_3(f, g, \mu_1, \mu_d) \ &\coloneqq \ \binom{f -\mu_d}{g}, \qquad && h_2(f, g, \mu_1, \mu_d) \ \coloneqq \ \binom{f-\mu_1- \mu_d}{g - \mu_1} 
	\end{alignat}
	and
	\begin{equation}
		H_i(f, g) \ \coloneqq \ \sum_{d = 1}^{f}  \sum_{\substack{\mu_1+\dots+\mu_{d} = f \\ \mu_i \ge 1}} \frac{(-1)^{d}}{\mu_1! \cdots \mu_{d}!}  h_i(f, g, \mu_1, \mu_{d})
	\end{equation}
for $i \in \{1, 2, 3, 4\}$. We will show that $H_i(f,g) = (-1)^f/g!(f-g)!$ independent of $i$, so that $H_1 -H_2 - H_3 + H_4 = 0$ as desired. For $H_1$ the result follows immediately from comparing coefficients of $z^f$ in the identity \eqref{eq: gen function ez}. For $H_2$, we pull out the $\mu_1$ term to get
\begin{equation}
\begin{split}
H_2(f, g)\ &= \ - \sum_{\mu_1 = 1}^f \frac{1}{\mu_1 !} \binom{f - \mu_1}{g - \mu_1} \sum_{d=1}^{f- \mu_1}  \sum_{\substack{\mu_2 + \cdots +  \mu_d = f - \mu_1 \\ \mu_i \ge 1}} \frac{(-1)^{d-1}}{\mu_2! \cdots \mu_{d}!}. \\
\end{split}
\end{equation}
Applying \eqref{eq: gen function ez} and simplifying gives
\begin{align}
H_2(f, g) \ &= \ - \sum_{\mu_1 = 1}^f \frac{1}{\mu_1 !} \binom{f - \mu_1}{g - \mu_1} \frac{(-1)^{f - \mu_1}}{(f-\mu_1)!}\nonumber\\
&= \ - \frac{(-1)^f}{(f-g)!} \sum_{\mu_1 = 1}^f \frac{(-1)^{\mu_1}}{\mu_1 ! (g-\mu_1)!} \ = \  \frac{(-1)^f}{g!(f-g)!},
\end{align}
where the last step comes from restricting the summation to $1\le \mu_1 \le g$ and using the binomial expansion of $(1-1)^g$. We can show the result for $H_3$ similarly. For $H_4$, we pull out the $\mu_1$ and $\mu_d$ terms to get
\begin{equation}
\begin{split}
H_4(f, g)\ &= \  \sum_{\mu_d = 1}^f \sum_{\mu_1 = 1}^{f- \mu_d} \frac{1}{\mu_1 ! \mu_d !} \binom{f - \mu_1 - \mu_d}{g - \mu_1} \sum_{d=1}^{f- \mu_1 - \mu_d}  \sum_{\substack{\mu_2 + \cdots +  \mu_{d-1} = f - \mu_1 - \mu_d \\ \mu_i \ge 1}} \frac{(-1)^{d-2}}{\mu_2! \cdots \mu_{d-1}!}. \\
\end{split}
\end{equation}
Applying \eqref{eq: gen function ez} and simplifying gives
\begin{align}
H_4(f, g) \ &= \ \sum_{\mu_1 = 1}^f \sum_{\mu_d = 1}^{f- \mu_1} \frac{1}{\mu_1 ! \mu_d !} \binom{f - \mu_1 - \mu_d}{g - \mu_1} \frac{(-1)^ {f- \mu_1 - \mu_d}}{(f - \mu_1 - \mu_d)!}\nonumber\\
&= \ (-1)^f \sum_{\mu_d = 1}^f \frac{(-1)^{\mu_d} }{(f-g - \mu_d)! \mu_d!} \sum_{\mu_1 = 1}^{f - \mu_d} \frac{(-1)^{\mu_1}}{\mu_1! (g-\mu_1)!}\nonumber \\
&= \ \frac{(-1)^{f+1}}{g!} \sum_{\mu_d = 1}^f \frac{(-1)^{\mu_d} }{(f-g - \mu_d)! \mu_d!} \ = \ \frac{(-1)^f}{g!(f-g)!}
\end{align}
where the last two steps come from restricting the summation to $1\le \mu_1 \le g$ and $1 \le \mu_d \le f-g$ and using the binomial expansion of $(1-1)^g$ and $(1-1)^{f-g}$.
\end{proof}

%%%%%%%%%%%%%%%%%%%%%%%%%%%%%%%%%%%%%%%%%%%%%%%%%%%%%%%%%%%%%%%%%%%%%%%%%%%%%%%%%%%%%%%%%%%%%%%%%%%%%%%%%%%%%%%%%%%%%%%%%%%%%%%%%%%%%
%%%%%%%%%%%%%%%%%%%%%%%%%%%%%%%%%%%%%%%%%%%%%%%%%%%%%%%%%%%%%%%%%%%%%%%%%%%%%%%%%%%%%%%%%%%%%%%%%%%%%%%%%%%%%%%%%%%%%%%%%%%%%%%%%%%%%
%%%%%%%%%%%%%%%%%%%%%%%%%%%%%%%%%%%%%%%%%%%%%%%%%%%%%%%%%%%%%%%%%%%%%%%%%%%%%%%%%%%%%%%%%%%%%%%%%%%%%%%%%%%%%%%%%%%%%%%%%%%%%%%%%%%%%
%%%%%%%%%%%%%%%%%%%%%%%%%%%%%%%%%%%%%%%%%%%%%%%%%%%%%%%%%%%%%%%%%%%%%%%%%%%%%%%%%%%%%%%%%%%%%%%%%%%%%%%%%%%%%%%%%%%%%%%%%%%%%%%%%%%%%

%%%%%%%%%%%%%%%%%%%%%%%%%%%%%%%%%%%%%%%%%%%%%%%%%%%%%%%%%%%%%%%%%%%%%%%%%%%%%%%%%%%%%%%%%%%%%%%%%%%%%%%%%%%%%%%%%%%%%%%%%%%%%%%%%%%%%
%%%%%%%%%%%%%%%%%%%%%%%%%%%%%%%%%%%%%%%%%%%%%%%%%%%%%%%%%%%%%%%%%%%%%%%%%%%%%%%%%%%%%%%%%%%%%%%%%%%%%%%%%%%%%%%%%%%%%%%%%%%%%%%%%%%%%
%%%%%%%%%%%%%%%%%%%%%%%%%%%%%%%%%%%%%%%%%%%%%%%%%%%%%%%%%%%%%%%%%%%%%%%%%%%%%%%%%%%%%%%%%%%%%%%%%%%%%%%%%%%%%%%%%%%%%%%%%%%%%%%%%%%%%
%%%%%%%%%%%%%%%%%%%%%%%%%%%%%%%%%%%%%%%%%%%%%%%%%%%%%%%%%%%%%%%%%%%%%%%%%%%%%%%%%%%%%%%%%%%%%%%%%%%%%%%%%%%%%%%%%%%%%%%%%%%%%%%%%%%%%

\section{Bounding the order of vanishing at the central point}\label{sec: order of vanishing}

In this section, we follow the arguments of Section 6 of \cite{HM} in order to bound the proportion of newforms with negative sign whose order of vanishing exceeds a certain threshold $r$. While they are conditional on GRH, our results surpass the best known conditional and unconditional bounds established in \cite{ILS}, \cite{HM}, and \cite{Boldy} when $r \ge 5$. We focus on the case $r = 5$, however our results may be easily generalized the case when $r  > 5$. Additionally, we study the $4^{\rm{th}}$-centered moment as it provides the best bounds for the case $r = 5$, but utilizing higher moments provides better bounds as $r$ increases. Lastly, similar results may be obtained for the positive sign family. See \cite{Dutta} for a more in-depth analysis, where the results of this paper are used to find excellent bounds for vanishing to order $r$ or more; specifically, for a fixed test function Dutta and Miller determine what level density gives the best bound.

We utilize Theorem \ref{thm:lfnsn-2} with $n = 4$ and
\begin{equation}
    \phi(x) \ = \ \left(\frac{\sin \pi \sigma x}{\pi \sigma x} \right)^2, \quad \hphi(y) \ = \ \begin{cases} \frac{1}{\sigma} - \frac{|y|}{\sigma^2} & |y| < \sigma \\ 0 & |y| \ge \sigma. \end{cases}
\end{equation}
This test function is likely not optimal in general for minimizing the $\nth$-centered moment, and optimal test functions for the case $n=1$ and $n=2$ are found in \cite{ILS} and \cite{Boldy}. However, they are sufficient to surpass the bounds established in those papers. While Theorem \ref{thm:lfnsn-2} requires $\sigma < 0.5$ when $n = 4$, we may utilize the bounds given by $\sigma = 0.5$ by setting $\sigma = 0.5 - \epsilon$ and letting $\epsilon \to 0$. Now, Theorem \ref{thm:lfnsn-2} gives
\begin{equation}\label{eq: D moment bound}
\begin{split}
    \lim_{\substack{N\to\infty \\ N \text{\rm prime}}} \<\left(D(f;\phi) - \< D(f;\phi) \>_- \right)^{4}\>_- \ &= \ 3 (\sigma^2_\phi)^2 - R(4, 2; \phi) \ = \ \frac{31}{105}.
\end{split}
\end{equation}
Now, if a newform $f$ with negative sign has order of vanishing $r \ge 5$ at the central point, then by Theorem \ref{thm:mock-Gaussian for D},
\begin{align}
    D(f;\phi) - \< D(f;\phi) \>_- \ &\ge \ r \phi(0) - \left(\hphi(0) +\frac{1}{2}\phi(0) \right) \ = \ r - \frac{5}{2}\ \ge\ \frac{5}{2}. \label{eq: D summand bound}
\end{align}
Let $\Pr(r \ge 5)$ be the proportion of newforms with negative sign whose order of vanishing at the central point is at least 5. Then \eqref{eq: D moment bound} and \eqref{eq: D summand bound} give
\begin{equation}
    \Pr(r \ge 5) \left(\frac{5}{2}\right)^4 \ \le\  \frac{31}{105}
\end{equation}
so $\Pr(r \ge 5) \le \left(\frac{2}{5}\right)^4 \frac{31}{105} = \frac{496}{65625} \approx 0.00756$. \cite{ILS} and \cite{HM} obtain upper bounds of $\frac{1}{32} = 0.03125$ and $\frac{1}{49} \approx 0.02040$, respectively, our results surpass both of these. As the order of vanishing increases, our results are even better. For instance, taking $r = 19$ and $n = 20$, we find the proportion of newforms with negative sign whose order of vanishing exceeds 19 is at most $2.86 \cdot 10^{-15}$, improving the upper bound $5.77\cdot 10^{-6}$ given in \cite{Boldy} and the upper bound $3.29 \cdot 10^{-3}$ implicit in \cite{ILS}.

%%%%%%%%%%%%%%%%%%%%%%%%%%%%%%%%%%%%%%%%%%%%%%%%%%%%%%%%%%%%%%%%%%%%%%%%%%%%%%%%%%%%%%%%%%%%%%%%%%%%%%%%%%%%%%%%%%%%%%%%%%%%%%%%%%%%%
%%%%%%%%%%%%%%%%%%%%%%%%%%%%%%%%%%%%%%%%%%%%%%%%%%%%%%%%%%%%%%%%%%%%%%%%%%%%%%%%%%%%%%%%%%%%%%%%%%%%%%%%%%%%%%%%%%%%%%%%%%%%%%%%%%%%%
%%%%%%%%%%%%%%%%%%%%%%%%%%%%%%%%%%%%%%%%%%%%%%%%%%%%%%%%%%%%%%%%%%%%%%%%%%%%%%%%%%%%%%%%%%%%%%%%%%%%%%%%%%%%%%%%%%%%%%%%%%%%%%%%%%%%%
%%%%%%%%%%%%%%%%%%%%%%%%%%%%%%%%%%%%%%%%%%%%%%%%%%%%%%%%%%%%%%%%%%%%%%%%%%%%%%%%%%%%%%%%%%%%%%%%%%%%%%%%%%%%%%%%%%%%%%%%%%%%%%%%%%%%%

%%%%%%%%%%%%%%%%%%%%%%%%%%%%%%%%%%%%%%%%%%%%%%%%%%%%%
%%%%% Begin Bibliography
%%%%%%%%%%%%%%%%%%%%%%%%%%%%%%%%%%%%%%%%%%%%%%%%%%%%%

\bibliographystyle{alpha}
\newpage
\bibliography{00.final_revision}

\end{document}